\documentclass{article}
\usepackage{geometry}
\usepackage{amsmath, amsthm, amsfonts, amssymb, bbm, enumitem}
\usepackage{chngcntr}
\usepackage{tikz}
\usepackage[hidelinks]{hyperref}
\usetikzlibrary{backgrounds}
\graphicspath{ {./images/} }

\numberwithin{equation}{section}

\newtheorem{theorem}{Theorem}[section]
\newtheorem{lemma}[theorem]{Lemma}
\newtheorem{proposition}[theorem]{Proposition}
\newtheorem{corollary}[theorem]{Corollary}
\theoremstyle{definition}
\newtheorem{definition}[theorem]{Definition}
\theoremstyle{remark}
\newtheorem{remark}{Remark}[theorem]

\newcommand{\overbar}[1]{\mkern 1.5mu\overline{\mkern-1.5mu#1\mkern-1.5mu}\mkern 1.5mu}
\newcommand{\card}{\operatorname{card}}

\newcommand{\dist}{\operatorname{dist}}
\newcommand{\diam}{\operatorname{diam}}

\newcommand{\cpct}{\operatorname{Cap}}

\title{\textbf{Homogenization of the random Neumann sieve problem under minimal assumptions on the size of the perforations}}
\author{\textbf{Mert Baştuğ}}
\date{}

\begin{document}

\maketitle

\begin{abstract}
    We study the limit behavior of the solutions to the Neumann sieve problem for the Poisson equation when the sieve-holes are randomly distributed according to a stationary marked point process. We determine the optimal stochastic integrability for the random radii of the perforations for which stochastic homogenization takes place despite the presence of clustering holes.
\end{abstract}


\section{Introduction}

In this article, we study the limit behavior of the solutions to the sequence of boundary value problems
\begin{equation} \label{eq:poisson_sieve}
    \begin{cases}
        \begin{aligned}
            -\Delta u_\varepsilon &= f \quad \text{in } U_\varepsilon, \\
            u_\varepsilon &= 0 \quad \text{on } \partial U, \\
            \nabla u_\varepsilon \cdot \nu &= 0 \quad \text{on both sides of } U^0 \setminus T_\varepsilon.
        \end{aligned}
    \end{cases}
\end{equation}
To define the domain $U_\varepsilon$ for $\varepsilon > 0$, we let $U$ be an open, bounded set with Lipschitz boundary in $\mathbb{R}^N$, where $N > 2$, and we set $U^+ := \{x \in U : x_N > 0\}$, $U^- := \{x \in U : x_N < 0\}$. Then $U_\varepsilon$ is obtained by connecting the upper set $U^+$ and the lower set $U^-$, whose boundaries intersect along a hyperplane, via a set of $(N - 1)$-dimensional circular holes. More specifically, we let $T_\varepsilon$ be the union of $(N - 1)$-dimensional balls in $U^0 := \{x \in U : x_N = 0\}$ with random centers and radii. Then we set $U_\varepsilon := U^+ \cup U^- \cup T_\varepsilon$ (see Figure \ref{fig:domain}). The parameter $\varepsilon$ is proportional to the average distance between the centers of the balls. The boundary of $U_\varepsilon$ consists of $\partial U$ and the ``sieve" $U^0 \setminus T_\varepsilon$. In \eqref{eq:poisson_sieve}, we solve the Poisson equation with homogeneous Dirichlet boundary condition prescribed on $\partial U$ and homogeneous Neumann boundary condition prescribed on $U^0 \setminus T_\varepsilon$. The unit vector $\nu$ is normal to the sieve and takes either the value $e_N$ or $-e_N$ depending on the side of the sieve, where $e_N = (0, \dots, 0, 1)$. The function $f$ belongs to $L^2(U)$.

\begin{figure}[t] \label{fig:domain}
    \centering
    \begin{tikzpicture}
        \node[anchor=south west,inner sep=0] (image) at (0,0) {\includegraphics[scale=0.27]{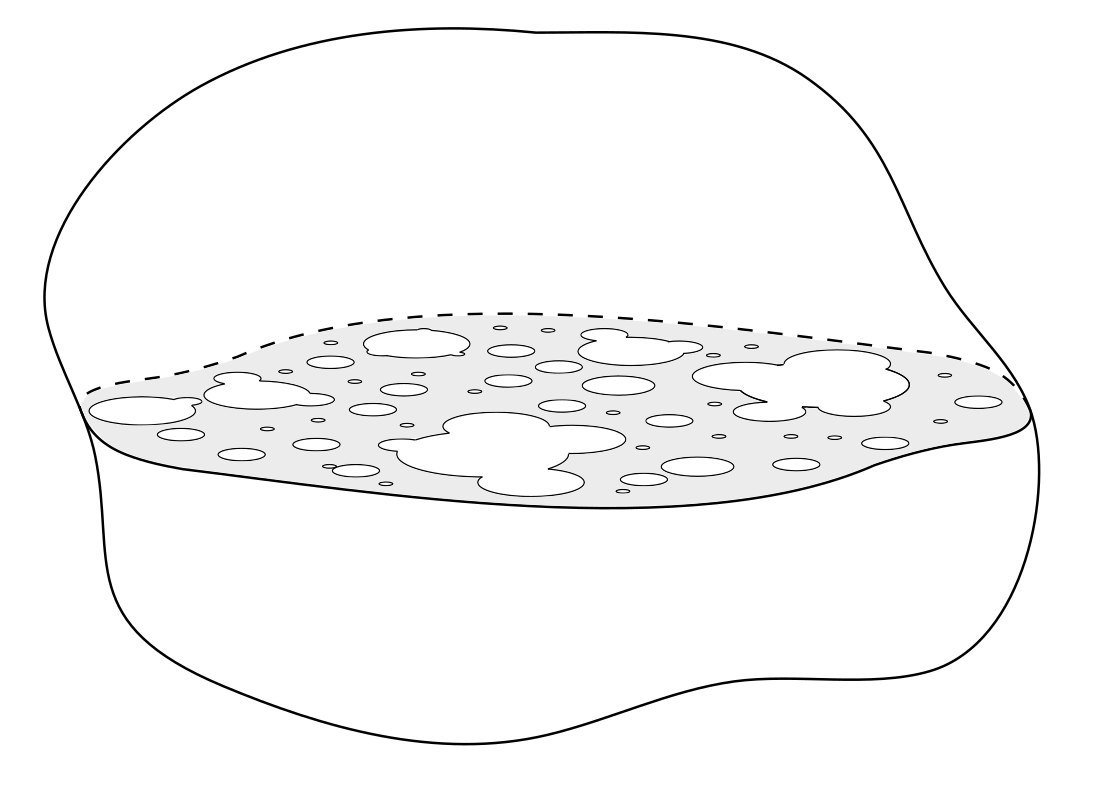}};
            \begin{scope}[x={(image.south east)},y={(image.north west)}]
                \draw (0, 0.8) node {$U_\varepsilon$};
            \end{scope}
    \end{tikzpicture}
    \quad
    \begin{tikzpicture}
        \node[anchor=south west,inner sep=0] (image) at (0,0) {\includegraphics[scale=0.27]{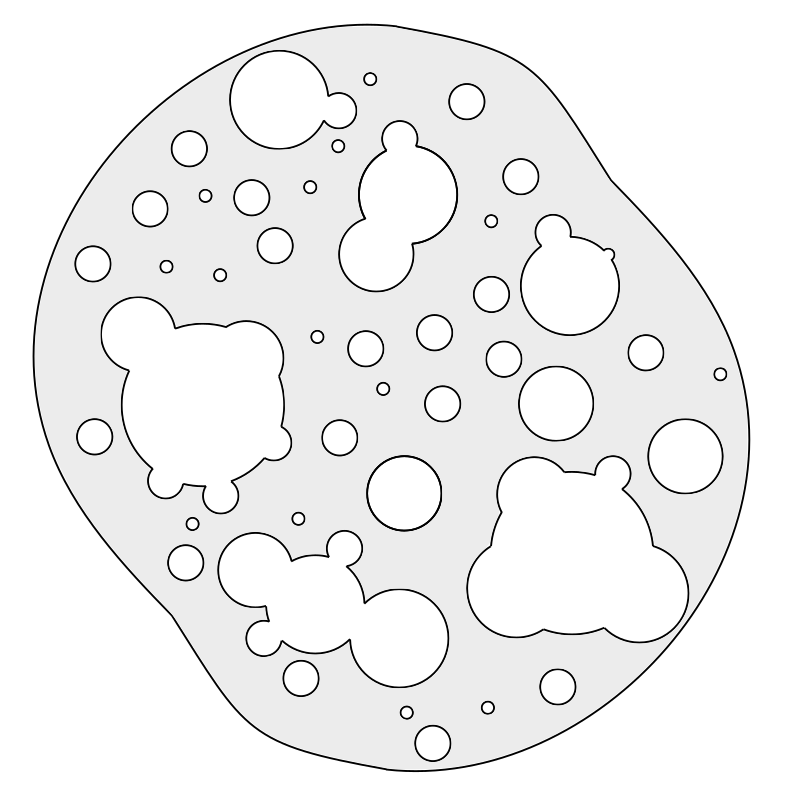}};
            \begin{scope}[x={(image.south east)},y={(image.north west)}]
                \draw (1, 0.8) node {$U^0 \setminus T_\varepsilon$};
            \end{scope}
    \end{tikzpicture}
    \caption{Illustrations of the domain $U_\varepsilon$ and the ``sieve" $U^0 \setminus T_\varepsilon$.}
\end{figure}

The interest in this kind of problem comes from hydrodynamics, namely from the study of fluid flow through walls perforated with holes \cite{C1, C2, SHSP}, where the goal is to understand the effective behavior of the fluid as the holes shrink while their number increases. The simplified model \eqref{eq:poisson_sieve} with the Poisson equation, also called Neumann's sieve, was proposed by Sánchez-Palencia in \cite{SP}. In this case, the solution $u_\varepsilon$ resembles the pressure of an incompressible, inviscid fluid where the fluid is subject to tangential motion on the sieve $U^0 \setminus T_\varepsilon$.

In this article, the perforations $T_\varepsilon$ are randomly generated by a marked point process. The goal of the paper is to prove the homogenization of Neumann's sieve problem with minimal assumptions on the marked point process. We now give a more precise description of the perforations $T_\varepsilon$. We define $\Sigma := \{x \in \mathbb{R}^N : x_N = 0\}$, and consider a set of points
\begin{equation} \label{eq:location_and_size}
    P = \{(y_i, \rho_i)\}_{i \in \mathbb{N}}
\end{equation}
such that $y_i \in \Sigma$ and $\rho_i > 0$ for all $i \in \mathbb{N}$. We assume that $(y_i)$ is a sequence without accumulation points. We consider perforations of the form
\[
\bigcup_{y_i \in \frac{1}{\varepsilon} U^0} B\left(\varepsilon y_i, \varepsilon^\frac{N - 1}{N - 2} \rho_i\right) \cap U^0.
\]
In other words, the centers of the holes correspond to the elements of the sequence $(\varepsilon y_i)$ that lie in $U^0$, while the radii are proportional to $\varepsilon^\frac{N - 1}{N - 2}$, with proportionality constants given by the sequence $(\rho_i)$. The choice of the factor $\varepsilon^\frac{N - 1}{N - 2}$ will be explained below. To generate the set $P$ randomly, we use the framework of marked point processes, as in the recent work of Giunti, Höfer and Velázquez \cite{GHV}. Let $(\Omega, \mathcal{F}, \mathbb{P})$ be a probability space. We consider a random variable $M$, defined on $\Omega$, such that its realization $M(\omega)$ is a set of the form \eqref{eq:location_and_size} satisfying the assumptions stated above. The random variable $M$ is called a marked point process on $\Sigma$ with marks in $(0, \infty)$. Then we define
\begin{equation} \label{eq:random_perforations}
    T_\varepsilon(\omega) := \bigcup \limits_{\substack{(y, \rho) \in M(\omega) \\ y \in \frac{1}{\varepsilon} U^0}} B\left(\varepsilon y, \varepsilon^\frac{N - 1}{N - 2} \rho\right) \cap U^0.
\end{equation}
The main result of this paper is that, under very mild assumptions on the marked point process $M$, the solutions to \eqref{eq:poisson_sieve} converge weakly in $H^1(U^+)$ and $H^1(U^-)$ almost surely to the solutions of the coupled boundary value problems
\begin{equation} \label{eq:limit_equations}
    \begin{cases}
        \begin{aligned}
            -\Delta u^+ &= f \quad \text{in } U^+, \\
            u^+ &= 0 \quad \text{on } (\partial U)^+, \\
            \nabla u^+ \cdot \nu &= -\gamma(u^+ - u^-) \quad \text{on } U^0,
        \end{aligned}
    \end{cases}
    \quad
    \begin{cases}
        \begin{aligned}
            -\Delta u^- &= f \quad \text{in } U^-, \\
            u^- &= 0 \quad \text{on } (\partial U)^-, \\
            \nabla u^- \cdot \nu &= \gamma(u^+ - u^-) \quad \text{on } U^0,
        \end{aligned}
    \end{cases}    
\end{equation}
where the sets $(\partial U)^\pm$ are defined in the same way as $U^\pm$, and $\gamma$ is a positive quantity that can be computed explicitly (cf. Equation \eqref{eq:random_gamma}). Since $\gamma$ depends on the realization of the perforations $T_\varepsilon$, it is, in general, a random variable. From a physical perspective, the homogenized equations show that the effective fluid flux across the surface $U^0$ depends \textit{uniformly} on the pressure difference $u^+ - u^-$ throughout $U^0$.

To motivate our assumptions on the marked point process, we briefly review known results for Neumann’s sieve in the periodic setting, where homogenization was established by Attouch, Damlamian, Murat, and Picard \cite{D,M,P}. More specifically, in \eqref{eq:poisson_sieve} the authors considered perforations of the form
\[
\bigcup \limits_{y \in \mathbb{Z}^N \cap \frac{1}{\varepsilon} U^0} B\left(\varepsilon y, \varepsilon^\frac{N - 1}{N - 2} \rho\right) \cap U^0,
\]
where $\rho > 0$ is fixed. In this case, they proved that the homogenized problem is given by \eqref{eq:limit_equations} with
\begin{equation} \label{eq:homogenization_factor}
    \gamma = \frac{\rho^{N - 2}}{4} \cpct(B(0, 1) \cap \Sigma, \mathbb{R}^N),
\end{equation}
where $\cpct(B(0, 1) \cap \Sigma, \mathbb{R}^N)$ denotes the harmonic capacity of $B(0, 1) \cap \Sigma$ in $\mathbb{R}^N$. We recall that the harmonic capacity of a set $A \subset \Sigma$ with respect to an open set $B \supset \bar{A}$ is defined by
\[
\cpct (A, B) := \inf \left\{\int_B |\nabla v|^2 \, dx : v \in H_0^1(B), \, v = 1 \text{ in $A$}\right\}.
\]
It is important to note that $\gamma$ is proportional to the asymptotic capacity density of the perforations: If $O \subset U^0$ is regular enough, then
\begin{multline*}
    \frac{1}{\mathcal{H}^{N - 1}(O)} \sum_{y \in \mathbb{Z}^N \cap \frac{1}{\varepsilon} O} \cpct\left(B\left(\varepsilon y, \varepsilon^\frac{N - 1}{N - 2} \rho\right) \cap \Sigma, \mathbb{R}^N\right) \\
    = \frac{1}{\mathcal{H}^{N - 1}(O)} \sum_{y \in \mathbb{Z}^N \cap \frac{1}{\varepsilon} O} \varepsilon^{N - 1} \rho^{N - 2} \cpct(B(0, 1) \cap \Sigma, \mathbb{R}^N) \to \rho^{N - 2} \cpct(B(0, 1) \cap \Sigma, \mathbb{R}^N)
\end{multline*}
as $\varepsilon \to 0$, where $\mathcal{H}^{N - 1}$ is the $(N - 1)$-dimensional Hausdorff measure. Going back to the stochastic setting, this suggests that we should look for conditions on $M$ such that the limit of the ``spatial" averages
\[
\lim_{\varepsilon \to 0} \frac{1}{\mathcal{H}^{N - 1}(O)} \sum \limits_{\substack{(y, \rho) \in M(\omega) \\ y \in \frac{1}{\varepsilon} O}} \varepsilon^{N - 1} \rho^{N - 2} = \lim_{\varepsilon \to 0} \frac{1}{\mathcal{H}^{N - 1}\left(\frac{1}{\varepsilon}O\right)} \sum \limits_{\substack{(y, \rho) \in M(\omega) \\ y \in \frac{1}{\varepsilon} O}} \rho^{N - 2}
\]
exists and is independent of $O$.

In this paper, we assume that $M$ is a stationary marked point process with finite intensity such that
\begin{equation} \label{eq:N-2_moment}
    \int_0^\infty \rho^{N - 2} \, d\lambda(\rho) < \infty,
\end{equation}
where $\lambda$ is, up to a multiplicative constant, the probability distribution of $\rho$. Due to the stationarity of $M$, the quantity in \eqref{eq:N-2_moment} is equivalent to the expectation of the random variable
\[
\frac{1}{\mathcal{H}^{N - 1}\left(\frac{1}{\varepsilon}O\right)} \sum \limits_{\substack{(y, \rho) \in M(\omega) \\ y \in \frac{1}{\varepsilon} O}} \rho^{N - 2}
\]
for any Borel set $O \subset U^0$. Under these assumptions, the Ergodic Theorem for marked point processes states that if $O \subset U^0$ is a sufficiently nice set, then there exists a random variable $\rho_0$ such that
\[
\lim_{\varepsilon \to 0} \frac{1}{\mathcal{H}^{N - 1}\left(\frac{1}{\varepsilon}O\right)} \sum \limits_{\substack{(y, \rho) \in M(\omega) \\ y \in \frac{1}{\varepsilon} O}} \rho^{N - 2} = \rho_0(\omega)^{N - 2}
\]
almost surely. In analogy with the result in the periodic setting, we prove that the factor $\gamma$ in \eqref{eq:limit_equations} is almost surely given by
\begin{equation} \label{eq:random_gamma}
\gamma(\omega) = \frac{\rho_0(\omega)^{N - 2}}{4} \cpct(B(0, 1) \cap \Sigma, \mathbb{R}^N)
\end{equation}
in the stochastic setting.

The proof of our main result, Theorem \ref{thm:homogenization_theorem}, relies on Tartar's \textit{method of oscillating test functions}. We explicitly construct functions that mimic the expected oscillatory behavior of solutions to \eqref{eq:poisson_sieve}, as established in Theorem \ref{thm:existence_of_oscillating_test_functions}. These functions serve as test functions in the weak formulation of \eqref{eq:poisson_sieve}, allowing us to extract information about the solutions $(u_\varepsilon)$. Our construction adapts techniques from the periodic case, where the separation between holes simplifies the analysis. However, when perforations are generated by a marked point process, clustering occurs with high probability. To address this, we follow Giunti, Höfer, and Velázquez' approach by treating clusters separately from sufficiently isolated holes. However, we do not impose any assumptions on the ergodicity of the marked point process or the correlations of the radii at large distances as in \cite{GHV}. The stochastic integrability assumption alone \eqref{eq:N-2_moment} ensures that clusters have asymptotically vanishing capacity, which, in turn, allows us to disregard them in the construction of oscillating test functions.

We note that the treatment of the periodic Neumann sieve in the literature is not limited to the method of oscillating test functions. In \cite{CDGO, O} the authors apply the periodic unfolding method to study variants of Neumann's sieve. In \cite{A, ABZ}, the classical homogenization result is extended to nonlinear variational problems using $\Gamma$-convergence. We opted for the method of oscillating test functions because this work is a first step toward a quantitative analysis in which the use of correctors plays a central role. For problems related to homogenization in similar geometries, we refer the reader to \cite{ABCP, AP, C1, C2, C, DMFZ, DV, DNPM, GPM, K} and references therein.

The Neumann sieve problem belongs to the broader class of boundary value problems in perforated domains whose study was initiated by the foundational works of Marchenko and Khruslov \cite{MK} and Cioranescu and Murat \cite{CM1, CM2} (see also \cite{CDG, DMD, DMG}). Early work on homogenization in randomly perforated domains is found in \cite{MK, PV}. Our work has been inspired by the recent application of the theory of marked point processes to the homogenization of Poisson's equation in random media \cite{GHV}. Further results have extended this framework to the homogenization of the Stokes equation \cite{GH} and nonlinear variational problems \cite{SZZ}.

To conclude the introduction, we outline the organization of the paper. In Section \ref{sec:preliminaries}, we provide a self-contained account of the theory of marked point processes relevant to our problem. Of particular importance is a variant of the Ergodic Theorem stated in Theorem \ref{thm:ergodic_theorem_second_version}, which we use frequently in the paper. A derivation of this result from the more classical Ergodic Theorem for marked point processes is included in the appendix. In Section \ref{sec:main_result}, we carefully describe the problem and state our main result, Theorem \ref{thm:homogenization_theorem}. In the same section, the theorem is proved, assuming the existence of the oscillating test functions stated in Theorem \ref{thm:existence_of_oscillating_test_functions}. In Section \ref{sec:analysis_of_perforations}, we define the important notions of $\varepsilon$-isolated and $\varepsilon$-cluster points. Roughly speaking, the $\varepsilon$-isolated points correspond to the centers of sufficiently isolated holes in $T_\varepsilon$, while the $\varepsilon$-cluster points correspond to holes that form clusters. In this section, we also introduce the thinning of point processes and use this notion in Theorem \ref{thm:vanishing_capacity_and_measure} to prove that the clusters have asymptotically vanishing capacities. Section \ref{sec:oscillating_test_functions} is devoted to the construction of the oscillating test functions. In Section \ref{ssec:cell_problem}, we study the so-called ``cell problem", which we use to define the test functions in the neighborhood of the isolated holes. In Sections \ref{ssec:construction_of_oscillations} and \ref{ssec:special_case}, we establish key properties of these functions in preparation for the proof of Theorem \ref{thm:existence_of_oscillating_test_functions}. Finally, the proof of Theorem \ref{thm:existence_of_oscillating_test_functions} is given in Section \ref{sec:proof_of_final_theorem}.

\section{Preliminaries} \label{sec:preliminaries}

In this section, we present the necessary definitions and key results from the theory of marked point processes.

\subsection{Fundamentals of marked point processes}

For the definitions and results presented in this section, we follow \cite[Chapter 4]{CSKM}. All definitions and theorems will be formulated in the context of Euclidean spaces, as this is sufficient for the purposes of our work. For a more comprehensive treatment of the subject, we refer the reader to the books \cite{CSKM, DVJ}.

Let $d \in \mathbb{N}$. If $O \subset \mathbb{R}^d$ is open, we denote the Borel $\sigma$-algebra on $O$ by $\mathcal{B}(O)$. The Lebesgue measure on $\mathbb{R}^d$ is denoted by $\mathcal{L}^d$. If $Z$ is any set, then we denote the cardinality of $Z$ by $\card(Z)$.

Intuitively, a point process on $\mathbb{R}^d$ is a set of randomly distributed points in $\mathbb{R}^d$. We speak of a marked point process (in short m.p.p.), if we attach a characteristic (a \textit{mark}) to each point. In our work, a mark will simply be a positive number. The set of positive numbers is denoted by $\mathbb{R}^+$.

To give the precise definition of an m.p.p. on $\mathbb{R}^d$ with positive marks, let us call $Y \subset \mathbb{R}^d \times \mathbb{R}^+$ \textit{admissible} if it satisfies the following conditions:
\begin{enumerate}[label=(\roman*)]
    \item Each point has a unique mark, that is, if $(y, \rho_1), (y, \rho_2) \in Y$, then $\rho_1 = \rho_2$,
    \item The projection of $Y$ onto $\mathbb{R}^d$ is \textit{locally finite}. Equivalently, the set $Y \cap (B \times \mathbb{R}^+)$ has finite cardinality for all bounded Borel sets $B \in \mathcal{B}(\mathbb{R}^d)$.
\end{enumerate}
We define $\mathbb{M}^d$ to be the set of all admissible $Y \subset \mathbb{R}^d \times \mathbb{R}^+$. Let $\mathcal{M}^d$ be the smallest $\sigma$-algebra on $\mathbb{M}^d$ that makes all functions $Y \mapsto \card(Y \cap E)$ measurable, as $E$ runs through the Borel sets in $\mathcal{B}(\mathbb{R}^d \times \mathbb{R}^+)$.

\begin{definition}
    A \textit{marked point process} on $\mathbb{R}^d$ with positive marks is a measurable mapping $M$ of a probability space $(\Omega, \mathcal{F}, \mathbb{P})$ into $(\mathbb{M}^d, \mathcal{M}^d)$. The \textit{probability distribution} of $M$ is the measure $\mathcal{L}_M$ defined by $\mathcal{L}_M(\mathcal{A}) = \mathbb{P}(M^{-1}(\mathcal{A}))$ for all $\mathcal{A} \in \mathcal{M}^d$. In other words $\mathcal{L}_M$ is the push-forward of $\mathbb{P}$ by $M$.
\end{definition}

\begin{remark}
    In the literature, point processes are also commonly defined as a class of random measures, that is, as measure-valued random variables. This is because we can naturally associate to each admissible set $Y$ the counting measure $E \mapsto \card(Y \cap E)$, where $E \in \mathcal{B}(\mathbb{R}^d \times \mathbb{R}^+)$. In what follows, we denote the measure associated with $Y$ by the same symbol. Hence, we write $Y(E)$ instead of $\card(Y \cap E)$.
\end{remark}

From now on, the term marked point process will refer exclusively to an m.p.p. on $\mathbb{R}^d$ with positive marks. Let $M : (\Omega, \mathcal{F}, \mathbb{P}) \rightarrow (\mathbb{M}^d, \mathcal{M}^d)$ be an m.p.p. If $E \in \mathcal{B}(\mathbb{R}^d \times \mathbb{R}^+)$, then the function $\omega \mapsto M(\omega)(E)$, where $\omega \in \Omega$, is an integer-valued random variable that represents the number of points in $M(\omega)$ lying in $E$. We denote this random variable simply by $M(E)$. The expected value $\mathbb{E}[M(E)]$ gives the average number of points in $M$ that belong to $E$. It is easy to check that the expected value defines a Borel measure on $\mathbb{R}^d \times \mathbb{R}^+$ as a function of $E$.

\begin{definition}
    Let $M : (\Omega, \mathcal{F}, \mathbb{P}) \rightarrow (\mathbb{M}^d, \mathcal{M}^d)$ be a marked point process. Define $\Lambda(E) := \mathbb{E}[M(E)]$ for all $E \in \mathcal{B}(\mathbb{R}^d \times \mathbb{R}^+)$. Then $\Lambda$ is a Borel measure called the \textit{intensity measure} of $M$. If $\Lambda(B \times \mathbb{R}^+) < \infty$ for all bounded Borel sets $B \in \mathcal{B}(\mathbb{R}^d)$, then we say that $M$ has \textit{finite intensity}.
\end{definition}

As the next theorem shows, the intensity measure enables the computation of the expected value of a random variable that depends on the m.p.p. We include the proof here, as it is elementary.

\begin{theorem}[Campbell's Theorem]
    Let $M : (\Omega, \mathcal{F}, \mathbb{P}) \rightarrow (\mathbb{M}^d, \mathcal{M}^d)$ be a marked point process and let $g : \mathbb{R}^d \times \mathbb{R}^+ \rightarrow [0, \infty)$ be a Borel measurable function. Define
    \[
    G(Y) := \sum_{(y, \rho) \in Y} g(y, \rho) = \int_{\mathbb{R}^d \times \mathbb{R}^+} g(y, \rho) \, dY(y, \rho),
    \]
    where $Y \in \mathbb{M}^d$. Then
    \[
    \mathbb{E}[G \circ M] = \int_{\mathbb{R}^d \times \mathbb{R}^+} g(y, \rho) \, d\Lambda(y, \rho).
    \]
\end{theorem}

\begin{proof}
    Suppose first that $g = \chi_E$ for some $E \in \mathcal{B}(\mathbb{R}^d \times \mathbb{R}^+)$. Then $G \circ M = M(E)$ and
    \begin{equation} \label{eq:campbell's_theorem}
        \mathbb{E}[G \circ M] = \Lambda(E) = \int_{\mathbb{R}^d \times \mathbb{R}^+} g(y, \rho) \, d\Lambda(y, \rho). 
    \end{equation}
    Since both the left- and right-hand sides of \eqref{eq:campbell's_theorem} are linear in $g$, we can extend the equality to simple functions and, by applying the Monotone Convergence Theorem, to all nonnegative Borel functions.
\end{proof}

Now, we introduce an important class of marked point processes. If $Y \in \mathbb{M}^d$ and $\tau \in \mathbb{R}^d$, we define
\[
Y_\tau := \{(y + \tau, \rho) : (y, \rho) \in Y\}.
\]
Analogously, if $\mathcal{A} \in \mathcal{M}^d$, then $\mathcal{A}_\tau := \{Y_\tau : Y \in \mathcal{A}\}$.

\begin{definition}
    A marked point process $M : (\Omega, \mathcal{F}, \mathbb{P}) \rightarrow (\mathbb{M}^d, \mathcal{M}^d)$ is called \textit{stationary} if $\mathcal{L}_M(\mathcal{A}) = \mathcal{L}_M(\mathcal{A}_\tau)$ for all $\tau \in \mathbb{R}^d$ and $\mathcal{A} \in \mathcal{M}^d$.
\end{definition}

Intuitively, stationarity means that the points of the m.p.p. are distributed spatially homogeneously. For a concrete interpretation of the definition, let us take $B \in \mathcal{B}(\mathbb{R}^d)$, $L \in \mathcal{B}(\mathbb{R}^+)$, and consider the set $\mathcal{A} = \{Y \in \mathbb{M}^d : Y(B \times L) = k\}$, where $k \in \mathbb{N}$. Then $\mathcal{A}_\tau = \{Y \in \mathbb{M}^d : Y((\tau + B) \times L) = k \}$. If $M$ is a stationary m.p.p., then $\mathcal{L}_M(\mathcal{A}) = \mathcal{L}_M(\mathcal{A}_\tau)$. Hence, the probability that $M$ has $k$ points in $B \times L$ is equal to the probability that it has $k$ points in $(\tau + B) \times L$. 

For a stationary m.p.p., the intensity measure can be factorized as shown in the following proposition. Since the proof is instructive, we include it here.

\begin{proposition} \label{pr:intensity_measure_stationary}
    Assume $M : (\Omega, \mathcal{F}, \mathbb{P}) \rightarrow (\mathbb{M}^d, \mathcal{M}^d)$ is a stationary marked point process with finite intensity. Then there exists a finite Borel measure $\lambda$ on $\mathbb{R}^+$ such that
    \begin{equation} \label{eq:intensity_measure_splits}
        \Lambda =  \lambda \otimes \mathcal{L}^d.
    \end{equation}
\end{proposition}

\begin{proof}
    Fix $L \in \mathcal{B}(\mathbb{R}^+)$. We begin by showing that $\Lambda(B \times L) = \Lambda((\tau + B) \times L)$ for all $\mathcal{B}(\mathbb{R}^d)$ and $\tau \in \mathbb{R}^d$. By the definition of stationarity, we know that $\mathcal{L}_M(\mathcal{A}) = \mathcal{L}_M(\mathcal{A}_\tau)$ for all $\mathcal{A} \in \mathcal{M}^d$ and $\tau \in \mathbb{R}^d$. As a consequence,
    \[
    \int_{\mathbb{M}^d} \chi_\mathcal{A}(Y) \, d\mathcal{L}_M(Y) = \int_{\mathbb{M}^d} \chi_{\mathcal{A}_{-\tau}}(Y) \, d\mathcal{L}_M(Y) = \int_{\mathbb{M}^d} \chi_{\mathcal{A}}(Y_\tau) \, d\mathcal{L}_M(Y).
    \]
    Using the linearity of the integral, we deduce
    \begin{equation} \label{eq:stationarity_in_terms_of_characteristic_functions}
        \int_{\mathbb{M}^d} F(Y) \, d\mathcal{L}_M(Y) = \int_{\mathbb{M}^d} F(Y_\tau) \, d\mathcal{L}_M(Y)
    \end{equation}
    for all simple functions $F : \mathbb{M}^d \rightarrow [0, \infty)$. By the Monotone Convergence Theorem, the equality in \eqref{eq:stationarity_in_terms_of_characteristic_functions} can be extended to all nonnegative random variables on $\mathbb{M}^d$. Hence, we get
    \begin{multline*}
    \Lambda((\tau + B) \times L) = \int_{\mathbb{M}^d} Y((\tau + B) \times L) \, d\mathcal{L}_M(Y) = \int_{\mathbb{M}^d} Y_\tau((\tau + B) \times L) \, d\mathcal{L}_M(Y) \\
    = \int_{\mathbb{M}^d} Y(B \times L) \, d\mathcal{L}_M(Y) = \Lambda(B \times L) \quad \text{for all } \tau \in \mathbb{R}^d.
    \end{multline*}
    Since any translation invariant Borel measure on $\mathbb{R}^d$ is proportional to the Lebesgue measure, we conclude that there exists a number $\lambda(L) \ge 0$ such that $\Lambda(B \times L) = \lambda(L) \mathcal{L}^d(B)$ for all $B \in \mathcal{B}(\mathbb{R}^d)$, $L \in \mathcal{B}(\mathbb{R}^+)$. By fixing $B$ and varying $L$, we observe that $\lambda$ is a finite measure.
\end{proof}

\begin{remark}
    Assume $\mathcal{L}^d(B) = 1$. Then \eqref{eq:intensity_measure_splits} yields $\Lambda(B \times L) = \lambda(L)$ for all $L \in \mathcal{B}(\mathbb{R}^+)$. Hence, we can interpret $\lambda(L)$ as the mean number of points of $M$ with marks in $L$ per unit volume.
\end{remark}

\subsection{Ergodic theorem for stationary marked point processes}

In this section, we collect some important consequences of stationarity, most notably the Ergodic Theorem, which provides information on spatial averages of stationary m.p.p. We closely follow \cite[Chapter 12.2]{DVJ}.

Let $M : (\Omega, \mathcal{F}, \mathbb{P}) \rightarrow (\mathbb{M}^d, \mathcal{M}^d)$ be an m.p.p. Assume $(C_k) \subset \mathbb{R}^d$ is an increasing sequence of sets such that $\bigcup_{k = 1}^\infty C_k = \mathbb{R}^d$. Often, one is interested in the spatial averages
\begin{equation} \label{eq:spatial_average}
    \frac{1}{\mathcal{L}^d(C_k)} \sum \limits_{\substack{(y, \rho) \in M(\omega) \\ y \in C_k}} g(\rho)
\end{equation}
for a nonnegative function $g$ and $\omega \in \Omega$. For instance, if we take $g(\rho) = \rho$, then \eqref{eq:spatial_average} is an approximation to the average mark size in $M(\omega)$ per unit volume.

Below, we present sufficient conditions under which the limit of \eqref{eq:spatial_average} as $k \to \infty$ can be represented as a conditional expectation with respect to the $\sigma$-algebra of \textit{translation invariant sets}.

\begin{definition}
    Let $M : (\Omega, \mathcal{F}, \mathbb{P}) \rightarrow (\mathbb{M}^d, \mathcal{M}^d)$ be a marked point process. A set $\mathcal{A} \in \mathcal{M}^d$ is called \textit{invariant} if $\mathcal{L}_M(\mathcal{A} \,\triangle\, \mathcal{A}_\tau) = 0$ for all $\tau \in \mathbb{R}^d$. The $\sigma$-algebra of invariant sets is denoted by $\mathcal{I}$.
\end{definition}
We denote the $\sigma$-algebra $\{M^{-1}(\mathcal{A}) \in \mathcal{F}: \mathcal{A} \in \mathcal{I}\}$ by $M^{-1}(\mathcal{I})$. By an $M^{-1}(\mathcal{I})$-measurable random measure we mean a function $\xi : \Omega \times \mathcal{B}(\mathbb{R}^d \times \mathbb{R}^+) \rightarrow [0, \infty]$ such that $\xi(\omega, \cdot)$ is a Borel measure for all $\omega \in \Omega$, and $\xi(\cdot, E)$ is $M^{-1}(\mathcal{I})$-measurable for all $E \in \mathcal{B}(\mathbb{R}^d \times \mathbb{R}^+)$.

The following lemma shows that the conditional expectation of a random variable with respect to the $\sigma$-algebra $M^{-1}(\mathcal{I})$ can be expressed as an integral with respect to a random measure.

\begin{lemma}[Lemma 12.2.III \cite{DVJ}] \label{lm:conditional_expectation_random_measure}
    Let $M : (\Omega, \mathcal{F}, \mathbb{P}) \rightarrow (\mathbb{M}^d, \mathcal{M}^d)$ be a stationary marked point process with finite intensity. Then there exists an $M^{-1}(\mathcal{I})$-measurable random measure $\xi$ such that 
    \begin{equation} \label{eq:conditional_expectation_integral_representation}
        \mathbb{E}[G \circ M \,|\, M^{-1} (\mathcal{I})](\omega) = \int_{\mathbb{R}^d} \int_0^\infty g(y, \rho) \, d\xi(\omega, \rho) \, dy \quad \mathbb{P}\text{-a.s.}
    \end{equation}
    for all Borel measurable $g : \mathbb{R}^d \times \mathbb{R}^+ \rightarrow [0, \infty)$, where
    \[
    G(Y) := \int_{\mathbb{R}^d \times \mathbb{R}^+} g(y, \rho) \, dY(y, \rho).
    \]
    In particular,
    \begin{equation} \label{eq:conditional_expectation_measure_representation}
        \mathbb{E}[M(B \times L) \,|\, M^{-1}(\mathcal{I})](\omega) = \xi(\omega, L) \mathcal{L}^d(B) \quad \mathbb{P}\text{-a.s.}
    \end{equation}
    for all bounded $B \in \mathcal{B}(\mathbb{R}^d)$ and for all $L \in \mathcal{B}(\mathbb{R}^+)$.
\end{lemma}

\begin{remark}
    Assume $\mathcal{L}^d(B) = 1$. Then \eqref{eq:conditional_expectation_measure_representation} yields $\mathbb{E}[M(B \times L) \,|\, M^{-1}(\mathcal{I})](\omega) = \xi(\omega, L)$ $\mathbb{P}$-a.s. for all $L \in \mathcal{B}(\mathbb{R}^+)$. Hence, we observe that the random measure $\xi$ is a conditional analogue of the measure $\lambda$ from Proposition \ref{pr:intensity_measure_stationary}. More precisely, $\xi(\cdot, L)$ is the mean number of points of $M$ with marks in $L$ per unit volume conditioned on the $\sigma$-algebra $M^{-1}(\mathcal{I})$.
\end{remark}

\begin{remark}
    If $\mathcal{L}_M(\mathcal{A}) = 0$ or $\mathcal{L}_M(\mathcal{A}) = 1$ for all $\mathcal{A} \in \mathcal{I}$, then $M$ is called \textit{ergodic}. For stationary ergodic m.p.p., the conditional expectations in \eqref{eq:conditional_expectation_integral_representation} and \eqref{eq:conditional_expectation_measure_representation} reduce to regular expectations. In this case, the random measure $\xi$ equals $\lambda$ and the lemma reduces to Campbell's Theorem.
\end{remark}

\begin{definition}
    A sequence of Borel sets $(C_k)$ in $\mathbb{R}^d$ is called a \textit{convex averaging sequence} if
    \begin{enumerate}[label={(\roman*)}]
        \item Each $C_k$ is convex and bounded,
        \item $C_k \subset C_{k + 1}$ for all $k \in \mathbb{N}$,
        \item $\lim_{k \to \infty} \sup \{r \in \mathbb{R}^+ : C_k \text{ contains a ball of radius }  r \} = \infty$.
    \end{enumerate}
\end{definition}

We are finally ready to state the first version of the Ergodic Theorem for stationary m.p.p.

\begin{theorem}[Theorem 12.2.IV \cite{DVJ}] \label{thm:ergodic_theorem_first_version}
    Let $M : (\Omega, \mathcal{F}, \mathbb{P}) \rightarrow (\mathbb{M}^d, \mathcal{M}^d)$ be a stationary marked point process with finite intensity. Let $\xi$ be the random measure defined in Lemma \ref{lm:conditional_expectation_random_measure}. Then
    \[
    \lim_{k \to \infty} \frac{1}{\mathcal{L}^d(C_k)} \sum \limits_{\substack{(y, \rho) \in M(\omega) \\ y \in C_k}} g(\rho) = \int_0^\infty g(\rho) \, d\xi(\omega, \rho) \quad \mathbb{P}\text{-a.s.}
    \]
    for all convex averaging sequences $(C_k)$ and all $\lambda$-integrable functions $g : \mathbb{R}^+ \rightarrow [0, \infty)$, where $\lambda$ is the measure on $\mathbb{R}^+$ given by $\Lambda = \lambda \otimes \mathcal{L}^d$.
\end{theorem}

In what follows, we shall work with the following variation of the ergodic theorem, which does not require the averaging sets to be convex.

\begin{theorem} \label{thm:ergodic_theorem_second_version}
    Let $M : (\Omega, \mathcal{F}, \mathbb{P}) \rightarrow (\mathbb{M}^d, \mathcal{M}^d)$ be a stationary marked point process with finite intensity. Let $\xi$ be the random measure defined in Lemma \ref{lm:conditional_expectation_random_measure}. Assume $B \in \mathcal{B}(\mathbb{R}^d)$ is bounded with nonempty interior such that $\mathcal{L}^d(\partial B) = 0$. Then
    \begin{equation} \label{eq:ergodic_theorem_second_version}
        \lim_{\varepsilon \to 0} \frac{1}{\mathcal{L}^d\left(\frac{1}{\varepsilon} B\right)} \sum \limits_{\substack{(y, \rho) \in M(\omega) \\ y \in \frac{1}{\varepsilon} B}} g(\rho) = \int_0^\infty g(\rho) \, d\xi(\omega, \rho) \quad \mathbb{P}\text{-a.s.}
    \end{equation}
    for all $\lambda$-integrable functions $g : \mathbb{R}^+ \rightarrow [0, \infty)$, where $\lambda$ is the measure on $\mathbb{R}^+$ given by $\Lambda = \lambda \otimes \mathcal{L}^d$.
\end{theorem}

As the proof of this theorem is not essential for the rest of our work, we postpone it to Section \ref{ssec:proof_of_ergodic_theorem} of the appendix.

\section{Statement of the main result} \label{sec:main_result}

In this section, we describe the setting of the problem and state our main result. First, we introduce the notation that we will frequently use in the rest of the article.

\subsection{Notation}

Let $N \in \mathbb{N}$. Given $x \in \mathbb{R}^{N - 1}$, we define $\bar{x} := (x_1, \dots, x_{N - 1}, 0) \in \mathbb{R}^N$. For $A \subset \mathbb{R}^N$ and $\tau \in \mathbb{R}^N$, we define the sets
\begin{alignat*}{2}
    &A^- &&:= \{x \in A : x_N < 0\}, \\
    &A^0 &&:= \{x \in A : x_N = 0\}, \\
    &A^+ &&:= \{x \in A : x_N > 0\}, \\
    \tau &+ A &&:= \{\tau + x : x \in A\}.
\end{alignat*}
The hyperplane $\{x \in \mathbb{R}^N : x_N = 0\}$ is denoted by $\Sigma$. Similarly, for a function $f : A \rightarrow \mathbb{R}$, we define $f^+ := f|_{A^+}$, $f^- := f|_{A^-}$. The characteristic function of $A$ is denoted by $\chi_A$, i.e.; $\chi_A(x) = 1$ if $x \in A$, and $\chi_A(x) = 0$ if $x \in \mathbb{R}^N \setminus A$. If $A' \subset \mathbb{R}^N$, then we define the distance between $A$ and $A'$ by $\dist(A, A') := \inf\{|x - x'| : x \in A, \, x' \in A' \}$. We use $B(x, r)$ to denote the the open ball centered at $x \in \mathbb{R}^N$ with radius $r$.

Given $A \subset \Sigma$ and an open set $B \subset \mathbb{R}^N$ containing $\bar{A}$, the harmonic capacity, or simply capacity, of $A$ with respect to $B$ is defined as
\[
\cpct (A, B) := \inf \left\{\int_B |\nabla v|^2 \, dx : v \in H_0^1(B), \, v = 1 \text{ in $A$ in the trace sense}\right\}.
\]
As before, the Lebesgue measure on $\mathbb{R}^N$ is denoted by $\mathcal{L}^N$. If $s \ge 0$, the $s$-dimensional Hausdorff measure on $\mathbb{R}^N$ is denoted by $\mathcal{H}^s$. If $Z$ is any set, then we denote the cardinality of $Z$ by $\card(Z)$.

Finally, we remark that in our estimates $C$ stands for any strictly positive constant that can be explicitly computed in terms of known quantities. The value of $C$ may therefore change from line to line in a given computation. If $a$, $b > 0$, then we also use $a \lesssim b$ as a shorthand for the inequality $a \le Cb$ for some constant $C$ depending only on known quantities.

\subsection{Main result}

We start this section by introducing the setting of the problem. Let $N \in \mathbb{N}$ with $N \ge 3$. We fix a bounded and open set $U \subset \mathbb{R}^N$ with Lipschitz boundary. We assume that $0 \in U$ and that the hyperplane $\Sigma$ intersects $U$ transversely. By the latter assumption, we mean that the sets $U^+$ and $U^-$ have Lipschitz boundary, and that $\dist(K, \partial U) > 0$ for any compact set $K \subset U^0$. 

We fix a stationary m.p.p. on $\mathbb{R}^{N - 1}$, that is, we consider $M : (\Omega, \mathcal{F}, \mathbb{P}) \rightarrow (\mathbb{M}^{N - 1}, \mathcal{M}^{N - 1})$ and we assume it has finite intensity. Let $\lambda$ and $\xi$ be as in Proposition \ref{pr:intensity_measure_stationary} and Lemma \ref{lm:conditional_expectation_random_measure}, respectively. We assume that
\begin{equation} \label{eq:moment_condition}
    \int_0^\infty \rho^{N - 2} \, d\lambda(\rho) < \infty.
\end{equation}

We now define the randomly perforated sieve. For $\varepsilon > 0$ and $\omega \in \Omega$, we consider the set of holes
\begin{equation} \label{eq:perforations}
    T_\varepsilon(\omega) := \bigcup \limits_{\substack{(y, \rho) \in M(\omega) \\ \bar{y} \in \frac{1}{\varepsilon} U^0}} B\left(\varepsilon \bar{y}, \varepsilon^\frac{N - 1}{N - 2}\rho\right) \cap U^0.
\end{equation}

We imagine the domain $U$ separated into $U^+$ and $U^-$ by the hyperplane $\Sigma$. The sets $U^+$ and $U^-$ are then connected through the perforations $T_\varepsilon(\omega)$ along the common boundary $U^0$. We set $U_\varepsilon(\omega):= U^+ \cup U^- \cup T_\varepsilon(\omega)$ (see Figure \ref{fig:domain}). The sieve is represented by $U^0 \setminus T_\varepsilon(\omega)$, which is a subset of the boundary of $U_\varepsilon(\omega)$. We note that the fact that the hyperplane $\Sigma$ intersects $U$ transversely implies that $U_\varepsilon(\omega)$ is an open set.

Let $f \in L^2(U)$ be fixed. For $\varepsilon > 0$ and $\omega \in \Omega$, we consider Poisson's equation in $U_\varepsilon(\omega)$ with Dirichlet boundary conditions on the outer boundary and Neumann boundary conditions on the sieve:
\begin{equation} \label{eq:poisson's_equation}
    \begin{cases}
        \begin{aligned}
            -\Delta u_\varepsilon(\omega, \cdot) &= f \quad \text{in } U_\varepsilon(\omega), \\
            u_\varepsilon(\omega, \cdot) &= 0 \quad \text{on } \partial U, \\
            \nabla u_\varepsilon(\omega, \cdot) \cdot \nu&= 0 \quad \text{on } U^0 \setminus T_\varepsilon(\omega).
        \end{aligned}
    \end{cases}
\end{equation}
We remark that the Neumann boundary condition holds separately on both sides of the sieve. Hence, the unit vector $\nu$ takes either the value $e_N$ or $-e_N$ depending the side of the sieve, where $e_N = (0, \dots, 0, 1)$. 

We work with the weak formulation of \eqref{eq:poisson's_equation}. We define the function space $V_\varepsilon(\omega)$ by
\[
V_\varepsilon(\omega) := \{u \in H^1(U_\varepsilon(\omega)) : u = 0 \text{ on } \partial U \text{ in the trace sense}\}.
\]
Clearly, we have $V_\varepsilon(\omega) \neq H_0^1(U_\varepsilon(\omega))$, as functions belonging to the latter space vanish on the sieve $U^0 \setminus T_\varepsilon(\omega)$ as well. With this notation, problem \eqref{eq:poisson's_equation} is formally equivalent to
\begin{equation} \label{eq:weak_poisson's_equation}
    \begin{cases}
        \begin{aligned}
            &u_\varepsilon(\omega, \cdot) \in V_\varepsilon(\omega), \\
            &\int_{U_\varepsilon(\omega)} \nabla u_\varepsilon(\omega, x) \nabla \varphi(x) \, dx = \int_{U_\varepsilon(\omega)} f(x) \, \varphi(x) \, dx \quad \text{for all } \varphi \in V_\varepsilon(\omega).
        \end{aligned}
    \end{cases}
\end{equation}
Let $\varepsilon > 0$ and $\omega \in \Omega$ be fixed. Then the Lax-Milgram Theorem ensures the existence of a unique solution $u_\varepsilon(\omega, \cdot)$ to \eqref{eq:weak_poisson's_equation}. 

The main goal of this paper is to show the existence of a set $\Omega' \subset \Omega$ with $\mathbb{P}(\Omega') = 1$ such that for every $\omega \in \Omega'$ the sequence $(u_\varepsilon(\omega, \cdot))$ converges weakly, as $\varepsilon \to 0$, to functions $u^+(\omega, \cdot)$ and $u^-(\omega, \cdot)$ belonging to $H^1(U^+)$ and $H^1(U^-)$, respectively. Moreover, $u^+$ and $u^-$ solve two limit boundary value problems which are coupled through a so-called transmission condition on $U^0$. Indeed, we can test equation \eqref{eq:weak_poisson's_equation} with an element of $H_0^1(U^+)$ or $H_0^1(U^-)$ and easily pass to the limit to prove that $u^+$ and $u^-$ must satisfy the Poisson equation with the same source term as in \eqref{eq:poisson's_equation}. Then the challenge is to identify the limit boundary condition on $U^0$. It will be observed that the normal derivatives of $u^+$ and $u^-$ on $U^0$ are coupled through the appearance of a capacitary term.

We are now ready to state the main theorem of our paper.

\begin{theorem} \label{thm:homogenization_theorem}
    Let $M :(\Omega, \mathcal{F}, \mathbb{P}) \rightarrow (\mathbb{M}^{N - 1}, \mathcal{M}^{N - 1})$ be a stationary marked point process with finite intensity such that
    \[
    \int_0^\infty \rho^{N - 2} \, d\lambda(\rho) < \infty.
    \]
    Let $u_\varepsilon(\omega, \cdot)$ be the unique solution of \eqref{eq:weak_poisson's_equation}. For $\mathbb{P}$-a.e. $\omega \in \Omega$, there exist functions $u^+(\omega, \cdot) \in H^1(U^+)$, $u^-(\omega, \cdot) \in H^1(U^-)$ such that
    \[
    u_\varepsilon(\omega, \cdot) \rightharpoonup u^+(\omega, \cdot) \text{ in } H^1(U^+), \quad u_\varepsilon(\omega, \cdot) \rightharpoonup u^-(\omega, \cdot) \text{ in } H^1(U^-) 
    \]
    as $\varepsilon \to 0$. Furthermore, the functions $u^\pm(\omega, \cdot)$ satisfy
    \begin{equation} \label{eq:effective_equation_weak_form}
        \int_{U^\pm} \nabla u^\pm(\omega, x) \nabla \varphi(x) \, dx = \int_U f(x) \, \varphi(x) \, dx \mp \gamma(\omega) \int_{U^0} (u^+(\omega, x) - u^-(\omega, x)) \varphi(x) \, d\mathcal{H}^{N - 1}(x)
    \end{equation}
    for all $\varphi \in H^1(U^\pm)$ that vanish on $(\partial U)^\pm$ in the trace sense, where $\gamma : \Omega \rightarrow \mathbb{R}$ is a random variable defined as
    \begin{equation} \label{eq:effective_factor}
        \gamma(\omega) := \frac{1}{4} \cpct\left(B(0, 1)^0, \mathbb{R}^N\right) \int_0^\infty \rho^{N - 2} \, d\xi(\omega, \rho).
    \end{equation}
\end{theorem}

\begin{remark}
    We note that \eqref{eq:effective_equation_weak_form} is the weak formulation of the following Poisson equation:
    \begin{equation*}
        \begin{cases}
            \begin{aligned}
                -\Delta u^\pm(\omega, \cdot) &= f \quad \text{in } U^\pm, \\
                u^\pm(\omega, \cdot) &= 0 \quad \text{on } (\partial U)^\pm, \\
                D_\nu u^\pm(\omega, \cdot) &= \mp \gamma(\omega)(u^+(\omega, \cdot) - u^-(\omega, \cdot)) \quad \text{on } U^0.
            \end{aligned}
        \end{cases}
    \end{equation*}
\end{remark}

\begin{remark}
    The homogeneous Dirichlet boundary condition in \eqref{eq:poisson's_equation} does not have an influence on the transmission condition across $U^0$. We have chosen it for simplicity. An analogous homogenization result holds true when considering other Dirichlet or Neumann boundary conditions on $\partial U$.
\end{remark}

We prove Theorem \ref{thm:homogenization_theorem} using the \textit{method of oscillating test functions}. The method involves the construction of suitable test functions that help extract information about the limit functions $u^+$ and $u^-$. The almost sure existence of a sequence of random oscillating test functions $(w_\varepsilon)$ possessing the desired properties is guaranteed by the following theorem.

\begin{theorem} \label{thm:existence_of_oscillating_test_functions}
    Let $M :(\Omega, \mathcal{F}, \mathbb{P}) \rightarrow (\mathbb{M}^{N - 1}, \mathcal{M}^{N - 1})$ be a stationary marked point process with finite intensity such that
    \[
    \int_0^\infty \rho^{N - 2} \, d\lambda(\rho) < \infty.
    \]
    For $\mathbb{P}$-a.e. $\omega \in \Omega$, there exist functions $w_\varepsilon(\omega, \cdot) \in H^1(U_\varepsilon(\omega))$ such that $w_\varepsilon(\omega, \cdot) \rightharpoonup \pm 1$ in $H^1(U^\pm)$ as $\varepsilon \to 0$. Moreover, the sequence $(w_\varepsilon(\omega, \cdot))$ satisfies the following property: Given $(v_\varepsilon)$ with $v_\varepsilon \in H^1(U_\varepsilon(\omega))$, if there exist functions $v^+$, $v^-$ and a subsequence $(v_{\varepsilon_k})$ such that $v_{\varepsilon_k} \rightharpoonup v^\pm$ in $H^1(U^\pm)$ as $k \to \infty$, then
    \begin{equation} \label{eq:product_of_gradients_convergence}
        \lim_{k \to \infty} \int_{U_{\varepsilon_k}(\omega)} \nabla w_{\varepsilon_k}(\omega, x) \nabla v_{\varepsilon_k}(x) \, dx = 2\gamma(\omega) \int_{U^0} v^+(x) - v^-(x) \, d\mathcal{H}^{N - 1}.
    \end{equation}
\end{theorem}

The proof of Theorem \ref{thm:existence_of_oscillating_test_functions} will be postponed to Section \ref{sec:proof_of_final_theorem}. With this theorem at our disposal, the proof of the homogenization result is straightforward.

\begin{proof}[Proof of Theorem \ref{thm:homogenization_theorem}]
    We fix a realization $\omega \in \Omega$ for which the existence of the sequence $(w_\varepsilon(\omega, \cdot))$ in Theorem \ref{thm:existence_of_oscillating_test_functions} is guaranteed. For simplicity, in what follows we omit the dependence on $\omega$. Choosing $\varphi = u_\varepsilon$ in \eqref{eq:weak_poisson's_equation} and using Poincaré's Inequality, we obtain the a priori estimate
    \[
    \|u_\varepsilon\|_{H^1(U_\varepsilon)} \le C\|f\|_{L^2(U)},
    \]
    where $C$ is independent of $\varepsilon$ and $\omega$. Hence, the sequence $(u_\varepsilon)$ is uniformly bounded in $H^1(U^+)$ and $H^1(U^-)$. It follows from Rellich-Kondrachov's Compactness Theorem that there exist functions $u^+$, $u^-$, and a subsequence $(\varepsilon_k)$, all depending on $\omega$, such that $u_{\varepsilon_k} \to u^\pm$ in $L^2(U^\pm)$ and $\nabla u_{\varepsilon_k} \rightharpoonup \nabla u^\pm$ in $H^1(U^\pm)$ as $k \to \infty$.

    Next, we prove that $u^\pm$ solves \eqref{eq:effective_equation_weak_form}; we do it only for $u^+$, the proof for $u^-$ being analogous. Let $\varphi \in C_c^\infty(U)$. We have
    \begin{equation} \label{eq:add_and_subtract}
        \int_{U^+} \nabla u^+ \nabla \varphi \, dx = \frac{1}{2}\left(\int_{U^+} \nabla u^+ \nabla \varphi \, dx + \int_{U^-} \nabla u^- \nabla \varphi \, dx\right) + \frac{1}{2}\left(\int_{U^+} \nabla u^+ \nabla \varphi \, dx - \int_{U^-} \nabla u^- \nabla \varphi \, dx\right).
    \end{equation}
    By \eqref{eq:weak_poisson's_equation}, we get
    \begin{equation} \label{eq:easy_limit}
        \int_{U^+} \nabla u^+ \nabla \varphi \, dx + \int_{U^-} \nabla u^- \nabla \varphi \, dx = \lim_{k \to \infty} \int_{U_{\varepsilon_k}} \nabla u_{\varepsilon_k} \nabla \varphi \, dx = \int_U f \varphi \, dx.
    \end{equation}
    Moreover, by the strong $L^2(U^\pm)$-convergence of $(w_\varepsilon)$ to $\pm 1$ we deduce that
    \[
    \lim_{k \to \infty} \int_{U_{\varepsilon_k}} (\nabla u_{\varepsilon_k} \nabla \varphi) w_{\varepsilon_k} \, dx = \int_{U^+} \nabla u^+ \nabla \varphi \, dx - \int_{U^-} \nabla u^- \nabla \varphi \, dx,
    \]
    while a simple manipulation of the terms also gives
    \begin{equation} \label{eq:simple_manipulation}
        \int_{U_{\varepsilon_k}} (\nabla u_{\varepsilon_k} \nabla \varphi) w_{\varepsilon_k} \, dx = \int_{U_{\varepsilon_k}} \nabla u_{\varepsilon_k} \nabla (w_{\varepsilon_k} \varphi) \, dx + \int_{U_{\varepsilon_k}} (\nabla w_{\varepsilon_k} \nabla \varphi) u_{\varepsilon_k} \, dx - \int_{U_{\varepsilon_k}} \nabla w_{\varepsilon_k} \nabla (\varphi u_{\varepsilon_k}) \, dx.
    \end{equation}
    Using \eqref{eq:weak_poisson's_equation} again, along with the fact that  $w_{\varepsilon_k} \varphi$ is an admissible test function, we can pass to the limit in the first term on the right-hand side of \eqref{eq:simple_manipulation} and get
    \begin{equation} \label{eq:the_other_easy_limit}
        \lim_{k \to \infty} \int_{U_{\varepsilon_k}} \nabla u_{\varepsilon_k} \nabla (w_{\varepsilon_k} \varphi) \, dx = \lim_{k \to \infty} \int_U f (w_{\varepsilon_k} \varphi) \, dx = \int_{U^+} f \varphi \, dx - \int_{U^-} f \varphi \, dx.
    \end{equation}
    Furthermore, using the fact that $\nabla w_\varepsilon \rightharpoonup 0$ in $L^2(U; \mathbb{R}^N)$ and that $(u_{\varepsilon_k})$ converges strongly in $L^2(U^\pm)$ as $\varepsilon \to 0$, we also obtain
    \[
    \lim_{k \to \infty} \int_{U_{\varepsilon_k}} (\nabla w_{\varepsilon_k} \nabla \varphi) u_{\varepsilon_k} \, dx = 0.
    \]
    To evaluate the limit of the last term on the right-hand side of \eqref{eq:simple_manipulation}, we invoke \eqref{eq:product_of_gradients_convergence}. To this end, set $v_\varepsilon := u_\varepsilon \varphi$. Then $v_{\varepsilon_k} \rightharpoonup v^\pm$ in $H^1(U^\pm)$ as $k \to \infty$ with $v^\pm := u^\pm \varphi$. Therefore, by virtue of Theorem \ref{thm:existence_of_oscillating_test_functions}, we get
    \begin{equation} \label{eq:hard_limit}
        \lim_{k \to \infty} \int_{U_{\varepsilon_k}} \nabla w_{\varepsilon_k} \nabla (u_{\varepsilon_k} \varphi) \, dx = 2\gamma \int_{U^0} (u^+ - u^-) \varphi \, d\mathcal{H}^{N - 1}.
    \end{equation}
    Finally, by combining \eqref{eq:add_and_subtract}, \eqref{eq:easy_limit}, \eqref{eq:the_other_easy_limit} and \eqref{eq:hard_limit}, we obtain
    \begin{equation} \label{eq:conclusion_for_smooth_test_functions}
        \int_{U^+} \nabla u^+ \nabla \varphi \, dx = \int_{U^+} f \varphi \, dx - \gamma \int_{U^0} (u^+ - u^-) \varphi \, d\mathcal{H}^{N - 1}.
    \end{equation}
    By the density of $C_c^\infty(U)$ in $H_0^1(U)$, we deduce that \eqref{eq:conclusion_for_smooth_test_functions} holds for all $\varphi \in H_0^1(U)$. In order to prove that \eqref{eq:conclusion_for_smooth_test_functions} holds also for test functions in $H^1(U^+)$ that vanish on $(\partial U)^+$, it suffices to show that these can be extended to test functions in $H_0^1(U)$. But this easily follows from the fact that $\Sigma$ intersects $U$ transversely. This concludes the proof of \eqref{eq:effective_equation_weak_form}.

    Lastly, since $u^+$ and $u^-$ are uniquely determined by \eqref{eq:effective_equation_weak_form}, we can actually deduce that the whole sequence $(u_\varepsilon)$ converges weakly to $u^\pm$ in $H^1(U^\pm)$ as $\varepsilon \to 0$, hence the claim.
\end{proof}

\section{Analysis close to the perforations} \label{sec:analysis_of_perforations}

To prove Theorem \ref{thm:existence_of_oscillating_test_functions}, we will adapt the construction of the so-called oscillating test functions introduced in the periodic setting (see \cite[Lemma 3.1]{D}, \cite[Lemma 2.2]{P}) to the stochastic setting, in the spirit of Giunti, Höfer and Velazquez. We observe that in the periodic setting the holes are well-separated on an $\varepsilon$-scale. This fact plays a crucial role in the construction, since it allows the test functions to be defined locally around each hole. However, in the stochastic setting, the centers of the holes are randomly distributed and their radii are unbounded random variables. As a result, the holes overlap with probability one. Nevertheless, in this section we will show that, thanks to the stochastic integrability assumption \eqref{eq:moment_condition}, the clusters of holes have asymptotically vanishing capacities almost surely. In turn, this ensures that we can construct test functions for which the $L^2$-norm of the gradients vanish asymptotically in a suitable neighborhood of the clusters. In other words, the formation of clusters does not prevent homogenization from occurring.

We now introduce some new notation. For $Y \in \mathbb{M}^{N - 1}$ and $(y, \rho) \in Y$, we define
\begin{equation} \label{eq:distance_to_closest_neighbor}
    d_Y(y) = \min \{|y - y'| :  (y', \rho') \in Y \setminus (y, \rho) \}
\end{equation}
if $Y \setminus (y, \rho)$ is nonempty. Otherwise, we set $d_Y(y) = \infty$. In other words, $d_Y(y)$ is the distance between $y$ and its closest neighbor in $Y$. We note that 
\begin{equation} \label{eq:nonintersecting_balls}
    B(\bar{y}_1, d_Y(y_1)/2) \cap B(\bar{y}_2, d_Y(y_2)/2) = \emptyset
\end{equation}
for all $(y_1, \rho_1)$, $(y_2, \rho_2) \in Y$ with $y_1 \neq y_2$.

Let $M :(\Omega, \mathcal{F}, \mathbb{P}) \rightarrow (\mathbb{M}^{N - 1}, \mathcal{M}^{N - 1})$ be an m.p.p. and define $T_\varepsilon(\omega)$ as in \eqref{eq:perforations}. For $\omega \in \Omega$ and $(y, \rho) \in M(\omega)$, we define the truncated radius
\[
r(\omega, y) := \min \left\{\frac{1}{2}d_{M(\omega)}(y), 1\right\}.
\]
If $\varepsilon \bar{y}_1$ and $\varepsilon \bar{y}_2$ are the centers of two distinct holes in $T_\varepsilon(\omega)$, then \eqref{eq:nonintersecting_balls} implies
\begin{equation} \label{eq:nonintersecting_scaled_balls}
    B(\varepsilon \bar{y}_1, \varepsilon r(\omega, y_1)) \cap B(\varepsilon \bar{y}_2, \varepsilon r(\omega, y_2)) = \emptyset.
\end{equation}

Given $\varepsilon > 0$, $\omega \in \Omega$, our first objective is to carefully distinguish isolated holes in $T_\varepsilon(\omega)$ from clusters.

\begin{definition} \label{def:isolated_and_cluster_points}
    Let $(y, \rho) \in M(\omega)$ with $\varepsilon \bar{y} \in U^0$. Then $(y, \rho)$ is called \textit{$\varepsilon$-isolated} if it satisfies the following two conditions:
    \begin{enumerate}[label=(\roman*)]
        \item $2\varepsilon^\frac{N - 1}{N - 2} \rho < \varepsilon r(\omega, y)$,
        \item If $(y', \rho') \in M(\omega)$ with $\varepsilon \bar{y}' \in U^0$ and $y \neq y'$, then $B\left(\varepsilon \bar{y}', 2\varepsilon^\frac{N - 1}{N - 2} \rho'\right) \cap B(\varepsilon \bar{y}, \varepsilon r(\omega, y)) = \emptyset$.
    \end{enumerate}
    Otherwise $(y, \rho)$ is called an \textit{$\varepsilon$-cluster} point.
\end{definition}

\begin{figure}[ht] \label{fig:isolated_and_cluster_points}
    \centering
    \begin{tikzpicture}[framed]
        \node[anchor=south west,inner sep=0] (image) at (0,0) {\includegraphics[scale=0.25]{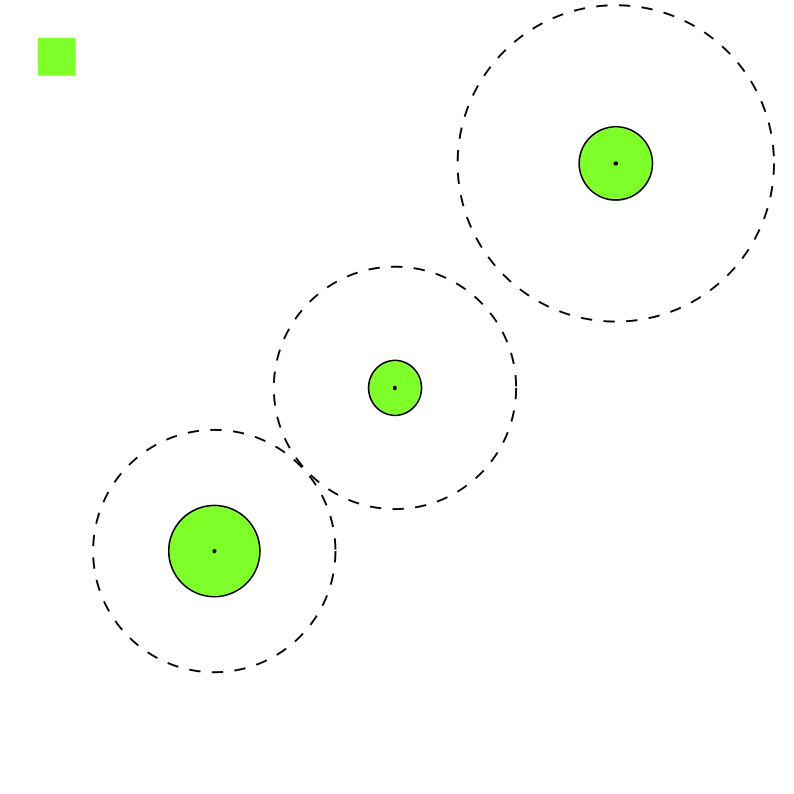}};
        \begin{scope}[x={(image.south east)},y={(image.north west)}]
            \draw (0.2, 0.925) node { \small $:I_\varepsilon(\omega)$};
        \end{scope}
    \end{tikzpicture}
    \begin{tikzpicture}[framed]
        \node[anchor=south west,inner sep=0] (image) at (0,0) {\includegraphics[scale=0.25]{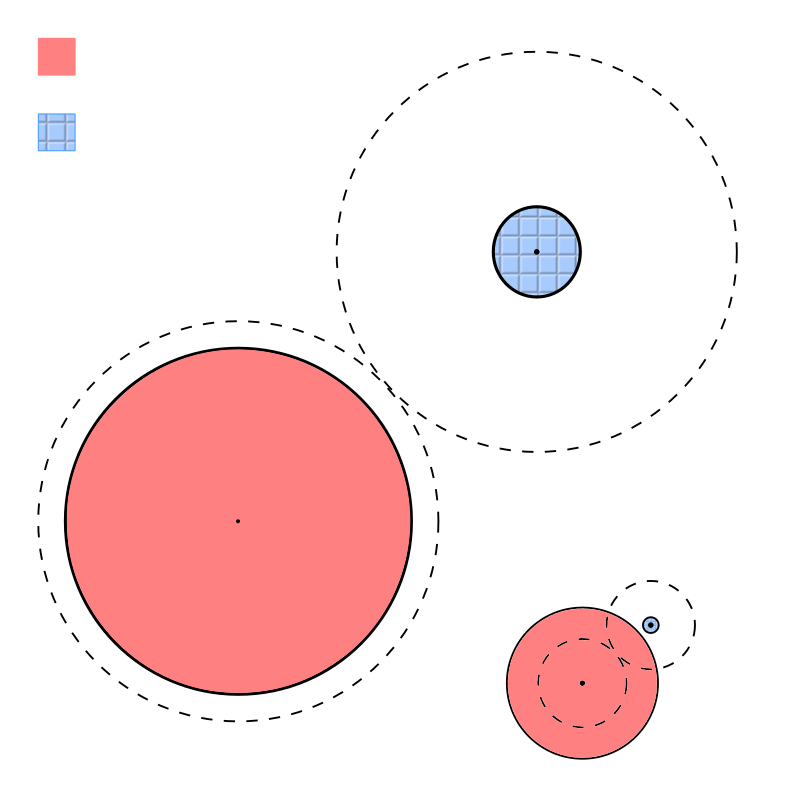}};
        \begin{scope}[x={(image.south east)},y={(image.north west)}]
            \draw (0.3, 0.3) node {\small $\varepsilon \bar{y}$};
            \draw (0.21, 0.93) node { \small $:C_\varepsilon^1(\omega)$};
            \draw (0.21, 0.83) node { \small $:C_\varepsilon^2(\omega)$};
        \end{scope}
    \end{tikzpicture}
    \caption{Illustrations of the balls $B(\varepsilon \bar{y}, \varepsilon^\frac{N - 1}{N - 2} \rho)$ (colored) and $B(\varepsilon \bar{y}, \varepsilon r(\omega, y))$ (dashed) for some points $(y, \rho) \in M(\omega)$. The color of the balls indicate which set their centers belong to. See \eqref{eq:cluster_points_classification} for the definitions of $C_\varepsilon^1(\omega)$ and $C_\varepsilon^2(\omega)$.}
\end{figure}
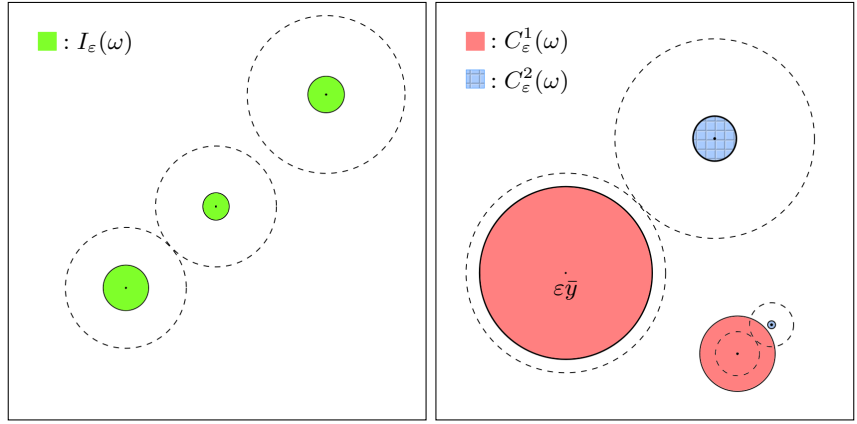

Thus, a point $(y, \rho) \in M(\omega)$ with $\varepsilon \bar{y} \in U^0$ is $\varepsilon$-isolated, if the ball $B(\varepsilon \bar{y}, \varepsilon r(\omega, y))$ separates the hole centered at $\varepsilon \bar{y}$ from the other holes in $T_\varepsilon(\omega)$. We denote by $I_\varepsilon(\omega)$ and $C_\varepsilon(\omega)$ the set of $\varepsilon$-isolated and $\varepsilon$-cluster points, respectively. We also define
\begin{equation} \label{eq:clusters_and_safety_layer}
    \begin{aligned}
        T_\varepsilon^C(\omega) &:= \bigcup_{(y, \rho) \in C_\varepsilon(\omega)} B\left(\varepsilon \bar{y}, \varepsilon^\frac{N - 1}{N - 2}\rho\right) \cap U^0, \\
        S_\varepsilon(\omega) &:= \bigcup_{(y, \rho) \in C_\varepsilon(\omega)} B\left(\varepsilon \bar{y}, 2\varepsilon^\frac{N - 1}{N - 2}\rho\right).
    \end{aligned}
\end{equation}
Hence, $T_\varepsilon^C(\omega)$ is the subset of $T_\varepsilon(\omega)$ that comprises the family of clusters. The set $S_\varepsilon(\omega)$ functions as a layer to separate $T_\varepsilon(\omega) \setminus T_\varepsilon^C(\omega)$ and $T_\varepsilon^C(\omega)$. Indeed, given $\varepsilon > 0$ and $\omega \in \Omega$, it is easily verified that
\begin{equation} \label{eq:separation_layer}
    B(\varepsilon \bar{y}, \varepsilon r(\omega, y)) \cap S_\varepsilon(\omega) = \emptyset \quad \text{for all } (y, \rho) \in I_\varepsilon(\omega).
\end{equation}

Now, we state the main result of this section.

\begin{theorem} \label{thm:vanishing_capacity_and_measure}
    Let $M :(\Omega, \mathcal{F}, \mathbb{P}) \rightarrow (\mathbb{M}^{N - 1}, \mathcal{M}^{N - 1})$ be a stationary marked point process with finite intensity such that
    \[
    \int_0^\infty \rho^{N - 2} \, d\lambda(\rho) < \infty.
    \]
    Then
    \[
    \lim_{\varepsilon \to 0} \cpct (T_\varepsilon^C(\omega), S_\varepsilon(\omega)) = \lim_{\varepsilon \to 0} \mathcal{L}^N(S_\varepsilon(\omega)) = 0 \quad \mathbb{P}\text{-a.s.}
    \]
\end{theorem}

The proof of Theorem \ref{thm:vanishing_capacity_and_measure} relies on the following proposition, whose proof we postpone to the end of this section.

\begin{proposition} \label{pr:cluster_points_ergodic_theorem}
    Let $M :(\Omega, \mathcal{F}, \mathbb{P}) \rightarrow (\mathbb{M}^{N - 1}, \mathcal{M}^{N - 1})$ be a stationary marked point process with finite intensity such that
    \[
    \int_0^\infty \rho^{N - 2} \, d\lambda(\rho) < \infty.
    \]
    Then
    \[
    \lim_{\varepsilon \to 0} \sum_{(y, \rho) \in C_\varepsilon(\omega)} \varepsilon^{N - 1} \rho^{N - 2} = 0 \quad \mathbb{P}\text{-a.s.}
    \]
\end{proposition}

\begin{proof}[Proof of Theorem \ref{thm:vanishing_capacity_and_measure}]
    By the subadditivity and the scaling property of the capacity,
    \begin{align*}
        \cpct (T_\varepsilon^C(\omega), S_\varepsilon(\omega)) &\le \sum_{(y, \rho) \in C_\varepsilon(\omega)} \cpct \left(B\left(\varepsilon \bar{y}, \varepsilon^\frac{N - 1}{N - 2} \rho\right)^0, S_\varepsilon(\omega)\right) \\
        &\le \sum_{(y, \rho) \in C_\varepsilon(\omega)} \cpct \left(B\left(\varepsilon \bar{y}, \varepsilon^\frac{N - 1}{N - 2} \rho\right)^0, B\left(\varepsilon \bar{y}, 2\varepsilon^\frac{N - 1}{N - 2} \rho\right)\right) \\
        &\le C \sum_{(y, \rho) \in C_\varepsilon(\omega)} \left(\varepsilon^\frac{N - 1}{N - 2} \rho\right)^{N - 2} = C \sum_{(y, \rho) \in C_\varepsilon(\omega)} \varepsilon^{N -1} \rho^{N - 2}.
    \end{align*}
    Also
    \begin{multline*}
        \mathcal{L}^N(S_\varepsilon(\omega)) \le C \sum_{(y, \rho) \in C_\varepsilon(\omega)} \left(\varepsilon^\frac{N - 1}{N - 2} \rho\right)^N \\
        = C \sum_{(y, \rho) \in C_\varepsilon(\omega)} \left(\varepsilon^{N - 1} \rho^{N - 2}\right)^\frac{N}{N - 2} \le C \left(\sum_{(y, \rho) \in C_\varepsilon(\omega)} \varepsilon^{N - 1} \rho^{N - 2}\right)^\frac{N}{N - 2}.
    \end{multline*}
    Hence the claim follows from Proposition \ref{pr:cluster_points_ergodic_theorem}.
\end{proof}

Intuitively, clusters are formed by those holes that either have centers that are too close to each other or radii that are too large. Therefore, to prove Proposition \ref{pr:cluster_points_ergodic_theorem}, we focus on such holes. For this, we introduce a concept that will be used extensively in what follows. Let $\delta > 0$. We define a new marked point process $M_\delta : (\Omega, \mathcal{F}, \mathbb{P}) \rightarrow (\mathbb{M}^{N - 1}, \mathcal{M}^{N - 1})$ as follows: given $\omega \in \Omega$, the point $(y, \rho)$ belongs to $M_\delta(\omega)$ if and only if $(y, \rho) \in M(\omega)$ and
\[
\min \left\{\frac{d_{M(\omega)}(y)}{2}, \frac{1}{\rho}\right\} < \delta.
\]
We call $M_\delta$ the \textit{thinned process}. The relevance of thinned processes in our analysis is mainly motivated by the following result which is a corollary of Theorem \ref{thm:ergodic_limit_for_thinned_processes} in the appendix. A further study of thinned processes in a broader context can be found in Sections \ref{ssec:thinning_of_marked_point_processes} and \ref{ssec:proof_of_measurability_of_the_thinning_map}.

\begin{lemma} \label{lm:thinned_limit}
    Let $M :(\Omega, \mathcal{F}, \mathbb{P}) \rightarrow (\mathbb{M}^{N - 1}, \mathcal{M}^{N - 1})$ be a stationary marked point process with finite intensity such that
    \[
    \int_0^\infty \rho^{N - 2} \, d\lambda(\rho) < \infty.
    \]
    Then
    \[
    \lim_{\delta \to 0} \, \lim_{\varepsilon \to 0} \sum \limits_{\substack{(y, \rho) \in M_\delta(\omega) \\ \bar{y} \in \frac{1}{\varepsilon} U^0}} \varepsilon^{N - 1} \rho^{N - 2} = 0 \quad \mathbb{P}\text{-a.s.}
    \]
\end{lemma}

\begin{proof}
    Define $g : \mathbb{R}^+ \rightarrow [0, \infty)$ by $g(\rho) := \rho^{N - 2}$. The function $g$ is locally bounded and $\lambda$-integrable. Hence, Theorem \ref{thm:ergodic_limit_for_thinned_processes} yields
    \[
    \lim_{\delta \to 0} \, \lim_{\varepsilon \to 0} \frac{1}{\mathcal{H}^{N - 1}\left(\frac{1}{\varepsilon} U^0\right)} \sum \limits_{\substack{(y, \rho) \in M_\delta(\omega) \\ \bar{y} \in \frac{1}{\varepsilon} U^0}} \rho^{N - 2} = 0 \quad \mathbb{P}\text{-a.s.}
    \]
    It is easy to see that
    \[
    \sum \limits_{\substack{(y, \rho) \in M_\delta(\omega) \\ \bar{y} \in \frac{1}{\varepsilon} U^0}} \varepsilon^{N - 1} \rho^{N - 2} = \frac{\mathcal{H}^{N - 1}(U^0)}{\mathcal{H}^{N - 1}\left(\frac{1}{\varepsilon} U^0\right)} \sum \limits_{\substack{(y, \rho) \in M_\delta(\omega) \\ \bar{y} \in \frac{1}{\varepsilon} U^0}} \rho^{N - 2}.
    \]
    Thus, the lemma is proved.
\end{proof}

The proof of Proposition \ref{pr:cluster_points_ergodic_theorem} follows from a number of intermediate results. Namely, we will establish
Proposition \ref{pr:cluster_points_ergodic_theorem} in two steps by proving separately that 
\begin{equation} \label{eq:sum_over_cluster_points}
    \lim_{\varepsilon \to 0} \sum_{(y, \rho) \in C_\varepsilon^i(\omega)} \varepsilon^{N - 1} \rho^{N - 2} = 0 \quad \mathbb{P}\text{-a.s.}
\end{equation}
for $i = 1, 2$, where we define
\begin{equation} \label{eq:cluster_points_classification}
    \begin{aligned}
        C_\varepsilon^1(\omega) &:= \left\{(y, \rho) \in C_\varepsilon(\omega) : 2\varepsilon^\frac{N - 1}{N - 2} \rho \ge \varepsilon r(\omega, y) \right\}, \\
        C_\varepsilon^2(\omega) &:= C_\varepsilon(\omega) \setminus C_\varepsilon^1(\omega)
    \end{aligned}
\end{equation}
for all $\varepsilon > 0$ and $\omega \in \Omega$. We note that $C_\varepsilon^1(\omega)$ consists of those points in $C_\varepsilon(\omega)$ that violate condition (i) in Definition \ref{def:isolated_and_cluster_points}, whereas $C_\varepsilon^2(\omega)$ consists of those points that only violate condition (ii) (see Figure 2).

The following result deals with the points in $C_\varepsilon^1(\omega)$.

\begin{lemma} \label{lm:negligible_cluster_1}
    Let $M :(\Omega, \mathcal{F}, \mathbb{P}) \rightarrow (\mathbb{M}^{N - 1}, \mathcal{M}^{N - 1})$ be a stationary marked point process with finite intensity such that
    \[
    \int_0^\infty \rho^{N - 2} \, d\lambda(\rho) < \infty.
    \]
    Then
    \[
    \lim_{\varepsilon \to 0} \sum_{(y, \rho) \in C_\varepsilon^1(\omega)} \varepsilon^{N - 1} \rho^{N - 2} = 0 \quad \mathbb{P}\text{-a.s.}
    \]
\end{lemma}

\begin{proof}
    Let $\varepsilon > 0$ and $\omega \in \Omega$. If $(y, \rho) \in C_\varepsilon^1(\omega)$, then $r(\omega, y)/\rho \le 2\varepsilon^{1/(N - 2)}$ by definition. Consequently, we must either have
    \[
    r(\omega, y) \le \left(2 \varepsilon^\frac{1}{N - 2}\right)^\frac{1}{2} \quad \text{or} \quad \frac{1}{\rho} \le \left(2 \varepsilon^\frac{1}{N - 2}\right)^\frac{1}{2}.
    \]
    If $\varepsilon$ is small enough and $\delta > 0$, then this implies
    \[
    \min \left\{\frac{d_{M(\omega)}(y)}{2}, \frac{1}{\rho}\right\} < \delta,
    \]
    that is, we have $(y, \rho) \in M_\delta(\omega)$. Therefore,
    \[
    \limsup_{\varepsilon \to 0} \sum_{(y, \rho) \in C_\varepsilon^1(\omega)} \varepsilon^{N - 1} \rho^{N - 2} \le \limsup_{\varepsilon \to 0} \sum \limits_{\substack{(y, \rho) \in M_\delta(\omega) \\ \bar{y} \in \frac{1}{\varepsilon} U^0}} \varepsilon^{N - 1} \rho^{N - 2}.
    \]
    By passing to the limit in $\delta$, we obtain
    \[
    \lim_{\varepsilon \to 0} \sum_{(y, \rho) \in C_\varepsilon^1(\omega)} \varepsilon^{N - 1} \rho^{N - 2} = \lim_{\delta \to 0} \, \lim_{\varepsilon \to 0} \sum \limits_{\substack{(y, \rho) \in M_\delta(\omega) \\ \bar{y} \in \frac{1}{\varepsilon} U^0}} \varepsilon^{N - 1} \rho^{N - 2} = 0 \quad \mathbb{P}\text{-a.s.},
    \]
    where the last equality follows from Lemma \ref{lm:thinned_limit}.
\end{proof}

We will show that the analogous result for $C_\varepsilon^2(\omega)$ is a consequence of Lemma \ref{lm:negligible_cluster_1}. The key observation is that the points of $C_\varepsilon^2(\omega)$ are clustered around the points of $C_\varepsilon^1(\omega)$.

\begin{lemma} \label{lm:cluster_within_cluster}
    Let $\varepsilon > 0$ and $\omega \in \Omega$. For all $(y, \rho) \in C_\varepsilon^2(\omega)$, there exists $(y', \rho') \in C_\varepsilon^1(\omega)$ such that
    \begin{equation} \label{eq:absorbed_by_cluster_point}
        B(\varepsilon \bar{y}, \varepsilon r(\omega, y)) \subset B\left(\varepsilon \bar{y}', 6\varepsilon^\frac{N - 1}{N - 2}\rho'\right).
    \end{equation}
\end{lemma}

\begin{proof}
    Assume $(y, \rho) \in C_\varepsilon^2(\omega)$. As $(y, \rho)$ violates condition (ii) in Definition \ref{def:isolated_and_cluster_points}, there exists a point $(y', \rho') \in M(\omega)$ with $\varepsilon \bar{y}' \in U^0$ such that
    \begin{equation} \label{eq:cluster_equation}
        B\left(\varepsilon \bar{y}', 2\varepsilon^\frac{N - 1}{N - 2} \rho'\right) \cap B(\varepsilon \bar{y}, \varepsilon r(\omega, y)) \neq \emptyset \iff \varepsilon |\bar{y} - \bar{y}'| < \varepsilon r(\omega, y) + 2\varepsilon^\frac{N - 1}{N - 2}\rho'.
    \end{equation}
    In fact, the point $(y', \rho')$ belongs to $C_\varepsilon^1(\omega)$. Otherwise, we would have the inclusion
    \[
    B\left(\varepsilon \bar{y}', 2\varepsilon^\frac{N - 1}{N - 2}\rho'\right) \subset B(\varepsilon \bar{y}', \varepsilon r(\omega, y')),
    \]
    which contradicts \eqref{eq:nonintersecting_scaled_balls} by \eqref{eq:cluster_equation}. Since $r(\omega, y) \le |y - y'|/2$, equation \eqref{eq:cluster_equation} implies $\varepsilon |\bar{y} - \bar{y}'| < 4 \varepsilon^\frac{N - 1}{N - 2} \rho'$. Therefore,
    \[
    B(\varepsilon \bar{y}, \varepsilon r(\omega, y)) \subset B\left(\varepsilon \bar{y}', 6\varepsilon^\frac{N - 1}{N - 2}\rho'\right).
    \]
\end{proof}

The following lemma deals with the points in $C_\varepsilon^2(\omega)$.

\begin{lemma} \label{lm:negligible_cluster_2}
    Let $M :(\Omega, \mathcal{F}, \mathbb{P}) \rightarrow (\mathbb{M}^{N - 1}, \mathcal{M}^{N - 1})$ be a stationary marked point process with finite intensity such that
    \[
    \int_0^\infty \rho^{N - 2} \, d\lambda(\rho) < \infty.
    \]
    Then
    \[
    \lim_{\varepsilon \to 0} \sum_{(y, \rho) \in C_\varepsilon^2(\omega)} \varepsilon^{N - 1} \rho^{N - 2} = 0 \quad \mathbb{P}\text{-a.s.}
    \]
\end{lemma}

\begin{proof}
    Thanks to Lemma \ref{lm:thinned_limit}, it is enough to prove
    \[
    \lim_{\varepsilon \to 0} \sum_{(y, \rho) \in C_\varepsilon^2(\omega) \setminus M_\delta(\omega)} \varepsilon^{N - 1} \rho^{N - 2} = 0 \quad \mathbb{P}\text{-a.s.}
    \]
    for all sufficiently small $\delta$. Let $\varepsilon > 0$ and $\omega \in \Omega$. As $\rho \le 1/\delta$ for all $(y, \rho) \in M(\omega) \setminus M_\delta(\omega)$, we observe that
    \begin{equation} \label{eq:bound_by_cardinality}
        \sum_{(y, \rho) \in C_\varepsilon^2(\omega) \setminus M_\delta(\omega)} \varepsilon^{N - 1} \rho^{N - 2} \le \frac{\varepsilon^{N - 1}}{\delta^{N - 2}}\card\left(C_\varepsilon^2(\omega) \setminus M_\delta(\omega)\right).
    \end{equation}
    We estimate this cardinality. Fixing $\omega$ for the moment, let us define
    \[
    A(y', \rho') := \left\{(y, \rho) \in C_\varepsilon^2(\omega) \setminus M_\delta(\omega) : B(\varepsilon \bar{y}, \varepsilon r(\omega, y)) \subset B\left(\varepsilon \bar{y}', 6\varepsilon^\frac{N - 1}{N - 2}\rho'\right)\right\}.
    \]
    It follows from Lemma \ref{lm:cluster_within_cluster} that
    \begin{equation} \label{eq:cardinality_estimate}
        \card\left(C_\varepsilon^2(\omega) \setminus M_\delta(\omega)\right) \le \sum_{(y', \rho') \in C_\varepsilon^1(\omega)} \card(A(y', \rho')).
    \end{equation}
    To estimate $\card(A(y', \rho'))$, we compute the sum of the cross sectional areas of the balls $B(\varepsilon \bar{y}, \varepsilon r(\omega, y))$ for $(y, \rho) \in A(y', \rho')$. Since the balls are disjoint, we get
    \[
    \sum_{(y, \rho) \in A(y', \rho')} \varepsilon^{N - 1} r(\omega, y)^{N - 1} \le 6^{N - 1} \varepsilon^\frac{(N - 1)^2}{N - 2} (\rho')^{N - 1}.
    \]
    Assume $\delta < 1$. As $r(\omega, y) \ge \delta$ for $(y, \rho) \in M(\omega) \setminus M_\delta(\omega)$, and $A(y', \rho') \subset M(\omega) \setminus M_\delta(\omega)$, we conclude that
    \[
    \card(A(y', \rho')) \le \left(\frac{6}{\delta}\right)^{N - 1} \varepsilon^\frac{N - 1}{N - 2} (\rho')^{N - 1}.
    \]
    Combining \eqref{eq:bound_by_cardinality} and \eqref{eq:cardinality_estimate} finally yields
    \begin{equation*}
    \sum_{(y, \rho) \in C_\varepsilon^2(\omega) \setminus M_\delta(\omega)} \varepsilon^{N - 1} \rho^{N - 2} \\
    \le C(\delta) \sum_{(y', \rho') \in C_\varepsilon^1(\omega)} (\varepsilon^{N - 1}(\rho')^{N - 2})^\frac{N - 1}{N - 2} \le C(\delta) \left(\sum_{(y', \rho') \in C_\varepsilon^1(\omega)} \varepsilon^{N - 1}(\rho')^{N - 2}\right)^\frac{N - 1}{N - 2}.
    \end{equation*}
    Thus, we are done by Lemma \ref{lm:negligible_cluster_1}.
\end{proof}

\begin{proof}[Proof of Proposition \ref{pr:cluster_points_ergodic_theorem}]   
    The proof follows as an immediate corollary of Lemma \ref{lm:negligible_cluster_1} and Lemma \ref{lm:negligible_cluster_2}.
\end{proof}

\section{The oscillating test functions} \label{sec:oscillating_test_functions}

In this section, we give an explicit construction of the oscillating test functions $(w_\varepsilon(\omega, \cdot))$ introduced in Theorem \ref{thm:existence_of_oscillating_test_functions}.

\subsection{The cell problem} \label{ssec:cell_problem}

The idea behind the construction of the oscillating test functions is to solve a PDE, the so-called cell problem, around each isolated hole. For $R > 1$, we define
\[
D_R := B(0, R)^+ \cup B(0, R)^- \cup B(0, 1)^0.
\]
Let $\eta_R$ be the unique solution of the boundary value problem
\begin{equation} \label{eq:cell_problem}
    \begin{cases}
        \begin{aligned}
            -\Delta \eta_R &= \phantom{\pm}0 \quad \text{in } D_R, \\
            \eta_R &= \pm 1 \quad \text{on } (\partial B(0, R))^\pm, \\
            \nabla \eta_R \cdot \nu &= \phantom{\pm}0 \quad \text{on } B(0, R)^0 \setminus B(0, 1)^0.
        \end{aligned}
    \end{cases}
\end{equation}
We would like to remark that, in this problem, the domain where the cell problem is solved depends on a positive parameter $R$. The reason for this dependence is that, while the radius of the perforations is of the order $\varepsilon^{\frac{N-1}{N-2}}$, the typical distance between the centers of two isolated holes is of the order $\varepsilon$.

The weak formulation of \eqref{eq:cell_problem} is given by
\begin{equation} \label{eq:weak_cell_problem}
    \begin{cases}
        \begin{aligned}
            &\eta_R \in H^1(D_R), \, \eta_R = \pm 1 \text{ on } (\partial B(0, R))^\pm, \\
            &\int_{D_R} \nabla \eta_R \nabla \varphi \, dx = 0 \quad \text{for all } \varphi \in H^1(D_R), \, \varphi = 0 \text{ on } \partial B(0, R).
        \end{aligned}
    \end{cases}
\end{equation}
We prove some simple properties of $\eta_R$.

\begin{proposition}
    Let $R > 1$. Then 
    \begin{gather} \label{eq:odd_function}
        \eta_R(x', x_N) = -\eta_R(x', -x_N) \quad \text{for all } x = (x', x_N) \in D_R, \\ \label{eq:maximum_principle}
        -1 \le \eta_R \le 0 \quad \mathcal{L}^N\text{-a.e. in } B(0, R)^-,\quad 0 \le \eta_R \le 1 \quad \mathcal{L}^N\text{-a.e. in } B(0, R)^+.
    \end{gather}
\end{proposition}

\begin{proof}
    It can be easily checked that the function $(x', x_N) \mapsto -\eta_R(x', -x_N)$ defines a solution of \eqref{eq:cell_problem}, where $x = (x', x_N) \in D_R$. Equation \eqref{eq:odd_function} then follows from the uniqueness of the solution for \eqref{eq:cell_problem}. We note that \eqref{eq:odd_function} implies that $\eta_R = 0$ in $B(0, 1)^0$. Now, we define a new function $\widetilde{\eta}_R$ in $D_R$ by
    \[
    \widetilde{\eta}_R(x) :=
    \begin{cases}
        \max \{\eta_R(x), 0\}, \quad &\text{if } x_N \ge 0 \\
        \min \{\eta_R(x), 0\}, \quad &\text{if } x_N < 0
    \end{cases}
    \]
    As $\eta_R = 0$ in $B(0, 1)^0$, it is not difficult to see that $\widetilde{\eta}_R \in H^1(D_R)$ and that $\eta_R - \widetilde{\eta}_R$ is an admissible test function in \eqref{eq:weak_cell_problem}. Hence, we obtain
    \[
    0 = \int_{D_R} \nabla \eta_R \nabla (\eta_R - \widetilde{\eta}_R) \, dx = \int_{\{\eta_R < 0\}^+ \cup \{\eta_R > 0\}^-} |\nabla \eta_R|^2 \, dx.
    \]
    As a result, we get that necessarily $\eta_R \le 0 $ $\mathcal{L}^N$-a.e. in $B(0, R)^-$ and $0 \le \eta_R$ $\mathcal{L}^N$-a.e. in $B(0, R)^+$. The other two bounds in \eqref{eq:maximum_principle} can be proved analogously, this time testing \eqref{eq:weak_cell_problem} with $\max \{\eta_R - 1, 0\}$ and $\min \{\eta_R + 1, 0\}$, respectively.
\end{proof}

\begin{proposition}
    Let $R > 1$. Then $1 - |\eta_R| \in H_0^1(B(0, R))$ and it is a weak solution of the capacity problem
    \begin{equation} \label{eq:capacity_problem}
        \begin{cases}
            \begin{aligned}
                -\Delta v &= 0 \quad \text{in } B(0, R), \\
                v &= 0 \quad \text{on } \partial B(0, R), \\
                v &= 1 \quad \text{on } \overline{B(0, 1)}^0.
            \end{aligned}
        \end{cases}
    \end{equation}
\end{proposition}

\begin{proof}
    It follows from \eqref{eq:odd_function} that $|\eta_R|$ belongs to $H^1(B(0, R))$ and that $\eta_R = 0$ in $B(0, 1)^0$. We show that $|\eta_R|$ is a harmonic function in $B(0, R) \setminus \overline{B(0, 1)}^0$. Let $\varphi \in H_0^1(B(0, R) \setminus \overline{B(0, 1)}^0)$. We define $\varphi_1$ and $\varphi_2$ by an ``odd extension" of $\varphi^+$ and $\varphi^-$. In other words,
    \begin{gather*}
        \varphi_1(x', x_N) :=
        \begin{cases}
                \varphi^+(x', x_N),\phantom{--} \quad \text{if } x_N \ge 0,\\
                -\varphi^+(x', -x_N), \quad \text{if } x_N \le 0.
        \end{cases}
        \; \;
        \varphi_2(x', x_N) :=
        \begin{cases}
                -\varphi^-(x', -x_N), \quad \text{if } x_N \ge 0,\\
                \varphi^-(x', x_N),\phantom{--} \quad \text{if } x_N \le 0.
        \end{cases}
    \end{gather*}
    for all $x \in D_R$.
    Then $\varphi_1$ and $\varphi_2$ are both admissible test functions in \eqref{eq:weak_cell_problem}. Hence,
    \[
    \int_{D_R} \nabla \eta_R \nabla \varphi_1 \, dx = \int_{D_R} \nabla \eta_R \nabla \varphi_2 \, dx = 0.
    \]
    On the other hand, equation \eqref{eq:odd_function} yields
    \begin{gather*}
    \int_{B(0, R)^+} \nabla \eta_R \nabla \varphi \, dx = \frac{1}{2} \int_{D_R} \nabla \eta_R \nabla \varphi_1 \, dx = 0, \\
    \int_{B(0, R)^-} \nabla \eta_R \nabla \varphi \, dx = \frac{1}{2} \int_{D_R} \nabla \eta_R \nabla \varphi_2 \, dx = 0.
    \end{gather*}
    Thus, \eqref{eq:maximum_principle} implies
    \[
    \int_{B(0, R)} \nabla |\eta_R| \nabla \varphi \, dx = \int_{B(0, R)^+} \nabla \eta_R \nabla \varphi \, dx - \int_{B(0, R)^-} \nabla \eta_R \nabla \varphi \, dx = 0.
    \]
    With this, we conclude that $|\eta_R|$ is harmonic. Since $1 - |\eta_R|$ clearly satisfies the boundary conditions in \eqref{eq:capacity_problem}, we are done.
\end{proof}

\begin{proposition} \label{pr:capacitary_potential_estimates}
    Let $R > 1$. Then
    \begin{gather} \label{eq:gradient_gives_capacity}
        \int_{B(0, R)^\pm} |\nabla \eta_R|^2 \, dx = \frac{1}{2} \cpct (B(0, 1)^0, B(0, R)), \\ \label{eq:distance_to_signed_1}
        \int_{B(0, R)^\pm} |\pm 1 - \eta_R|^2 \, dx \le CR^2 \int_{B(0, R)} |\nabla \eta_R|^2 \, dx,
    \end{gather}
    where $C$ is independent of $R$.
\end{proposition}

\begin{proof}
    Set $v_R := 1 - |\eta_R|$. As $v_R$ is the solution of \eqref{eq:capacity_problem}, it is the minimizer of the Dirichlet energy among all functions in $H^1(B(0, R) \setminus \overline{B(0,1)}^0)$ sharing the same boundary values. Hence, we get
    \[
    \int_{D_R} |\nabla \eta_R|^2 \, dx = \int_{B(0, R)} |\nabla v_R|^2 \, dx = \cpct (B(0, 1)^0, B(0, R))
    \]
    from the definition of capacity. Moreover, it is an easy consequence of \eqref{eq:odd_function} that
    \[
    \int_{B(0, R)^+} |\nabla \eta_R|^2 \, dx = \int_{B(0, R)^-} |\nabla \eta_R|^2 \, dx = \frac{1}{2} \cpct (B(0, 1)^0, B(0, R)).
    \]
    Finally, applying Poincaré's inequality to $v_R$ and using \eqref{eq:maximum_principle} gives
    \begin{multline*}
    \int_{B(0, R)^+} |1 - \eta_R|^2 \, dx + \int_{B(0, R)^-} |1 + \eta_R|^2 \, dx \\
    = \int_{B(0, R)} |v_R|^2 \, dx \le C R^2 \int_{B(0, R)} |\nabla v_R|^2 \, dx = C R^2 \int_{B(0, R)} |\nabla \eta_R|^2 \, dx
    \end{multline*}
    for all $R > 1$.
\end{proof}

\begin{corollary} \label{co:minimization_problem}
    Let $R > 1$. Then
    \[
    \inf \left\{\int_{D_R} |\nabla v|^2 \, dx : v \in H^1(D_R), \, v = \pm 1 \text{ on } (\partial B(0, R))^\pm\right\} = \cpct \left(B(0, 1)^0, B(0, R)\right).
    \]
\end{corollary}

\begin{proof}
    The result is a direct consequence of Proposition \ref{pr:capacitary_potential_estimates} and the fact that $\eta_R$ is the unique solution of the corresponding Euler-Lagrange equation.
\end{proof}

\subsection{Construction of the random oscillating test functions} \label{ssec:construction_of_oscillations}

We are now in position to define the oscillating test functions. For the remainder of this section, we fix a stationary m.p.p. $M :(\Omega, \mathcal{F}, \mathbb{P}) \rightarrow (\mathbb{M}^{N - 1}, \mathcal{M}^{N - 1})$ with finite intensity such that
\[
\int_0^\infty \rho^{N - 2} \, d\lambda(\rho) < \infty.
\]
Let $\varepsilon > 0$ and $\omega \in \Omega$ be fixed. Recall that $I_\varepsilon(\omega)$ and $C_\varepsilon(\omega)$, defined in Section \ref{sec:analysis_of_perforations}, are the set of $\varepsilon$-isolated and the set of $\varepsilon$-cluster points, respectively. Recall also that $T_\varepsilon^C(\omega)$ and $S_\varepsilon(\omega)$ are the family of clusters and its separation layer as defined in \eqref{eq:clusters_and_safety_layer}. Given an $\varepsilon$-isolated point $(y, \rho) \in I_\varepsilon(\omega)$, we define $w_\varepsilon(\omega, \cdot)$ in $B(\varepsilon \bar{y}, \varepsilon r(\omega, y))$ by
\begin{equation} \label{eq:local_definition}
    w_\varepsilon(\omega, x) := \eta_R\left(\frac{x - \varepsilon \bar{y}}{\varepsilon^\frac{N - 1}{N - 2}\rho}\right), \quad R := \frac{r(\omega, y)}{\varepsilon^\frac{1}{N - 2}\rho}.
\end{equation}
This defines $w_\varepsilon(\omega, \cdot)$ uniquely in
\[
\bigcup_{(y, \rho) \in I_\varepsilon(\omega)} B(\varepsilon \bar{y}, \varepsilon r(\omega, y)),
\]
as the balls are pairwise disjoint. Next, we define $w_\varepsilon(\omega, \cdot)$ in $S_\varepsilon(\omega)$. Choose a function $\zeta_\varepsilon \in C_c^\infty(S_\varepsilon(\omega); [0, 1])$ such that $\zeta_\varepsilon = 1$ in $T_\varepsilon^C(\omega)$ and such that
\[
\int_{S_\varepsilon(\omega)} |\nabla \zeta_\varepsilon|^2 \, dx \le 2 \cpct \left(T_\varepsilon^C(\omega), S_\varepsilon(\omega)\right).
\]
Set $w_\varepsilon(\omega, \cdot) := 1 - \zeta_\varepsilon$ in $S_\varepsilon(\omega)^+$ and $w_\varepsilon(\omega, \cdot) := \zeta_\varepsilon - 1$ in $S_\varepsilon(\omega)^-$. Then $w_\varepsilon(\omega, \cdot) = \pm 1$ on $(\partial S_\varepsilon(\omega))^\pm$ and
\begin{equation} \label{eq:vanishing_gradient}
    \int_{S_\varepsilon(\omega)} |\nabla w_\varepsilon(\omega, x)|^2 \, dx \le 2 \cpct \left(T_\varepsilon^C(\omega), S_\varepsilon(\omega)\right).
\end{equation}
Note that the definition of $w_\varepsilon(\omega, \cdot)$ in $S_\varepsilon(\omega)$ does not interfere with the definition in \eqref{eq:local_definition} due to \eqref{eq:separation_layer}. We complete the definition of $w_\varepsilon(\omega, \cdot)$ by setting
\[
w_\varepsilon(\omega, x) :=
\begin{cases}
    \phantom{-}1, \quad \text{if } x_N > 0, \\
    -1, \quad \text{if } x_N < 0
\end{cases}
\]
for
\[
x \in \mathbb{R}^N \setminus \left(\Sigma \cup S_\varepsilon(\omega) \cup \bigcup_{(y, \rho) \in I_\varepsilon(\omega)} B(\varepsilon \bar{y}, \varepsilon r(\omega, y))\right).
\]
With this definition, we can easily verify that $w_\varepsilon(\omega, \cdot) \in H^1((\mathbb{R}^N)^+ \cup (\mathbb{R}^N)^- \cup T_\varepsilon(\omega))$ and that $-1 \le w_\varepsilon \le 1$ $\mathcal{L}^N$-a.e. in $\mathbb{R}^N$.

Having defined the oscillating test functions, our next goal is to prove that $w_\varepsilon(\omega, \cdot) \rightharpoonup \pm 1$ in $H^1(U^\pm)$ $\mathbb{P}$-a.s. as $\varepsilon \to 0$. We start with the computation of the $L^2$-norm of the gradients, for which we rely heavily on the Ergodic Theorem for marked point processes given by Theorem \ref{thm:ergodic_theorem_second_version}. We introduce the shorthand notation
\begin{equation} \label{eq:shorthand_capacity}
    \sigma(R) := \cpct \left(B(0, 1)^0, B(0, R)\right), \quad \sigma(\infty) := \cpct (B(0, 1)^0, \mathbb{R}^N).
\end{equation}
Note that $\sigma$ is monotonically decreasing in $R$ and that $\lim_{R \to \infty} \sigma(R) = \inf_{R > 0} \sigma(R) = \sigma(\infty)$. Recall also the definition of $\gamma$ in \eqref{eq:effective_factor}:
\[
\gamma(\omega) = \frac{1}{4} \sigma(\infty) \int_0^\infty \rho^{N - 2} \, d\xi(\omega, \rho)
\]
for all $\omega \in \Omega$.

\begin{lemma} \label{lm:gradient_computation}
    Let $O \subset \Sigma$ be a Borel measurable set in $\mathbb{R}^N$ whose boundary in the relative topology of $\Sigma$ is $\mathcal{H}^{N - 1}$-negligible. Then
    \[
    \lim_{\varepsilon \to 0} \sum \limits_{\substack{(y, \rho) \in I_\varepsilon(\omega) \\ \bar{y} \in \frac{1}{\varepsilon} O}} \int_{B(\varepsilon \bar{y}, \varepsilon r(\omega, y))^\pm} |\nabla w_\varepsilon(\omega, x)|^2 \, dx = 2\gamma(\omega) \mathcal{H}^{N - 1}(U^0 \cap O) \quad \mathbb{P}\text{-a.s.}
    \]
\end{lemma}

\begin{proof}
    Without loss of generality, we prove the claim only in the upper half-space. Let $\varepsilon > 0$ and $\omega \in \Omega$. If $(y, \rho) \in I_\varepsilon(\omega)$, then after a change of variables, it follows from \eqref{eq:local_definition} and Proposition \ref{pr:capacitary_potential_estimates} that
    \begin{equation} \label{eq:scaled_capacitary_estimate}
    \int_{B(\varepsilon \bar{y}, \varepsilon r(\omega, y))^+} |\nabla w_\varepsilon(\omega, x)|^2 \, dx = \frac{1}{2} \varepsilon^{N - 1} \rho^{N - 2} \sigma\left(\frac{r(\omega, y)}{\varepsilon^\frac{1}{N - 2}\rho}\right).
    \end{equation}
    Using the monotonicity of $\sigma$, we obtain the simple lower bound
    \begin{multline} \label{eq:lower_bound_for_gradient}
        \sum \limits_{\substack{(y, \rho) \in I_\varepsilon(\omega) \\ \bar{y} \in \frac{1}{\varepsilon} O}} \int_{B(\varepsilon \bar{y}, \varepsilon r(\omega, y))^+} |\nabla w_\varepsilon(\omega, x)|^2 \, dx > \frac{1}{2} \sigma(\infty) \sum \limits_{\substack{(y, \rho) \in I_\varepsilon(\omega) \\ \bar{y} \in \frac{1}{\varepsilon} O}} \varepsilon^{N - 1} \rho^{N - 2} \\
        \ge \frac{1}{2} \sigma(\infty) \sum \limits_{\substack{(y, \rho) \in M(\omega) \\ \bar{y} \in \frac{1}{\varepsilon} (U^0 \cap O)}} \varepsilon^{N - 1} \rho^{N - 2} - \frac{1}{2} \sigma(\infty) \sum \limits_{\substack{(y, \rho) \in C_\varepsilon(\omega) \\ \bar{y} \in \frac{1}{\varepsilon} U^0}} \varepsilon^{N - 1} \rho^{N - 2}.
    \end{multline}
    An application of Proposition \ref{pr:cluster_points_ergodic_theorem} gives
    \[
    \lim_{\varepsilon \to 0} \sum \limits_{\substack{(y, \rho) \in C_\varepsilon(\omega) \\ \bar{y} \in \frac{1}{\varepsilon} U^0}} \varepsilon^{N - 1} \rho^{N - 2} = 0 \quad \mathbb{P}\text{-a.s.}
    \]
    On the other hand, applying Theorem \ref{thm:ergodic_theorem_second_version} with $g(\rho) = \rho^{N - 2}$, where we use assumption \eqref{eq:moment_condition}, yields
    \[
    \lim_{\varepsilon \to 0} \sum \limits_{\substack{(y, \rho) \in M(\omega) \\ \bar{y} \in \frac{1}{\varepsilon} (U^0 \cap O)}} \varepsilon^{N - 1} \rho^{N - 2} = \mathcal{H}^{N - 1}(U^0 \cap O) \int_0^\infty \rho^{N - 2} \, d\xi(\omega, \rho) \quad \mathbb{P}\text{-a.s.}
    \]
    Consequently, equation \eqref{eq:lower_bound_for_gradient} delivers the asymptotic lower bound
    \[
    \liminf_{\varepsilon \to 0} \sum \limits_{\substack{(y, \rho) \in I_\varepsilon(\omega) \\ \bar{y} \in \frac{1}{\varepsilon} O}} \int_{B(\varepsilon \bar{y}, \varepsilon r(\omega, y))^+} |\nabla w_\varepsilon(\omega, x)|^2 \, dx \ge 2 \gamma(\omega) \mathcal{H}^{N - 1}(U^0 \cap O) \quad \mathbb{P}\text{-a.s.}
    \]
    Let $\delta \in (0, 1)$. To obtain an upper bound, we initially focus on points in $I_\varepsilon(\omega) \setminus M_\delta(\omega)$, where $M_\delta$ is the thinned process. Let $R > 0$. For all $(y, \rho) \in M(\omega) \setminus M_\delta(\omega)$, we have
    \[
    \frac{r(\omega, y)}{\varepsilon^\frac{1}{N - 2}\rho} > \frac{\delta^2}{\varepsilon^\frac{1}{N - 2}} > R \quad \text{if } \varepsilon < \left(\frac{\delta^2}{R}\right)^{N - 2}.
    \]
    Thus, we have
    \begin{multline*}
        \sum \limits_{\substack{(y, \rho) \in I_\varepsilon(\omega) \setminus M_\delta(\omega)\\ \bar{y} \in \frac{1}{\varepsilon} O}} \int_{B(\varepsilon \bar{y}, \varepsilon r(\omega, y))^+} |\nabla w_\varepsilon(\omega, x)|^2 \, dx \\
        < \sum \limits_{\substack{(y, \rho) \in I_\varepsilon(\omega) \setminus M_\delta(\omega) \\ \bar{y} \in \frac{1}{\varepsilon} O}} \frac{1}{2} \sigma(R) \varepsilon^{N - 1} \rho^{N - 2} \le \sum \limits_{\substack{(y, \rho) \in M(\omega) \\ \bar{y} \in \frac{1}{\varepsilon} (U^0 \cap O)}} \frac{1}{2} \sigma(R) \varepsilon^{N - 1} \rho^{N - 2}
    \end{multline*}
    for small enough $\varepsilon$. Another application of Theorem \ref{thm:ergodic_theorem_second_version} yields
    \[
    \limsup_{\varepsilon \to 0} \sum \limits_{\substack{(y, \rho) \in I_\varepsilon(\omega) \setminus M_\delta(\omega)\\ \bar{y} \in \frac{1}{\varepsilon} O}} \int_{B(\varepsilon \bar{y}, \varepsilon r(\omega, y))^+} |\nabla w_\varepsilon(\omega, x)|^2 \, dx \le \frac{\sigma(R)}{\sigma(\infty)}2 \gamma(\omega) \mathcal{H}^{N - 1}(U^0 \cap O) \quad \mathbb{P}\text{-a.s.}
    \]
    As the choice of $R$ is arbitrary and $\sigma(R) \to \sigma(\infty)$ as $R \to \infty$, we deduce that
    \begin{equation} \label{eq:gradient_bound_regular_points}
        \limsup_{\varepsilon \to 0} \sum \limits_{\substack{(y, \rho) \in I_\varepsilon(\omega) \setminus M_\delta(\omega)\\ \bar{y} \in \frac{1}{\varepsilon} O}} \int_{B(\varepsilon \bar{y}, \varepsilon r(\omega, y))^+} |\nabla w_\varepsilon(\omega, x)|^2 \, dx \le 2 \gamma(\omega) \mathcal{H}^{N - 1}(U^0 \cap O) \quad \mathbb{P}\text{-a.s.}
    \end{equation}
    for all $\delta > 0$. To conclude the proof, we need to bound \eqref{eq:scaled_capacitary_estimate} for points in $I_\varepsilon(\omega) \cap M_\delta(\omega)$. By the first property of $\varepsilon$-isolated points in Definition \ref{def:isolated_and_cluster_points}, we know that
    \begin{equation} \label{eq:safe_lower_bound}
        \frac{r(\omega, y)}{\varepsilon^\frac{1}{N - 2}\rho} > 2 \quad \text{for all } (y, \rho) \in I_\varepsilon(\omega).
    \end{equation}
    Therefore,
    \begin{equation} \label{eq:ignored_points}
        \sum \limits_{\substack{(y, \rho) \in I_\varepsilon(\omega) \cap M_\delta(\omega) \\ \bar{y} \in \frac{1}{\varepsilon} O}} \int_{B(\varepsilon \bar{y}, \varepsilon r(\omega, y))^+} |\nabla w_\varepsilon(\omega, x)|^2 \, dx < \sum \limits_{\substack{(y, \rho) \in I_\varepsilon(\omega) \cap M_\delta(\omega)\\ \bar{y} \in \frac{1}{\varepsilon} O}} \frac{1}{2} \sigma(2) \varepsilon^{N - 1} \rho^{N - 2}.
    \end{equation}
    It follows from Lemma \ref{lm:thinned_limit} that
    \begin{equation} \label{eq:gradient_bound_thin_points}
    \lim_{\delta \to 0} \lim_{\varepsilon \to 0} \sum \limits_{\substack{(y, \rho) \in I_\varepsilon(\omega) \cap M_\delta(\omega)\\ \bar{y} \in \frac{1}{\varepsilon} O}} \varepsilon^{N - 1} \rho^{N - 2} \le \lim_{\delta \to 0} \lim_{\varepsilon \to 0} \sum \limits_{\substack{(y, \rho) \in M_\delta(\omega)\\ \bar{y} \in \frac{1}{\varepsilon} U^0}} \varepsilon^{N - 1} \rho^{N - 2}  = 0 \quad \mathbb{P}\text{-a.s.}
    \end{equation}
    Finally, by combining \eqref{eq:gradient_bound_regular_points}, \eqref{eq:ignored_points} and \eqref{eq:gradient_bound_thin_points} we are able to get
    \[
    \limsup_{\varepsilon \to 0} \sum \limits_{\substack{(y, \rho) \in I_\varepsilon(\omega) \\ \bar{y} \in \frac{1}{\varepsilon} O}} \int_{B(\varepsilon \bar{y}, \varepsilon r(\omega, y))^+} |\nabla w_\varepsilon(\omega, x)|^2 \, dx \le 2 \gamma(\omega) \mathcal{H}^{N - 1}(U^0 \cap O) \quad \mathbb{P}\text{-a.s.}
    \]
    This concludes the proof.    
\end{proof}

Using the previous lemma, we can localize the limit of the $L^2$-norm of the gradients.

\begin{proposition} \label{pr:local_gradient_limit}
    For $\mathbb{P}$-a.e. $\omega \in \Omega$ and all $\varphi^\pm \in C^0(\overline{U^\pm})$, we have
    \begin{equation} \label{eq:local_gradient_limit}
        \lim_{\varepsilon \to 0} \int_{U^\pm} |\nabla w_\varepsilon(\omega, x)|^2 \varphi^\pm(x) \, dx = 2\gamma(\omega) \int_{U^0} \varphi^\pm(x) \, d\mathcal{H}^{N - 1}(x) \quad \mathbb{P}\text{-a.s.}
    \end{equation}
    In particular,
    \[
    \lim_{\varepsilon \to 0} \int_{U^\pm} |\nabla w_\varepsilon(\omega, x)|^2 \, dx = 2\gamma(\omega) \mathcal{H}^{N - 1}(U^0) \quad \mathbb{P}\text{-a.s.}
    \]
\end{proposition}

\begin{proof}
    As in the proof of the previous lemma, we prove the claim only in the upper half-space. We start by showing that
    \begin{equation} \label{eq:limit_in_cubes}
        \lim_{\varepsilon \to 0} \int_{Q^+} |\nabla w_\varepsilon(\omega, x)|^2 \, dx = 2\gamma(\omega) \mathcal{H}^{N - 1}(U^0 \cap Q) \quad \mathbb{P}\text{-a.s.}
    \end{equation}
    for any closed cube $Q \subset \mathbb{R}^N$ that either has its center on the hyperplane $\Sigma$ or is disjoint from $\Sigma$. Let $\varepsilon > 0$ and $\omega \in \Omega$. As $\nabla w_\varepsilon(\omega, \cdot)$ is nonzero only around the holes, we can write
    \begin{equation} \label{eq:splitting_over_the_balls}
        \int_{Q^+} |\nabla w_\varepsilon(\omega, x)|^2 \, dx = \int_{Q^+ \cap S_\varepsilon(\omega)} |\nabla w_\varepsilon(\omega, x)|^2 \, dx \\
        + \sum_{(y, \rho) \in I_\varepsilon(\omega)} \int_{Q^+ \cap B(\varepsilon \bar{y}, \varepsilon r(\omega, y))} |\nabla w_\varepsilon(\omega, x)|^2 \, dx.
    \end{equation}
    We estimate the first integral on the right-hand side by using \eqref{eq:vanishing_gradient}
    \[
    \int_{Q^+ \cap S_\varepsilon(\omega)} |\nabla w_\varepsilon(\omega, x)|^2 \, dx \le 2 \cpct \left(T_\varepsilon^C(\omega), S_\varepsilon(\omega)\right).
    \]
    As a result of Theorem \ref{thm:vanishing_capacity_and_measure}, we see that
    \[
    \lim_{\varepsilon \to 0} \int_{Q^+ \cap S_\varepsilon(\omega)} |\nabla w_\varepsilon(\omega, x)|^2 \, dx = 0 \quad \mathbb{P}\text{-a.s.}
    \]
    If $Q$ is disjoint from $\Sigma$, then the sum on the right-hand side of \eqref{eq:splitting_over_the_balls} is zero for small enough $\varepsilon$, as $r(\omega, y) \le 1$ for all $(y, \rho) \in M(\omega)$. In this case, the limit in \eqref{eq:limit_in_cubes} is proved since both sides equal zero.
    
    Let us now assume that $Q$ has its center on $\Sigma$. To deal with the second term on the right-hand side of \eqref{eq:splitting_over_the_balls}, we choose closed cubes $Q_a$, $Q_b$ having the same center as $Q$ such that $Q_a \subsetneq Q \subsetneq Q_b$. If $Q^+ \cap B(\varepsilon \bar{y}, \varepsilon r(\omega, y)) \neq \emptyset$ for some $(y, \rho) \in I_\varepsilon(\omega)$, then $\varepsilon \bar{y} \in Q_b^0$ for small enough $\varepsilon$. Similarly, if $\varepsilon \bar{y} \in Q_a^0$, then $B(\varepsilon \bar{y}, \varepsilon r(\omega, y))^+ \subset Q^+$ for small enough $\varepsilon$. This allows us to write the inequalities
    \begin{equation} \label{eq:cube_sandwich}
        \begin{aligned}
            \sum \limits_{\substack{(y, \rho) \in I_\varepsilon(\omega) \\ \bar{y} \in \frac{1}{\varepsilon} Q_a^0}} \int_{B(\varepsilon \bar{y}, \varepsilon r(\omega, y))^+} |\nabla w_\varepsilon(\omega, x)|^2 \, dx &\le \sum_{(y, \rho) \in I_\varepsilon(\omega)} \int_{Q^+ \cap B(\varepsilon \bar{y}, \varepsilon r(\omega, y))} |\nabla w_\varepsilon(\omega, x)|^2 \, dx \\
            &\le \sum \limits_{\substack{(y, \rho) \in I_\varepsilon(\omega) \\ \bar{y} \in \frac{1}{\varepsilon} Q_b^0}} \int_{B(\varepsilon \bar{y}, \varepsilon r(\omega, y))^+} |\nabla w_\varepsilon(\omega, x)|^2 \, dx.
        \end{aligned}
    \end{equation}
    for small enough $\varepsilon$. Applying Lemma \ref{lm:gradient_computation} immediately gives
    \begin{multline*}
        2 \gamma(\omega) \mathcal{H}^{N - 1}(U^0 \cap Q_a) \le \liminf_{\varepsilon \to 0} \sum_{(y, \rho) \in I_\varepsilon(\omega)} \int_{Q^+ \cap B(\varepsilon \bar{y}, \varepsilon r(\omega, y))} |\nabla w_\varepsilon(\omega, x)|^2 \, dx\\
        \le \limsup_{\varepsilon \to 0} \sum_{(y, \rho) \in I_\varepsilon(\omega)} \int_{Q^+ \cap B(\varepsilon \bar{y}, \varepsilon r(\omega, y))} |\nabla w_\varepsilon(\omega, x)|^2 \, dx \le 2 \gamma(\omega) \mathcal{H}^{N - 1}(U^0 \cap Q_b) \quad \mathbb{P}\text{-a.s.}
    \end{multline*}
    Since $\mathcal{H}^{N - 1}(U^0 \cap Q_a)$ and $\mathcal{H}^{N - 1}(U^0 \cap Q_b)$ can be made arbitrarily close to $\mathcal{H}^{N - 1}(U^0 \cap Q)$, we conclude the proof of \eqref{eq:limit_in_cubes}.

    Let $\varphi^+ \in C^0(\overline{U^+})$. In order to prove \eqref{eq:local_gradient_limit} we approximate $U^+$ by cubes. Let $Q_1^k, \dots, Q_{m_k}^k$ and $P_1^k, \dots, P_{n_k}^k$ denote those cubes of side length $1/(2k)$ centered at points of $\frac{1}{k}\mathbb{Z}^N$ such that $(Q_j^k)^+ \subset U^+$ and $P_j^k \cap (\partial U)^+ \neq \emptyset$ for all $j$. We claim that
    \begin{equation} \label{eq:divide_and_conquer}
        \lim_{k \to \infty} \lim_{\varepsilon \to 0} \left|\int_{U^+} |\nabla w_\varepsilon(\omega, x)|^2 \varphi^+(x) \, dx - \sum_{j = 1}^{m_k} \int_{(Q_j^k)^+} |\nabla w_\varepsilon(\omega, x)|^2 \varphi^+(x) \, dx \right| = 0 \quad \mathbb{P}\text{-a.s.}
    \end{equation}
    The difference is clearly bounded by
    \[
    \|\varphi^+\|_{C^0(\overbar{U^+})} \sum_{j = 1}^{n_k} \int_{(P_j^k)^+} |\nabla w_\varepsilon(\omega, x)|^2 \, dx.
    \]
    Then, \eqref{eq:limit_in_cubes} immediately yields
    \begin{multline*}
        \lim_{k \to \infty} \lim_{\varepsilon \to 0} \left|\int_{U^+} |\nabla w_\varepsilon(\omega, x)|^2 \varphi^+(x) \, dx - \sum_{j = 1}^{m_k} \int_{(Q_j^k)^+} |\nabla w_\varepsilon(\omega, x)|^2 \varphi^+(x) \, dx \right| \\
        \le \lim_{k \to \infty} 2\gamma(\omega) \|\varphi^+\|_{C^0(\overbar{U^+})} \mathcal{H}^{N - 1}\left(U^0 \cap \bigcup_{j = 1}^{n_k} P_j^k \right) = 0 \quad \mathbb{P}\text{-a.s.}
    \end{multline*}
    Next, we evaluate the limit of the second term in \eqref{eq:divide_and_conquer}. We shall do this by approximating $\varphi^+$ by its average in the cubes. Due to the uniform continuity of $\varphi^+$ and \eqref{eq:limit_in_cubes}, we know that
    \[
    \lim_{k \to \infty} \lim_{\varepsilon \to 0} \sum_{j = 1}^{m_k} \int_{(Q_j^k)^+} |\nabla w_\varepsilon(\omega, x)|^2 |\varphi^+(x) - (\varphi^+)_{(Q_j^k)^+}| \, dx = 0 \quad \mathbb{P}\text{-a.s.},
    \]
    where $(\varphi^+)_{(Q_j^k)^+}$ denotes the average of $\varphi^+$ in $(Q_j^k)^+$. Another application of \eqref{eq:limit_in_cubes} also yields
    \[
    \lim_{\varepsilon \to 0} \sum_{j = 1}^{m_k} \int_{(Q_j^k)^+} |\nabla w_\varepsilon(\omega, x)|^2 (\varphi^+)_{(Q_j^k)^+} \, dx = 2\gamma(\omega) \sum_{j = 1}^{m_k} \mathcal{H}^{N - 1}(U^0 \cap Q_j^k) (\varphi^+)_{(Q_j^k)^+} \quad \mathbb{P}\text{-a.s.}
    \]
    Finally, it is clear that
    \[
    \lim_{k \to \infty} 2\gamma(\omega) \sum_{j = 1}^{m_k} \mathcal{H}^{N - 1}(U^0 \cap Q_j^k) (\varphi^+)_{(Q_j^k)^+} = 2\gamma(\omega) \int_{U^0} \varphi^+ \, d\mathcal{H}^{N - 1}.
    \]
    This concludes the proof.
\end{proof}

To conclude the section, we prove the convergence of $(w_\varepsilon(\omega, \cdot))$ for $\mathbb{P}$-almost every $\omega \in \Omega$.

\begin{proposition} \label{pr:first_part_of_existence}
    The sequence $(w_\varepsilon(\omega, \cdot))$ converges weakly in $H^1(U^\pm)$ to $\pm 1$ $\mathbb{P}$-a.s. as $\varepsilon \to 0$.
\end{proposition}

\begin{proof}
    We consider only the upper half-space as usual. Let $\varepsilon > 0$ and $\omega \in \Omega$ be fixed. Since $w_\varepsilon(\omega, \cdot)$ differs from $1$ only around the holes, we can write
    \begin{equation*}
        \int_{U^+} |1 - w_\varepsilon(\omega, x)|^2 \, dx \le \int_{S_\varepsilon(\omega)^+} |1 - w_\varepsilon(\omega, x)|^2 \, dx \\
        + \sum_{(y, \rho) \in I_\varepsilon(\omega)} \int_{B(\varepsilon \bar{y}, \varepsilon r(\omega, y))^+} |1 - w_\varepsilon(\omega, x)|^2 \, dx.
    \end{equation*}
    Since $0 \le w_\varepsilon(\omega, \cdot) \le 1$ $\mathcal{L}^N$-a.e. in $U^+$, Theorem \ref{thm:vanishing_capacity_and_measure} implies
    \[
    \limsup_{\varepsilon \to 0} \int_{S_\varepsilon(\omega)^+} |1 - w_\varepsilon(\omega, x)|^2 \, dx \le \limsup_{\varepsilon \to 0} \mathcal{L}^N(S_\varepsilon(\omega)) = 0 \quad \mathbb{P}\text{-a.s.}
    \]
    Let $(y, \rho) \in I_\varepsilon(\omega)$. Applying Proposition \ref{pr:capacitary_potential_estimates} to \eqref{eq:local_definition} with a change of variables gives the estimate
    \[
    \int_{B(\varepsilon \bar{y}, \varepsilon r(\omega, y))^+} |1 - w_\varepsilon(\omega, x)|^2 \, dx \le C(\varepsilon r(\omega, y))^2 \int_{B(\varepsilon \bar{y}, \varepsilon r(\omega, y))} |\nabla w_\varepsilon(\omega, x)|^2 \, dx.
    \]
    We sum over all points in $I_\varepsilon(\omega)$ and use Lemma \ref{lm:gradient_computation} to obtain
    \begin{multline*}
        \limsup_{\varepsilon \to 0} \sum_{(y, \rho) \in I_\varepsilon(\omega)} \int_{B(\varepsilon \bar{y}, \varepsilon r(\omega, y))^+} |1 - w_\varepsilon(\omega, x)|^2 \, dx \\
        \le \limsup_{\varepsilon \to 0} C \varepsilon^2 \sum_{(y, \rho) \in I_\varepsilon(\omega)} \int_{B(\varepsilon \bar{y}, \varepsilon r(\omega, y))^+} |\nabla w_\varepsilon(\omega, x)|^2 \, dx = 0 \quad \mathbb{P}\text{-a.s.}.
    \end{multline*}
    Hence, we have shown that $w_\varepsilon(\omega, \cdot) \to 1$ in $L^2(U^+)$ $\mathbb{P}$-a.s. as $\varepsilon \to 0$. The weak convergence in $H^1(U^+)$ follows from Proposition \ref{pr:local_gradient_limit}, which implies that $(\nabla w_\varepsilon(\omega, \cdot))$ is $\mathbb{P}$-a.s. bounded in $L^2(U^+; \mathbb{R}^N)$.
\end{proof}

\subsection{A special case of Theorem \ref{thm:existence_of_oscillating_test_functions}} \label{ssec:special_case}

In the last section, we constructed the oscillating test functions $(w_\varepsilon(\omega, \cdot))$ for all $\omega \in \Omega$ and proved their convergence for $\mathbb{P}$-a.e $\omega \in \Omega$ in Proposition \ref{pr:first_part_of_existence}. In this section, we shall prove that the sequence satisfies a special case of the property stated in \eqref{eq:product_of_gradients_convergence} $\mathbb{P}$-almost surely. This will be an important ingredient in the proof of the general case. Our main result is the following proposition.

\begin{proposition} \label{pr:special_case}
    For $\mathbb{P}$-a.e. $\omega \in \Omega$, the sequence $(w_\varepsilon(\omega, \cdot))$ satisfies the following property: Given $(v_\varepsilon)$ with $v_\varepsilon \in H^1(U_\varepsilon(\omega))$, if there exists a subsequence $(v_{\varepsilon_k})$ such that $v_{\varepsilon_k} \rightharpoonup 0$ in $H^1(U^\pm)$ as $k \to \infty$, then
    \[
    \lim_{k \to \infty} \int_{U_{\varepsilon_k}(\omega)} \nabla w_{\varepsilon_k}(\omega, x) \nabla v_{\varepsilon_k}(x) \, dx = 0.
    \]
\end{proposition}

We shall deduce Proposition \ref{pr:special_case} from a minimization property of the oscillating test functions stated in Proposition \ref{pr:minimization_proposition}. First, we define the family of admissible sequences of functions.

\begin{definition}
    For $\omega \in \Omega$, let $\mathcal{O}(\omega)$ be the set of all sequences $(\widetilde{w}_\varepsilon)$ with $\widetilde{w}_\varepsilon \in H^1(U_\varepsilon(\omega))$ such that $\widetilde{w}_\varepsilon \rightharpoonup \pm 1$ in $H^1(U^\pm)$ as $\varepsilon \to 0$.    
\end{definition}

 By Proposition \ref{pr:first_part_of_existence}, we know that $(w_\varepsilon(\omega, \cdot)) \in \mathcal{O}(\omega)$ $\mathbb{P}$-a.s.

\begin{proposition} \label{pr:minimization_proposition}
    For $\mathbb{P}$-a.e. $\omega \in \Omega$, we have
    \begin{equation} \label{eq:liminf_inequality}
        \liminf_{\varepsilon \to 0} \int_{U_\varepsilon(\omega)} |\nabla w_\varepsilon(\omega, x)|^2 \, dx \le \liminf_{\varepsilon \to 0} \int_{U_\varepsilon(\omega)} |\nabla \widetilde{w}_\varepsilon(x)|^2 \, dx
    \end{equation}
    for all $(\widetilde{w}_\varepsilon) \in \mathcal{O}(\omega)$.
\end{proposition}

\begin{proof}[Proof of Proposition \ref{pr:special_case}]
    Fix an $\omega \in \Omega$ for which $(w_\varepsilon(\omega, \cdot)) \in \mathcal{O}(\omega)$ and for which \eqref{eq:liminf_inequality} holds. As $\omega$ is fixed, we omit it from the notation. Assume $(v_\varepsilon)$ is a sequence with $v_\varepsilon \in H^1(U_\varepsilon)$ such that $v_\varepsilon \rightharpoonup 0$ in $H^1(U^\pm)$ as $\varepsilon \to 0$. Clearly, the sequence $(w_\varepsilon + tv_\varepsilon)$ lies in $\mathcal{O}(\omega)$ for all $t > 0$. Hence, Proposition \ref{pr:minimization_proposition} implies
    \begin{equation} \label{eq:nonnegative_liminf}
        \liminf_{\varepsilon \to 0} \int_{U_\varepsilon} (|\nabla (w_\varepsilon + tv_\varepsilon)|^2 - |\nabla w_\varepsilon|^2) \, dx \ge 0.
    \end{equation}
    Since
    \[
    \nabla w_\varepsilon \nabla v_\varepsilon = \frac{1}{2t}\left(|\nabla (w_\varepsilon + tv_\varepsilon)|^2 - |\nabla w_\varepsilon|^2 - t^2|\nabla v_\varepsilon|^2\right),
    \]
    invoking \eqref{eq:nonnegative_liminf}, we get
    \[
    \liminf_{\varepsilon \to 0} \int_{U_\varepsilon} \nabla w_\varepsilon \nabla v_\varepsilon \, dx \ge \liminf_{\varepsilon \to 0} \frac{1}{2t} \int_{U_\varepsilon} (|\nabla (w_\varepsilon + tv_\varepsilon)|^2 - |\nabla w_\varepsilon|^2) \, dx - \frac{Ct}{2} \ge -\frac{Ct}{2},
    \]
    where $C > 0$ is chosen such that $\sup_{\varepsilon} \|\nabla v_\varepsilon\|_{L^2(U; \mathbb{R}^N)} \le C$. Letting $t \to 0$ yields
    \[
    \liminf_{\varepsilon \to 0} \int_{U_\varepsilon} \nabla w_\varepsilon \nabla v_\varepsilon \, dx \ge 0.
    \]
    Replacing $v_\varepsilon$ by $-v_\varepsilon$ provides also the reverse inequality.

    Given a sequence $(v_\varepsilon)$ for which only a subsequence $(v_{\varepsilon_k})$ converges weakly to $0$ in $H^1(U^\pm)$ as $k \to \infty$, we define
    \[
    \widetilde{v}_\varepsilon :=
        \begin{cases}
            v_{\varepsilon_k}, \quad &\text{if } \varepsilon = \varepsilon_k \text{ for some } k \in \mathbb{N}, \\
            1 - |w_\varepsilon(\omega, \cdot)|, \quad &\text{otherwise}.
        \end{cases}
    \]
    Then $\widetilde{v}_\varepsilon \rightharpoonup 0$ in $H^1(U^\pm)$ as $\varepsilon \to 0$. Hence, we can apply the argument given above to conclude that
    \[
    \lim_{k \to \infty} \int_{U_{\varepsilon_k}} \nabla w_{\varepsilon_k} \nabla v_{\varepsilon_k} \, dx = 0.
    \]
\end{proof}

Our goal for the remainder of this section is to prove Proposition \ref{pr:minimization_proposition}. As we have seen in Proposition \ref{pr:local_gradient_limit},
\[
\lim_{\varepsilon \to 0} \int_{U_\varepsilon(\omega)} |\nabla w_\varepsilon(\omega, x)|^2 \, dx = 4\gamma(\omega)\mathcal{H}^{N - 1}(U^0) \quad \mathbb{P}\text{-a.s.}
\]
Therefore, Proposition \ref{pr:minimization_proposition} is equivalent to: for $\mathbb{P}$-a.e. $\omega \in \Omega$,
\[
4\gamma(\omega) \mathcal{H}^{N - 1}(U^0) \le \liminf_{\varepsilon \to 0} \int_{U_\varepsilon(\omega)} |\nabla \widetilde{w}_\varepsilon(x)|^2 \, dx
\]
for all $(\widetilde{w}_\varepsilon) \in \mathcal{O}(\omega)$. To understand the validity of this statement intuitively, we consider an isolated point $(y, \rho) \in I_\varepsilon(\omega)$  for some $\varepsilon > 0$ and $\omega \in \Omega$. As a consequence of Corollary \ref{co:minimization_problem}, we observe that 
\begin{equation} \label{eq:expected_lower_bound}
    \inf_{v} \int_{\varepsilon \bar{y} + \varepsilon^\frac{N - 1}{N - 2}\rho \, D_R} |\nabla v|^2 \, dx = \varepsilon^{N - 1} \rho^{N - 2} \cpct (B(0, 1)^0, B(0, R)) = \varepsilon^{N - 1} \rho^{N - 2} \sigma(R),
\end{equation}
where
\[
R := \frac{r(\omega, y)}{\varepsilon^\frac{1}{N - 2} \rho},
\]
and $v$ ranges over functions in $H^1(\varepsilon \bar{y} + \varepsilon^\frac{N - 1}{N - 2}\rho \, D_R)$ that assume the values $\pm 1$ on the boundary $(\partial B(\varepsilon \bar{y}, \varepsilon r(\omega, y)))^\pm$. Let $(\widetilde{w}_\varepsilon) \in \mathcal{O}(\omega)$. Assuming $\widetilde{w}_\varepsilon$ is close to $\pm 1$ on the boundary $(\partial B(\varepsilon \bar{y}, \varepsilon r(\omega, y)))^\pm$, we expect its Dirichlet energy in $B(\varepsilon \bar{y}, \varepsilon r(\omega, y))$ to be bounded from below by
\[
\varepsilon^{N - 1} \rho^{N - 2} \sigma(R) > \varepsilon^{N - 1} \rho^{N - 2} \sigma(\infty)
\]
up to some small error. If we sum up the approximate lower bounds for each $(y, \rho) \in I_\varepsilon(\omega)$, we obtain
\[
\sum_{(y, \rho) \in I_\varepsilon(\omega)} \int_{B(\varepsilon \bar{y}, \varepsilon r(\omega, y))} |\nabla \widetilde{w}_\varepsilon|^2 \, dx \gtrsim \sigma(\infty) \sum_{(y, \rho) \in I_\varepsilon(\omega)} \varepsilon^{N- 1} \rho^{N - 2} \approx 4 \gamma(\omega) \mathcal{H}^{N - 1}(U^0),
\]
where the last approximation can be justified by the ergodic theorem.

In order to make the intuitive argument above rigorous, we shall show that if the Dirichlet energy of $\widetilde{w}_\varepsilon$ around an $\varepsilon$-isolated point is smaller than the expected lower bound \eqref{eq:expected_lower_bound}, then $\widetilde{w}_\varepsilon$ is not sufficiently close to $\pm 1$ around that point. We state the result in rescaled form.

\begin{lemma} \label{lm:can't_always_have_everything}
    Let $0 < \theta < 1$. There exist $R_0 > 1$ and $c > 0$, depending only on $\theta$ and $N$, such that if
    \[
    \int_{D_R} |\nabla v|^2 \, dx < \theta \sigma(R),
    \]
    then
    \begin{equation} \label{eq:sobolev_distance_to_signed_1}
        \int_{B(0, R)^+} |1 - v|^{2^*} \, dx + \int_{B(0, R)^-} |1 + v|^{2^*} \, dx \ge c \, \mathcal{L}^N(B(0, R))
    \end{equation}
    for all $R \ge R_0$, $v \in H^1(D_R)$, where $2^* = 2N/(N - 2)$.
\end{lemma}

\begin{proof}
    Suppose \eqref{eq:sobolev_distance_to_signed_1} is false. Then, there exist an increasing sequence $(R_k)$ of positive numbers and a sequence of functions $(v_k)$ with $v_k \in H^1(D_{R_k})$ such that
    \begin{gather} \label{eq:gradient_estimate}
        \int_{D_{R_k}} |\nabla v_k|^2 \, dx < \theta \sigma(R_k), \\ \label{eq:function_estimate}
         \int_{B(0, R_k)^+} |1 - v_k|^{2^*} \, dx + \int_{B(0, R_k)^-} |1 + v_k|^{2^*} \, dx < \frac{1}{k}\mathcal{L}^N(B(0, R_k)).
    \end{gather}
    For each $k$, set
    \[
    \bar{v}_k(x', x_N) := \frac{v_k(x', x_N) - v_k(x', -x_N)}{2}.
    \]
    for $x = (x', x_N) \in D_{R_k}$. By definition, $\bar{v}_k$ is odd in the last variable, and $|\bar{v}_k|$ belongs to $H^1(B(0, R_k))$. Thanks to convexity, we have
    \begin{equation*}
        \int_{B(0, R_k)} |\nabla \bar{v}_k(x)|^2 \, dx \le \int_{B(0, R_k)} \frac{1}{2} |\nabla v_k(x', x_N)|^2 + \frac{1}{2} |\nabla v_k(x', -x_N)|^2 \, dx = \int_{B(0, R_k)} |\nabla v_k(x)|^2 \, dx.
    \end{equation*}
    For each $k$, we now consider the function $\widetilde{v}_k := 1 - |\bar{v}_k|$. It is easy to see that $\widetilde{v}_k = 1$ in $B(0, 1)^0$. The estimate in \eqref{eq:gradient_estimate} carries over to $\widetilde{v}_k$:
    \begin{equation} \label{eq:different_gradient_estimate_but_same}
        \int_{B(0, R_k)} |\nabla \widetilde{v}_k|^2 \, dx = \int_{B(0, R_k)} |\nabla \bar{v}_k|^2 \, dx \le \int_{D_{R_k}} |\nabla v_k|^2 \, dx < \theta \sigma(R_k).
    \end{equation}
    Similarly, the bound \eqref{eq:function_estimate} implies
    \begin{multline} \label{eq:different_function_estimate_but_same}
        \int_{B(0, R_k)} |\widetilde{v}_k|^{2^*} \, dx \le \int_{B(0, R_k)^+} |1 - \bar{v}_k|^{2^*} \, dx + \int_{B(0, R_k)^-} |1 + \bar{v}_k|^{2^*} \, dx \\
        \le \int_{B(0, R_k)^+} |1 - v_k|^{2^*} \, dx + \int_{B(0, R_k)^-} |1 + v_k|^{2^*} \, dx < \frac{1}{k} \mathcal{L}^N(B(0, R_k)).    
    \end{multline}
    Here, the first inequality follows from $||a| - |b|| \le \min\{|a - b|, |a + b|\}$ for all $a, b \in \mathbb{R}$. The second inequality is again an application of convexity.
    
    Truncating the functions at $1$ and $-1$ if necessary, we can assume without loss of generality that $-1 \le \widetilde{v}_k \le 1$ in $B(0, R_k)$. Consequently, $(\widetilde{v}_k)$ is uniformly bounded in the $H^1$-norm in any compact domain, thanks to \eqref{eq:different_gradient_estimate_but_same}. Hence, we can apply Rellich-Kondrachov's Compactness Theorem in an increasing sequence of sufficiently regular subsets of $\mathbb{R}^N$ to obtain a function $\widetilde{v} \in W_\text{loc}^{1, 1}(\mathbb{R}^N)$ with $\nabla \widetilde{v} \in L^2(\mathbb{R}^N; \mathbb{R}^N)$, and a subsequence, still denoted by $(\widetilde{v}_k)$ such that
    \[
    \widetilde{v}_k \to \widetilde{v} \text{ pointwise } \mathcal{L}^N\text{-a.e. in } \mathbb{R}^N, \qquad \nabla \widetilde{v}_k \rightharpoonup \nabla \widetilde{v} \text{ in } L^2(\mathbb{R}^N; \mathbb{R}^N),
    \]
    where we extend $\widetilde{v}_k$ and $\nabla \widetilde{v}_k$ by $0$ for both convergences. Our goal is to show that $\widetilde{v} \in L^{2^*}(\mathbb{R}^N)$. By Poincaré's Inequality for the pair of exponents $(2^*, 2)$ \cite[Theorem 8.12]{LL}, there exists a constant $C$ such that
    \begin{equation} \label{eq:poincaré_inequality}
        \left(\int_{B(0, R_k)} |\widetilde{v}_k - (\widetilde{v}_k)_{B(0, R_k)}|^{2^*} \, dx\right)^\frac{1}{2^*} \le C \left(\int_{B(0, R_k)} |\nabla \widetilde{v}_k|^2 \, dx\right)^\frac{1}{2},
    \end{equation}
    where $(\widetilde{v}_k)_{B(0, R_k)}$ denotes the average of $\widetilde{v}_k$ over $B(0, R_k)$. Since the inequality is scaling invariant, it can be verified that $C$ does not depend ond $k$. On the other hand, Hölder's inequality and \eqref{eq:different_function_estimate_but_same} yield
    \[
    |(\widetilde{v}_k)_{B(0, R_k)}| = \left|\frac{1}{\mathcal{L}^N(B(0, R_k))} \int_{B(0, R_k)} \widetilde{v}_k \, dx \right| \le \left(\frac{1}{\mathcal{L}^N(B(0, R_k))} \int_{B(0, R_k)} |\widetilde{v}_k|^{2^*} \, dx \right)^\frac{1}{2^*} \le \frac{1}{k^\frac{1}{2^*}}.
    \]
    It follows that $\lim_{k \to \infty} (\widetilde{v}_k)_{B(0, R_k)} = 0$. We can now apply Fatou's lemma in \eqref{eq:poincaré_inequality} and use \eqref{eq:different_gradient_estimate_but_same} to obtain
    \[
    \int_{\mathbb{R}^N} |\widetilde{v}|^{2^*} \, dx \le \liminf_{k \to \infty} C \left(\int_{B(0, R_k)} |\nabla \widetilde{v}_k|^2 \, dx\right)^\frac{2^*}{2} < \infty.
    \]
    Thus, we conclude that $\widetilde{v} \in L^{2^*}(\mathbb{R}^N)$. Since $\widetilde{v} = 1$ in $B(0, 1)^0$, it may be used as a competitor in the definition of the capacity of $B(0, 1)^0$ in $\mathbb{R}^N$. This implies
    \[
    \sigma(\infty) \le \int_{\mathbb{R}^N} |\nabla \widetilde{v}|^2 \, dx \le \liminf_{k \to \infty} \int_{\mathbb{R}^N} |\nabla \widetilde{v}_k|^2 \, dx \le \theta \sigma(\infty).
    \]
    Since $\theta < 1$, we obtain a contradiction.
\end{proof}

Fix $\omega \in \Omega$ and a sequence $(\widetilde{w}_\varepsilon) \in \mathcal{O}(\omega)$. Let $\delta > 0$ and $0 < \theta < 1$. We denote by $P_\varepsilon^b(\delta, \theta)$ the set of points $(y, \rho) \in I_\varepsilon(\omega) \setminus M_\delta(\omega)$ for which we have
\[
B(\varepsilon \bar{y}, \varepsilon r(\omega, y)) \subset U, \quad \int_{B(\varepsilon \bar{y}, \varepsilon r(\omega, y))} |\nabla \widetilde{w}_\varepsilon|^2 \, dx < \theta \varepsilon^{N - 1} \rho^{N - 2} \sigma\left(\frac{r(\omega, y)}{\varepsilon^\frac{1}{N - 2} \rho}\right).
\]
We also define $P_\varepsilon^g(\delta, \theta)$ as the complement of $P_\varepsilon^b(\delta, \theta)$ in $I_\varepsilon(\omega) \setminus M_\delta(\omega)$. The letters ``b" and ``g" in the superscript stand for ``bad" and ``good", respectively. 

By Lemma \ref{lm:can't_always_have_everything}, we know that there exist $R_0 > 0$ and $c > 0$ such that if
\[
\int_{D_R} |\nabla v|^2 \, dx < \theta \sigma(R),
\]
then
\[
\int_{B(0, R)^+} |1 - v|^{2^*} \, dx + \int_{B(0, R)^-} |1 + v|^{2^*} \, dx \ge c \, \mathcal{L}^N(B(0, R))
\]
for all $R \ge R_0$, $v \in H^1(D_R)$. Note that 
\[
\frac{r(\omega, y)}{\varepsilon^\frac{1}{N - 2} \rho} > \frac{\delta^2}{\varepsilon^\frac{1}{N - 2}} > R_0
\]
for all $(y, \rho) \in P_\varepsilon^b(\delta, \theta)$ and $\varepsilon$ small enough. Hence, after a change of variables, we conclude that
\begin{multline} \label{eq:distance_from_pm1}
    \int_{B(\varepsilon \bar{y}, \varepsilon r(\omega, y))^+} |1 - \widetilde{w}_\varepsilon|^{2^*} \, dx + \int_{B(\varepsilon \bar{y}, \varepsilon r(\omega, y))^-} |1 + \widetilde{w}_\varepsilon|^{2^*} \, dx \\
    \ge c \mathcal{L}^N(B(\varepsilon \bar{y}, \varepsilon r(\omega, y))) \ge c \varepsilon^N \delta^N \mathcal{L}^N(B(0, 1))
\end{multline}
for all $(y, \rho) \in P_\varepsilon^b(\delta, \theta)$ and $\varepsilon$ small enough.

\begin{lemma} \label{lm:counting_argument}
    Fix $\omega \in \Omega$ and a sequence $(\widetilde{w}_\varepsilon) \in \mathcal{O}(\omega)$. Then
    \[
    \lim_{\varepsilon \to 0} \varepsilon^{N - 1} \card(P_\varepsilon^b(\delta, \theta)) = 0
    \]
    for all $\delta > 0$ and $0 < \theta < 1$.
\end{lemma}

\begin{proof}
    Let $\delta > 0$ and $0 < \theta < 1$. Without loss of generality, we can assume that $-1 \le \widetilde{w}_\varepsilon \le 1$ in $U_\varepsilon(\omega)$ since truncating $\widetilde{w}_\varepsilon$ would only increase the cardinality of $P_\varepsilon^b(\delta, \theta)$. Since with this assumption $(\widetilde{w}_\varepsilon)$ is uniformly bounded in $L^\infty(U_\varepsilon(\omega))$, we can estimate the $L^{2*}$-norm from above by the $L^2$-norm in \eqref{eq:distance_from_pm1} to obtain
    \[
    \int_{B(\varepsilon \bar{y}, \varepsilon r(\omega, y))^+} |1 - \widetilde{w}_\varepsilon|^2 \, dx + \int_{B(\varepsilon \bar{y}, \varepsilon r(\omega, y))^-} |1 + \widetilde{w}_\varepsilon|^2 \, dx \ge C \varepsilon^N \delta^N
    \]
    for all $(y, \rho) \in P_\varepsilon^b(\delta, \theta)$. Consequently, we see that
    \begin{align*}
        \int_{U^+ \cap \{x_N \le \varepsilon\}} &|1 - \widetilde{w}_\varepsilon|^2 \, dx + \int_{U^- \cap \{x_N \ge -\varepsilon\}} |1 + \widetilde{w}_\varepsilon|^2 \, dx \\
        &\ge \sum_{(y, \rho) \in P_\varepsilon^b(\delta, \theta)} \left(\int_{B(\varepsilon \bar{y}, \varepsilon r(\omega, y))^+} |1 - \widetilde{w}_\varepsilon|^2 \, dx + \int_{B(\varepsilon \bar{y}, \varepsilon r(\omega, y))^-} |1 + \widetilde{w}_\varepsilon|^2 \, dx\right) \\
        &\ge C \varepsilon^N \delta^N \card(P_\varepsilon^b(\delta, \theta)).
    \end{align*}
    Hence, to conclude we need to show that
    \begin{equation} \label{eq:improved_convergence_order}
        \lim_{\varepsilon \to 0} \frac{1}{\varepsilon} \left(\int_{U^+ \cap \{x_N \le \varepsilon\}} |1 - \widetilde{w}_\varepsilon|^2 \, dx + \int_{U^- \cap \{x_N \ge -\varepsilon\}} |1 - \widetilde{w}_\varepsilon|^2 \, dx\right) = 0.
    \end{equation}
    We do this by applying the Trace Theorem on the upper and lower half-spaces separately. We have the inequalities
    \begin{equation} \label{eq:trace_theorem}
        \begin{aligned}
            \int_{U^+ \cap \{x_N \le \varepsilon\}} |1 - \widetilde{w}_\varepsilon|^2 \, dx &\le C\left(\varepsilon \int_{U^0} |1 - \widetilde{w}_\varepsilon^+|^2 \, d\mathcal{H}^{N - 1} + \varepsilon^2 \int_{U^+ \cap \{x_N \le \varepsilon\}} |\nabla \widetilde{w}_\varepsilon|^2 \, dx\right), \\
            \int_{U^- \cap \{x_N \ge -\varepsilon\}} |1 + \widetilde{w}_\varepsilon|^2 \, dx &\le C\left(\varepsilon \int_{U^0} |1 + \widetilde{w}_\varepsilon^-|^2 \, d\mathcal{H}^{N - 1} + \varepsilon^2 \int_{U^- \cap \{x_N \ge -\varepsilon\}} |\nabla \widetilde{w}_\varepsilon|^2 \, dx\right)
        \end{aligned}
    \end{equation}
    for all $\varepsilon > 0$, where $C$ depends only on the domain $U$. Here, $\widetilde{w}_\varepsilon^+$ and $\widetilde{w}_\varepsilon^-$ stand for the traces of $\widetilde{w}_\varepsilon$ with respect to $U^+$ and $U^-$ respectively. Since $\widetilde{w}_\varepsilon \rightharpoonup \pm 1$ in $H^1(U^\pm)$ as $\varepsilon \to 0$ and since the trace operator is compact, we know that $\widetilde{w}_\varepsilon^\pm \to \pm 1$ in $L^2(U^0; \mathcal{H}^{N - 1})$ as $\varepsilon \to 0$. Hence, dividing the inequalities in \eqref{eq:trace_theorem} by $\varepsilon$ and passing to the limit proves \eqref{eq:improved_convergence_order}.
\end{proof}

We end this section with the proof of Proposition \ref{pr:minimization_proposition}.

\begin{proof}[Proof of Proposition \ref{pr:minimization_proposition}]
    Assume $O \subset \Sigma$ is a Borel measurable set in $\mathbb{R}^N$ with $\overline{O} \subset U^0$ such that its boundary in the relative topology of $\Sigma$ is $\mathcal{H}^{N - 1}$-negligible. Since $\Sigma$ intersects $U$ transversely, we know that $\dist(\overline{O}, \partial U) > 0$.
    
    Let $\omega \in \Omega$ and fix a sequence $(\widetilde{w}_\varepsilon) \in \mathcal{O}(\omega)$. Let $\delta > 0$ and $0 < \theta < 1$. If $(y, \rho) \in P_\varepsilon^g(\delta, \theta)$ and $\varepsilon \bar{y} \in O$, then $B(\varepsilon \bar{y}, \varepsilon r(\omega, y)) \subset U$ for small enough $\varepsilon$ and
    \[
    \int_{B(\varepsilon \bar{y}, \varepsilon r(\omega, y))} |\nabla \widetilde{w}_\varepsilon|^2 \, dx \ge \theta \varepsilon^{N - 1} \rho^{N - 2} \sigma\left(\frac{r(\omega, y)}{\varepsilon^\frac{1}{N - 2} \rho}\right) > \theta \varepsilon^{N - 1} \rho^{N - 2}  \sigma(\infty).
    \]
    Therefore, for small enough $\varepsilon$,
    \begin{multline} \label{eq:estimate_from_below}
        \int_{U_\varepsilon(\omega)} |\nabla \widetilde{w}_\varepsilon|^2 \, dx \ge \sum \limits_{\substack{(y, \rho) \in P_\varepsilon^g(\delta, \theta) \\ \bar{y} \in \frac{1}{\varepsilon} O}} \int_{B(\varepsilon \bar{y}, \varepsilon r(\omega, y))} |\nabla \widetilde{w}_\varepsilon|^2 \, dx \ge \theta \sigma(\infty) \sum \limits_{\substack{(y, \rho) \in P_\varepsilon^g(\delta, \theta) \\ \bar{y} \in \frac{1}{\varepsilon} O}} \varepsilon^{N - 1} \rho^{N - 2} \\
        \ge \theta \sigma(\infty) \sum \limits_{\substack{(y, \rho) \in I_\varepsilon(\omega) \setminus M_\delta(\omega) \\ \bar{y} \in \frac{1}{\varepsilon} O}} \varepsilon^{N - 1} \rho^{N - 2} - \theta \sigma(\infty) \sum_{(y, \rho) \in P_\varepsilon^b(\delta, \theta)} \varepsilon^{N - 1} \rho^{N - 2}.
    \end{multline}
    An application of Lemma \ref{lm:counting_argument} yields
    \[
    \lim_{\varepsilon \to 0} \sum_{(y, \rho) \in P_\varepsilon^b(\delta, \theta)} \varepsilon^{N - 1} \rho^{N - 2} \le \lim_{\varepsilon \to 0} \frac{\varepsilon^{N - 1}}{\delta^{N - 2}}\card(P_\varepsilon^b(\delta, \theta)) = 0.
    \]
    On the other hand, it follows from Proposition \ref{pr:cluster_points_ergodic_theorem} and Lemma \ref{lm:thinned_limit} that
    \[
    \lim_{\delta \to 0} \lim_{\varepsilon \to 0} \Bigg|\sum \limits_{\substack{(y, \rho) \in I_\varepsilon(\omega) \setminus M_\delta(\omega) \\ \bar{y} \in \frac{1}{\varepsilon} O}} \varepsilon^{N - 1} \rho^{N - 2} - \sum \limits_{\substack{(y, \rho) \in M(\omega) \\ \bar{y} \in \frac{1}{\varepsilon} O}} \varepsilon^{N - 1} \rho^{N - 2} \Bigg| = 0 \quad \mathbb{P}\text{-a.s.}
    \]
    Thus, we can use Theorem \ref{thm:ergodic_theorem_second_version} to obtain that
    \begin{equation*}
    \lim_{\delta \to 0} \lim_{\varepsilon \to 0} \sum \limits_{\substack{(y, \rho) \in I_\varepsilon(\omega) \setminus M_\delta(\omega) \\ \bar{y} \in \frac{1}{\varepsilon} O}} \varepsilon^{N - 1} \rho^{N - 2} = \lim_{\varepsilon \to 0} \sum \limits_{\substack{(y, \rho) \in M(\omega) \\ \bar{y} \in \frac{1}{\varepsilon} O}} \varepsilon^{N - 1} \rho^{N - 2} \\
    = \mathcal{H}^{N - 1}(O) \int_0^\infty \rho^{N - 2} \, d\xi(\omega, \rho) \quad \mathbb{P}\text{-a.s.}
    \end{equation*}
    Finally, we conclude with the help of \eqref{eq:estimate_from_below} that
    \[
    \liminf_{\varepsilon \to 0} \int_{U_\varepsilon(\omega)} |\nabla \widetilde{w}_\varepsilon|^2 \, dx \ge \lim_{\delta \to 0} \lim_{\varepsilon \to 0} \, \theta \sigma(\infty) \sum \limits_{\substack{(y, \rho) \in I_\varepsilon(\omega) \setminus M_\delta(\omega) \\ \bar{y} \in \frac{1}{\varepsilon} O}} \varepsilon^{N - 1} \rho^{N - 2} = 4 \theta \gamma(\omega) \mathcal{H}^{N - 1}(O) \quad \mathbb{P}\text{-a.s.}
    \]
    Letting $\theta \to 1$ and approximating $U^0$ by a suitable sequence of increasing sets $(O_k)$ proves the result.
\end{proof}

\section{Proof of Theorem \ref{thm:existence_of_oscillating_test_functions}} \label{sec:proof_of_final_theorem}

In this section, we prove that the oscillating test functions $(w_\varepsilon(\omega, \cdot))$ constructed in Section \ref{ssec:construction_of_oscillations} satisfy all properties stated in Theorem \ref{thm:existence_of_oscillating_test_functions} $\mathbb{P}$-a.s. Our proof is based on an idea introduced by Casado-Díaz in \cite[Theorem 2.1]{CD}.

We begin with an extension of Proposition \ref{pr:local_gradient_limit}.

\begin{lemma} \label{lm:extended_local_gradient}
    For $\mathbb{P}$-a.e. $\omega \in \Omega$ and all $v^\pm \in H^1(U^\pm) \cap L^\infty(U^\pm)$, we have
    \[
    \lim_{\varepsilon \to 0} \int_{U^\pm} |\nabla w_\varepsilon(\omega, x)|^2 v^\pm(x) \, dx = 2\gamma(\omega) \int_{U^0} v^\pm(x) \, d\mathcal{H}^{N - 1}(x).
    \]
\end{lemma}

\begin{proof}
    We restrict ourselves to the upper half-space as usual. Fix some $\omega \in \Omega$ that satisfies Propositions \ref{pr:local_gradient_limit}, \ref{pr:first_part_of_existence} and \ref{pr:special_case}. For simplicity, we omit $\omega$ from the notation. If $v^+ \in C^0(U^+)$, then the result follows directly from Proposition \ref{pr:local_gradient_limit}. For the proof of the general case, we approximate $v^+$ in $H^1(U^+)$ by a sequence $(\varphi_k^+) \subset C^\infty(\overline{U^+})$. Then Proposition \ref{pr:local_gradient_limit} and the trace theorem imply
    \[
    \lim_{k \to \infty} \lim_{\varepsilon \to 0} \int_{U^+} |\nabla w_\varepsilon|^2 \varphi_k^+ \, dx = \lim_{k \to \infty} 2\gamma \int_{U^0} \varphi_k^+ \, d\mathcal{H}^{N - 1} = 2\gamma \int_{U^0} v^+ \, d\mathcal{H}^{N - 1}.
    \]
    Hence, to conclude, we need to show that
    \[
    \lim_{k \to \infty} \lim_{\varepsilon \to 0} \int_{U^+} |\nabla w_\varepsilon|^2 |v^+ - \varphi_k^+| \, dx = 0.
    \]
    This is equivalent to the assertion that
    \begin{equation} \label{eq:does_it_go_to_zero?}
        \lim_{k \to \infty} \int_{U^+} |\nabla w_{\varepsilon_k}|^2 |v^+ - \varphi_k^+| \, dx = 0
    \end{equation}
    for all subsequences $(\varepsilon_k)$. Given a subsequence $(\varepsilon_k)$, we define $v_k^+ := w_{\varepsilon_k}|v^+ - \varphi_k^+|\chi_{U^+}$ for each $k \in \mathbb{N}$ and consider $v_k$ as an element of $H^1(U_{\varepsilon_k})$. Since $v_k \rightharpoonup 0$ in $H^1(U^\pm)$ as $k \to \infty$, Proposition \ref{pr:special_case} implies
    \[
    \lim_{k \to \infty} \int_{U_{\varepsilon_k}} \nabla w_{\varepsilon_k} \nabla v_k^+ \, dx = 0.
    \]
    Computing the gradient of $v_k^+$ gives
    \begin{equation} \label{eq:sum_of_two_terms}
        0 = \lim_{k \to \infty} \int_{U^+} |\nabla w_{\varepsilon_k}|^2 |v^+ - \varphi_k| + (\nabla w_{\varepsilon_k} \nabla |v^+ - \varphi_k|) w_{\varepsilon_k} \, dx.
    \end{equation}
    Since, $\varphi_k \to v^+$ in $H^1(U^+)$ as $k \to \infty$ by assumption, the sequence $(\nabla w_\varepsilon)$ is uniformly bounded in $L^2(U^+; \mathbb{R}^N)$ and $|w_\varepsilon| \le 1$ $\mathcal{L}^N$-a.e. in $U^+$, we can see that
    \[
     \lim_{k \to \infty} \int_{U^+}(\nabla w_{\varepsilon_k} \nabla |v^+ - \varphi_k|) w_{\varepsilon_k} \, dx = 0.
    \]
    Therefore, equation \eqref{eq:sum_of_two_terms} yields \eqref{eq:does_it_go_to_zero?}.
\end{proof}

We conclude with the proof of Theorem \ref{thm:existence_of_oscillating_test_functions}.

\begin{proof}[Proof of Theorem \ref{thm:existence_of_oscillating_test_functions}]
    We fix some $\omega \in \Omega$ for which Propositions \ref{pr:first_part_of_existence}, \ref{pr:special_case} and Lemma \ref{lm:extended_local_gradient} hold. Having fixed $\omega$, we omit it from the notation. Let $(v_\varepsilon)$ be a sequence with $v_\varepsilon \in H^1(U_\varepsilon)$ and suppose there exist functions $v^+$, $v^-$ and a subsequence $(v_{\varepsilon_k})$ such that $v_{\varepsilon_k} \rightharpoonup v^\pm$ in $H^1(U^\pm)$ as $k \to \infty$. In the first part of the proof, we assume additionally that $v^\pm \in L^\infty(U^\pm)$. 
    
    Set $\widetilde{v}_\varepsilon := v_\varepsilon - w_\varepsilon(v^+ \chi_{U^+} - v^- \chi_{U^-})$. We can check easily that $\widetilde{v}_\varepsilon \in H^1(U_\varepsilon)$ for all $\varepsilon > 0$ and that $\widetilde{v}_{\varepsilon_k} \rightharpoonup 0$ in $H^1(U^\pm)$ as $k \to \infty$. Therefore, Proposition \ref{pr:special_case} implies
    \[
    \lim_{k \to \infty} \int_{U_{\varepsilon_k}} \nabla w_{\varepsilon_k} \nabla \widetilde{v}_{\varepsilon_k} \, dx = 0.
    \]
    Computing the gradient of $\widetilde{v}_{\varepsilon_k}$ explicitly, we obtain
    \begin{multline} \label{eq:explicitly_written_out}
        \int_{U_{\varepsilon_k}} \nabla w_{\varepsilon_k} \nabla v_{\varepsilon_k} \, dx - \int_{U^+}|\nabla w_{\varepsilon_k}|^2 v^+ dx  + \int_{U^-} |\nabla w_{\varepsilon_k}|^2 v^- \, dx \\
        - \int_{U^+} (\nabla w_{\varepsilon_k} \nabla v^+) w_{\varepsilon_k} \, dx + \int_{U^-} (\nabla w_{\varepsilon_k} \nabla v^-) w_{\varepsilon_k} \, dx \to 0
    \end{multline}
    as $k \to \infty$. As $v^\pm \in L^\infty(U^\pm)$, it follows from Proposition \ref{pr:first_part_of_existence} that
    \[
    \lim_{k \to \infty} \int_{U^+} (\nabla w_{\varepsilon_k} \nabla v^+) w_{\varepsilon_k} \, dx = \lim_{k \to \infty} \int_{U^-} (\nabla w_{\varepsilon_k} \nabla v^-) w_{\varepsilon_k} \, dx = 0.
    \]
    Hence, with the help of Lemma \ref{lm:extended_local_gradient}, we deduce from \eqref{eq:explicitly_written_out} that
    \[
    \lim_{k \to \infty} \int_{U_{\varepsilon_k}} \nabla w_{\varepsilon_k} \nabla v_{\varepsilon_k} \, dx = 2\gamma \int_{U^0} (v^+ - v^-) \, d\mathcal{H}^{N - 1}.
    \]

    Let us now consider the general case. We define $v_\varepsilon^n := \min \{\max \{v_\varepsilon, -n\}, n\}$ and $(v^\pm)^n := \min \{\max \{v^\pm, -n\}, n\}$. Clearly, we have
    \begin{equation*}
        \lim_{n \to \infty} \lim_{k \to \infty} \int_{U_{\varepsilon_k}} \nabla w_{\varepsilon_k} \nabla v_{\varepsilon_k}^n \, dx = \lim_{n \to \infty} 2\gamma\int_{U^0} ((v^+)^n - (v^-)^n) \, d\mathcal{H}^{N - 1} = 2\gamma \int_{U^0} (v^+ - v^-) \, d\mathcal{H}^{N - 1}.
    \end{equation*}
    Thus, all that remains be to proved is
    \[
    \lim_{n \to \infty} \lim_{k \to \infty} \int_{U_{\varepsilon_k}} \nabla w_{\varepsilon_k} (\nabla v_{\varepsilon_k} - \nabla v_{\varepsilon_k}^n) \, dx = 0.
    \]
    To this end, we let $(\varepsilon_{k_n})$ be a subsequence. Then, it is easy to verify that $v_{\varepsilon_{k_n}} - v_{\varepsilon_{k_n}}^n \rightharpoonup 0$ in $H^1(U^\pm)$ as $n \to \infty$. Hence, as a result of Proposition \ref{pr:special_case}, we obtain
    \[
    \lim_{n \to \infty} \int_{U_{\varepsilon_{k_n}}} \nabla w_{\varepsilon_{k_n}} \nabla (v_{\varepsilon_{k_n}} - v_{\varepsilon_{k_n}}^n) \, dx = 0.
    \]
    The claim now follows from the arbitrariness of the sequence $(\varepsilon_{k_n})$.
\end{proof}

\section{Appendix}

This appendix includes proofs of some results on marked point processes that we have left out from the main body of the article.

\subsection{Proof of Theorem \ref{thm:ergodic_theorem_second_version}} \label{ssec:proof_of_ergodic_theorem}

Let $B \in \mathcal{B}(\mathbb{R}^d)$ be bounded with nonempty interior such that $\mathcal{L}^d(\partial B) = 0$. For the proof, we rely on Theorem \ref{thm:ergodic_theorem_first_version}. However, since $B$ is not necessarily convex, Theorem \ref{thm:ergodic_theorem_first_version} cannot be directly applied. To address this, we first prove the result for half-open rectangles using Theorem \ref{thm:ergodic_theorem_first_version}, and then we use these to approximate $B$. Recall that a half-open rectangle is a set of the form $\Pi_{i = 1}^d I_i$, where $I_i \subset \mathbb{R}$ is a half-open interval for all $i$. Assume for the moment that the result has been proved for all half-open rectangles. Set
\[
S_\varepsilon(\omega) = \frac{1}{\mathcal{L}^d\left(\frac{1}{\varepsilon}B\right)} \sum \limits_{\substack{(y, \rho) \in M(\omega) \\ y \in \frac{1}{\varepsilon}B}} g(\rho), \quad \omega \in \Omega.
\]
For $l \in \mathbb{N}$, let $Q_1, \dots, Q_{m_l}$ denote the half-open dyadic cubes with side length $2^{-l}$ contained in the interior of $B$, and let $Q'_1, \dots, Q'_{n_l}$ denote those that intersect $\bar{B}$. Then
\[
S_\varepsilon(\omega) \le \sum_{i = 1}^{n_l} \frac{\mathcal{L}^d(Q'_i)}{\mathcal{L}^d(B)} \frac{1}{\mathcal{L}^d\left(\frac{1}{\varepsilon}Q'_i\right)} \sum \limits_{\substack{(y, \rho) \in M(\omega) \\ y \in \frac{1}{\varepsilon}Q'_i}} g(\rho).
\]
Similarly,
\[
S_\varepsilon(\omega) \ge \sum_{i = 1}^{m_l} \frac{\mathcal{L}^d(Q_i)}{\mathcal{L}^d(B)} \frac{1}{\mathcal{L}^d\left(\frac{1}{\varepsilon}Q_i\right)} \sum \limits_{\substack{(y, \rho) \in M(\omega) \\ y \in \frac{1}{\varepsilon}Q_i}} g(\rho).
\]
Consequently, from our assumption that the result holds for all half-open rectangles, we obtain
\begin{align*}
    \frac{\mathcal{L}^d\left(\bigcup_{i = 1}^{m_l} Q_i \right)}{\mathcal{L}^d(B)} \int_0^\infty g(\rho) \, d\xi(\omega, \rho) &\le \liminf_{\varepsilon \to 0} S_\varepsilon(\omega) \\
    &\le \limsup_{\varepsilon \to 0} S_\varepsilon(\omega) \le \frac{\mathcal{L}^d\left(\bigcup_{i = 1}^{n_l} Q'_i \right)}{\mathcal{L}^d(B)} \int_0^\infty g(\rho) \, d\xi(\omega, \rho) \quad \mathbb{P}\text{-a.s.}
\end{align*}
Taking $l \to \infty$ and recalling that $\mathcal{L}^d(\partial B) = 0$, we conclude the proof of the general case.

It remains to prove the result for half-open rectangles. Assume $Q = \Pi_{i = 1}^d I_i$, where $I_i \subset \mathbb{R}$ is a half-open interval for all $i$. Notice first that we have
\begin{align*}
    \left(\frac{k - 1}{k} \right)^d \frac{1}{\mathcal{L}^d((k - 1)Q)} \sum \limits_{\substack{(y, \rho) \in M(\omega) \\ y \in (k - 1)Q}} g(\rho) &\le \frac{1}{\mathcal{L}^d\left(\frac{1}{\varepsilon} Q \right)} \sum \limits_{\substack{(y, \rho) \in M(\omega) \\ y \in \frac{1}{\varepsilon} Q}} g(\rho) \\
    &\le \left(\frac{k}{k - 1} \right)^d \frac{1}{\mathcal{L}^d(kQ)} \sum \limits_{\substack{(y, \rho) \in M(\omega) \\ y \in kQ}} g(\rho)
\end{align*}
if $1/k \le \varepsilon \le 1/(k - 1)$. Hence, it is enough to prove convergence along the subsequence $(1/k)$. Let $J_Q$ be the largest index set such that $0 \in \bar{I}_i$ for all $i \in J_Q$. If $\card(J_Q) = d$, then $0 \in \bar{Q}$, so that $(kQ)$ is a convex averaging sequence. In this case, the result follows from Theorem \ref{thm:ergodic_theorem_first_version}. Now, we proceed inductively. Suppose that the result holds when $\card(J_Q) = j$ with $1 \le j \le d$. If $\card(J_Q) = j - 1$, then there exists an index $i_0$ such that $I_{i_0} = [a, b)$ with either $a > 0$ or $b < 0$. Without loss of generality, we assume $a > 0$. We consider the sets $Q^a= \Pi_{i = 1}^d I^a_i$ and $Q^b = \Pi_{i = 1}^d I^b_i$, where 
\[
    I_i = I^a_i = I^b_i \text{ for } i \neq i_0, \quad I^a_{i_0} = [0, a), \quad I^b_{i_0} = [0, b).
\]
Clearly, we have $Q^b = Q \cup Q^a$ and $Q \cap Q^a = \emptyset$. Hence,
\begin{equation} \label{eq:partition_of_the_cube}
    \frac{1}{\mathcal{L}^d(kQ)} \sum \limits_{\substack{(y, \rho) \in M(\omega) \\ y \in kQ}} g(\rho) = \frac{b}{b - a} \frac{1}{\mathcal{L}^d(kQ^b)} \sum \limits_{\substack{(y, \rho) \in M(\omega) \\ y \in kQ^b}} g(\rho) - \frac{a}{b - a} \frac{1}{\mathcal{L}^d(kQ^a)} \sum \limits_{\substack{(y, \rho) \in M(\omega) \\ y \in kQ^a}} g(\rho).
\end{equation}
Also, by definition, we know that $\card(J_{Q^a}) = \card(J_{Q^b}) = j$. Therefore, by the induction hypothesis, the right-hand side of \eqref{eq:partition_of_the_cube} converges to $\int_0^\infty g(\rho) \, d\xi(\omega, \rho)$ $\mathbb{P}$-a.s. as $k \to \infty$. This concludes the proof of Theorem \ref{thm:ergodic_theorem_second_version}.

\subsection{Thinning of marked point processes} \label{ssec:thinning_of_marked_point_processes}

In this section we study the thinning of a marked point process introduced in Section \ref{sec:analysis_of_perforations}. Let $d \in \mathbb{N}$. For $Y \in \mathbb{M}^d$ and $(y, \rho) \in Y$, we recall from \eqref{eq:distance_to_closest_neighbor} that $d_Y(y)$ is the distance of $y$ to its closest neighbor in $Y$. Given $\delta > 0$, we define the \textit{\textbf{thinning map}} $\mathcal{T}_\delta : \mathbb{M}^d \rightarrow \mathbb{M}^d$ as follows: the pair $(y, \rho)$ belongs to $\mathcal{T}_\delta(Y)$ if and only if $(y, \rho) \in Y$ and
\[
\min \left\{\frac{d_Y(y)}{2}, \frac{1}{\rho}\right\} < \delta.
\]
Let $M : (\Omega, \mathcal{F}, \mathbb{P}) \rightarrow (\mathbb{M}^d, \mathcal{M}^d)$ be a marked point process. As in Section \ref{sec:analysis_of_perforations}, we denote the thinned process $\mathcal{T}_\delta \circ M$ simply by $M_\delta$. To be sure that $M_\delta$ is again a marked point process, we have to show that $\mathcal{T}_\delta$ is a measurable mapping. The proof of this fact will be given in Section \ref{ssec:proof_of_measurability_of_the_thinning_map}.

The following proposition shows that stationarity is preserved under thinning. This is not surprising, as the thinning condition does not depend on the absolute position of a point in space.

\begin{proposition} \label{pr:thinned_process_is_stationary}
    Let $M : (\Omega, \mathcal{F}, \mathbb{P}) \rightarrow (\mathbb{M}^d, \mathcal{M}^d)$ be a stationary marked point process. Then $M_\delta$ is stationary for all $\delta > 0$. Furthermore, if $M$ has finite intensity, then $M_\delta$ has finite intensity as well.
\end{proposition}

\begin{proof}
    First, we show that the thinning operation commutes with translation, that is, $\mathcal{T}_\delta(Y_\tau) = \mathcal{T}_\delta(Y)_\tau$ for all $\delta > 0$, $\tau \in \mathbb{R}^d$. We know that $(y, \rho) \in \mathcal{T}_\delta(Y)_\tau$ if and only if $(y - \tau, \rho) \in \mathcal{T}_\delta(Y)$. Since $d_Y(y - \tau) = d_{Y_\tau}(y)$, we see that $(y - \tau, \rho) \in \mathcal{T}_\delta(Y)$ if and only if $(y, \rho) \in \mathcal{T}_\delta(Y_\tau)$.
    
    Now assume $\mathcal{A} \in \mathcal{M}^d$. It follows from the commutativity relation that $\mathcal{T}_\delta^{-1}(\mathcal{A}_\tau) = \mathcal{T}_\delta^{-1}(\mathcal{A})_\tau$. Finally, we conclude
    \[
    \mathbb{P}(M_\delta^{-1}(\mathcal{A}_\tau)) = \mathbb{P}(M^{-1}(\mathcal{T}_\delta^{-1}(\mathcal{A}_\tau))) = \mathbb{P}(M^{-1}(\mathcal{T}_\delta^{-1}(\mathcal{A})_\tau)) = \mathbb{P}(M^{-1}(\mathcal{T}_\delta^{-1}(\mathcal{A}))) = \mathbb{P}(M_\delta^{-1}(\mathcal{A})),
    \]
    where we used the stationarity of $M$ in the next to last equality. The second claim in the statement of the theorem follows from the inequality $M_\delta(E) \le M(E)$ for all $E \in \mathcal{B}(\mathbb{R}^d \times \mathbb{R}^+)$.
\end{proof}

Let $M : (\Omega, \mathcal{F}, \mathbb{P}) \rightarrow (\mathbb{M}^d, \mathcal{M}^d)$ be a stationary marked point process with finite intensity. Let $\lambda$ denote the finite measure obtained by applying Proposition \ref{pr:intensity_measure_stationary} to $M$. By the previous result, the marked point process $M_\delta$ is similarly stationary with finite intensity. Assume $g : \mathbb{R}^+ \rightarrow [0, \infty)$ is a $\lambda$-integrable function and $B \in \mathcal{B}(\mathbb{R}^d)$ is bounded with nonempty interior such that $\mathcal{L}^d(\partial B) = 0$. The ergodic theorem for stationary marked point processes implies that the limit
\begin{equation} \label{eq:ergodic_theorem_for_each_delta}
    \lim_{\varepsilon \to 0} \frac{1}{\mathcal{L}^d\left(\frac{1}{\varepsilon} B\right)} \sum \limits_{\substack{(y, \rho) \in M_\delta(\omega) \\ y \in \frac{1}{\varepsilon} B}} g(\rho)
\end{equation}
exists almost surely for all $\delta > 0$. Since the marked point processes $M_\delta$ get thinner as $\delta \to 0$, we expect this limit to vanish almost surely. Under the assumption that $g$ is locally bounded, this is indeed true, as we state in the next theorem.

\begin{theorem} \label{thm:ergodic_limit_for_thinned_processes}
    Let $M : (\Omega, \mathcal{F}, \mathbb{P}) \rightarrow (\mathbb{M}^d, \mathcal{M}^d)$ be a stationary marked point process with finite intensity. Assume $B \in \mathcal{B}(\mathbb{R}^d)$ is bounded with nonempty interior such that $\mathcal{L}^d(\partial B) = 0$. Then
    \begin{equation} \label{eq:ergodic_limit_vanishes}
        \lim_{\delta \to 0} \, \lim_{\varepsilon \to 0} \frac{1}{\mathcal{L}^d\left(\frac{1}{\varepsilon} B\right)} \sum \limits_{\substack{(y, \rho) \in M_\delta(\omega) \\ y \in \frac{1}{\varepsilon} B}} g(\rho) = 0 \quad \mathbb{P}\text{-a.s.}
    \end{equation}
    for any locally bounded and $\lambda$-integrable function $g : \mathbb{R}^+ \rightarrow [0, \infty)$, where $\lambda$ is the measure on $\mathbb{R}^+$ given by $\Lambda = \lambda \times \mathcal{L}^d$.
\end{theorem}

Let $\lambda_\delta$ denote the finite measure obtained by applying Proposition \ref{pr:intensity_measure_stationary} to $M_\delta$, and $\xi_\delta$ the random measure obtained by applying Lemma \ref{lm:conditional_expectation_random_measure} to $M_\delta$. It follows from Theorem \ref{thm:ergodic_theorem_second_version} that we can express the limit in \eqref{eq:ergodic_theorem_for_each_delta} as the integral
\[
\int_0^\infty g(\rho) \, d\xi_\delta(\omega, \rho)
\]
almost surely for all $\delta > 0$. To prove Theorem \ref{thm:ergodic_limit_for_thinned_processes}, we shall show that the integrals converge to $0$ almost surely as $\delta \to 0$. In order to avoid measure-theoretic complications, we replace $\delta$ with the countable sequence $(1/k)$. We begin with a lemma.

\begin{lemma} \label{lm:random_measure_goes_to_zero} For $\mathbb{P}$-a.e. $\omega \in \Omega$, we have
    \[
    \lim_{k \to \infty} \xi_{1/k}(\omega, \mathbb{R}^+) = 0.
    \]
\end{lemma}

\begin{proof}
    Let $B \in \mathcal{B}(\mathbb{R}^d)$ be bounded with $\mathcal{L}^d(B) = 1$. By \eqref{eq:conditional_expectation_measure_representation} and the definition of conditional expectation,
    \begin{equation} \label{eq:conditional_expectation_application}
        \int_{\Omega} \xi_{1/k}(\omega, \mathbb{R}^+) \, d\mathbb{P}(\omega) = \int_{\Omega} M_{1/k}(\omega)(B \times \mathbb{R}^+) \, d\mathbb{P}(\omega).
    \end{equation}
    Since $M(\omega)(B \times \mathbb{R}^+)$ is finite for all $\omega \in \Omega$, we observe that 
    \[
    \lim_{k \to \infty} M_{1/k}(\omega)(B \times \mathbb{R}^+) = 0
    \]
    for all $\omega \in \Omega$. Moreover, the integrable random variable $M(B \times \mathbb{R}^+)$ dominates the sequence $(M_{1/k}(B \times \mathbb{R}^+))$. Applying the dominated convergence theorem to \eqref{eq:conditional_expectation_application} results in
    \[
    \lim_{k \to \infty} \int_{\Omega} \xi_{1/k}(\omega, \mathbb{R}^+) \, d\mathbb{P}(\omega) = 0.
    \]
    Therefore, there exists a subsequence $(1/k_n)$ such that $(\xi_{1/k_n}(\omega, \mathbb{R}^+))$ converges to $0$ $\mathbb{P}$-a.s. At the same time, the inequality 
    \[
    M_{1/k}(B \times \mathbb{R}^+) \le M_{1/l}(B \times \mathbb{R}^+), \quad k \ge l
    \]
    implies
    \[
    \xi_{1/k}(\omega, \mathbb{R}^+) \le \xi_{1/l}(\omega, \mathbb{R}^+) \quad \mathbb{P}\text{-a.s.}, \quad k \ge l.
    \]
    Thus, $\lim_{k \to \infty} \xi_{1/k}(\omega, \mathbb{R}^+) = 0$ $\mathbb{P}$-a.s.
    
\end{proof}

\begin{proposition} \label{pr:thinned_limit}
    Let $g : \mathbb{R}^+ \rightarrow [0, \infty)$ be locally bounded and $\lambda$-integrable. Then
    \[
    \lim_{k \to \infty} \int_0^\infty g(\rho) \, d\xi_{1/k}(\omega, \rho) = 0 \quad \mathbb{P}\text{-a.s.}
    \]
\end{proposition}

\begin{proof}
    For $m \in \mathbb{N}$, define
    \[
    X_m(\omega) := \lim_{k \to \infty} \int_m^\infty g(\rho) \, d\xi_{1/k}(\omega, \rho),
    \]
    which is well-defined almost surely due to the monotonicity of the random measures. We claim that $(X_m)$ converges pointwise to $0$ almost surely. Let $B \in \mathcal{B}^d(\mathbb{R}^d)$ such that $\mathcal{L}^d(B) = 1$. If we define
    \[
    G_m(Y) := \int_{B \times (m, \infty)} g(\rho) \, dY(y, \rho), \quad Y \in \mathbb{M}^d,
    \]
    then, using \eqref{eq:conditional_expectation_integral_representation} and Campbell's Theorem, we obtain
    \[
    \int_\Omega \int_m^\infty g(\rho) \, d\xi_{1/k}(\omega, \rho) \, d\mathbb{P}(\omega) = \int_\Omega (G_m \circ M)(\omega) \, d\mathbb{P}(\omega) = \int_m^\infty g(\rho) \, d\lambda_{1/k}(\rho).
    \]
    Consequently, Fatou's lemma implies
    \begin{equation*}
    \int_\Omega X_m(\omega) \, d\mathbb{P}(\omega) \le \liminf_{k \to \infty} \int_\Omega \int_m^\infty g(\rho) \, d\xi_{1/k}(\omega, \rho) \, d\mathbb{P}(\omega) \\
    = \liminf_{k \to \infty} \int_m^\infty g(\rho) \, d\lambda_{1/k}(\rho) \le \int_m^\infty g(\rho) \, d\lambda(\rho).
    \end{equation*}
    Since $g$ is $\lambda$-integrable,
    \[
    \lim_{m \to \infty} \int_\Omega X_m(\omega) \, d\mathbb{P}(\omega) = 0.
    \]
    Therefore, there exists a subsequence $(m_n)$ such that $(X_{m_n})$ converges pointwise to $0$ almost surely. At the same time, Lemma \ref{lm:random_measure_goes_to_zero} implies
    \[
    \limsup_{k \to \infty} \int_0^m g(\rho) \, d\xi_{1/k}(\omega, \rho) \le \limsup_{k \to \infty} \|g\|_{L^\infty(0, m)} \xi_{1/k}(\omega, \mathbb{R}^+) = 0 \quad \mathbb{P}\text{-a.s.}
    \]
    for all $m \in \mathbb{N}$. Thus,
    \[
    \limsup_{k \to \infty} \int_0^\infty g(\rho) \, d\xi_{1/k}(\omega, \rho) \le \limsup_{n \to \infty} X_{m_n}(\omega) = 0 \quad \mathbb{P}\text{-a.s.}
    \]
    This concludes the proof.
\end{proof}

We can now deduce Theorem \ref{thm:ergodic_limit_for_thinned_processes} easily from Proposition \ref{pr:thinned_limit}.

\begin{proof}[Proof of Theorem \ref{thm:ergodic_limit_for_thinned_processes}]
    By Theorem \ref{thm:ergodic_theorem_second_version}, we know that
    \[
    \lim_{\varepsilon \to 0} \frac{1}{\mathcal{L}^d\left(\frac{1}{\varepsilon} B\right)} \sum \limits_{\substack{(y, \rho) \in M_{1/k}(\omega) \\ y \in \frac{1}{\varepsilon} B}} g(\rho) = \int_0^\infty g(\rho) \, d\xi_{1/k}(\omega, \rho) \quad \mathbb{P}\text{-a.s.}
    \]
    Hence, Proposition \ref{pr:thinned_limit} implies
    \[
    \lim_{k \to \infty} \, \lim_{\varepsilon \to 0} \frac{1}{\mathcal{L}^d\left(\frac{1}{\varepsilon} B\right)} \sum \limits_{\substack{(y, \rho) \in M_{1/k}(\omega) \\ y \in \frac{1}{\varepsilon} B}} g(\rho) = 0 \quad \mathbb{P}\text{-a.s.}
    \]
    The claim now follows from the fact that
    \[
    \sum \limits_{\substack{(y, \rho) \in M_{\delta'}(\omega) \\ y \in \frac{1}{\varepsilon} B}} g(\rho) \le \sum \limits_{\substack{(y, \rho) \in M_\delta(\omega) \\ y \in \frac{1}{\varepsilon} B}} g(\rho)
    \]
    for all $\delta' < \delta$ and $\omega \in \Omega$.
\end{proof}

\subsection{Proof of the measurability of the thinning map} \label{ssec:proof_of_measurability_of_the_thinning_map}

Let $d \in \mathbb{N}$ and $\delta > 0$. We require some simple facts before we begin the proof of the measurability of the thinning map $\mathcal{T}_\delta$. Let $X : \mathbb{M}^d \rightarrow \mathbb{M}^d$ be an arbitrary mapping. We note that $X$ is measurable with respect to the $\sigma$-algebra $\mathcal{M}^d$ if and only if the function $Y \mapsto X(Y)(E)$ is measurable for all $E \in \mathcal{B}(\mathbb{R}^d \times \mathbb{R}^+)$. In fact, it is enough to consider only sets of the form $Q \times I$, where $Q \in \mathbb{R}^d$ and $I \subset \mathbb{R}^+$ are half-open rectangles. We merely sketch the proof here, leaving rigorous details to the reader. Consider the family of all sets $E \in \mathcal{B}(\mathbb{R}^d \times \mathbb{R}^+)$ for which $Y \mapsto X(Y)(E)$ is measurable. This family forms a $\lambda$-system. Meanwhile, the collection of sets of the form $Q \times I$, where $Q \in \mathbb{R}^d$, $I \subset \mathbb{R}^+$ are half-open rectangles, forms a $\pi$-system that generates the $\sigma$-algebra $\mathcal{B}(\mathbb{R}^d \times \mathbb{R}^+)$. By the $\pi$-$\lambda$ theorem, the claim follows.

We shall require the following lemma.

\begin{lemma} \label{lm:measurability_for_special_sets}
    Let $\delta > 0$. Assume $Q \in \mathbb{R}^d$, $I \subset \mathbb{R}^+$ are half-open rectangles. Define
    \[
    \mathcal{A} := \{Y \in \mathbb{M}^d : (\mathcal{T}_\delta \circ Y)(Q \times I) > 0\}.
    \]
    Then $\mathcal{A} \in \mathcal{M}^d$.
\end{lemma}

\begin{proof}
    Define $I_\delta := (0, 1/\delta)$. If $Y \in \mathcal{A}$, then there exists $(y, \rho) \in Y \cap (Q \times I)$ such that either $\rho \notin I_\delta$, or $\rho \in I_\delta$ and $d_Y(y) < 2\delta$. Thus, we can write $\mathcal{A} = \mathcal{A}_1 \cup \mathcal{A}_2$, where
    \begin{align*}
        \mathcal{A}_1 &:= \{Y \in \mathbb{M}^d : Y(Q \times (I \cap I_\delta^c)) > 0\}, \\
        \mathcal{A}_2 &:= \{Y \in \mathbb{M}^d : \text{there exists } (y, \rho) \in Y \cap (Q \times (I \cap I_\delta)) \text{ such that } d_Y(y) < 2\delta\}.
    \end{align*}
    Clearly, the set $\mathcal{A}_1$ belongs to $\mathcal{M}^d$. Let $Y \in \mathbb{M}^d$. If there exists $(y, \rho) \in Y \cap (Q \times (I \cap I_\delta))$ such that $d_Y(y) < 2\delta$, then we can find $q \in \mathbb{Q}^d$ and a rational number $s < \delta$ such that $(y, \rho)$ is the only element of $Y$ in $(Q \cap B(q, s)) \times (I \cap I_\delta)$, and there is another element of $Y$ in $B(q, 2\delta - s) \times \mathbb{R}^+$ distinct from $(y, \rho)$. In other words, if $Y \in \mathcal{A}_2$, then there exist $q \in \mathbb{Q}^d$ and $s \in \mathbb{Q}$ with $s < \delta$ such that
    \[
    Y((Q \cap B(q, s)) \times (I \cap I_\delta)) = 1 \quad \text{and} \quad Y(B(q, 2\delta - s) \times \mathbb{R}^+) > 1.
    \]
    It is easy to verify that the reverse implication holds as well, so we can write
    \[
    \mathcal{A}_2 = \bigcup_{q \in \mathbb{Q}^d} \bigcup \limits_{\substack{s \in \mathbb{Q} \\ s < \delta}} \{Y : Y((Q \cap B(q, s)) \times (I \cap I_\delta)) = 1\} \cap \{Y : Y(B(q, 2\delta - s) \times \mathbb{R}^+) > 1\}.
    \]
    As all sets in the union belong to $\mathcal{M}^d$, so does $\mathcal{A}_2$. This concludes the proof that $\mathcal{A} \in \mathcal{M}^d$.
\end{proof}

Now, we come to the proof of the measurability of the thinning map.

\begin{proposition}
    The thinning map $\mathcal{T}_\delta$ is measurable for all $\delta > 0$.
\end{proposition}

\begin{proof}
    Fix $\delta > 0$. Let $Q \subset \mathbb{R}^d$, $I \subset \mathbb{R}^+$ be half-open rectangles. We prove that the function $Y \mapsto \mathcal{T}_\delta(Y)(Q \times I)$ is measurable. For each $k \in \mathbb{N}$, we partition $Q$ into disjoint half-open rectangles $Q_{1, k}, \dots, Q_{n_k, k}$ such that $\diam (Q_{i, k}) < 1/k$ for all $i$. Define
    \[
    \mathcal{A}_{i, k} := \{Y \in \mathbb{M}^d : (\mathcal{T}_\delta \circ Y)(Q_{i, k} \times I) > 0\}
    \]
    for $i = 1, \dots, n_k$. By Lemma \ref{lm:measurability_for_special_sets}, the sets $\mathcal{A}_{i, k}$ are measurable. We claim that
    \begin{equation} \label{eq:sum_of_characteristic_functions}
        \mathcal{T}_\delta(Y)(Q \times I) = \lim_{k \to \infty} \sum_{i = 1}^{n_k} \chi_{\mathcal{A}_{i, k}}(Y)
    \end{equation}
    for all $Y \in \mathbb{M}^d$. As the characteristic functions $\chi_{\mathcal{A}_{i, k}}$ are measurable, this will conclude the proof. Fix $Y \in \mathbb{M}^d$. Since $Y(Q \times I)$ is finite, there exists $k_0 \in \mathbb{N}$ such that $Y(Q_{i, k} \times I) \le 1$ for all $k \ge k_0$ and for all $i$. Hence, $\mathcal{T}_\delta(Y)(Q_{i, k} \times I) = \chi_{\mathcal{A}_{i, k}}(Y)$ for all $k \ge k_0$ and for all $i$. As a result,
    \[
    \mathcal{T}_\delta(Y)(Q \times I) = \sum_{i = 1}^{n_k} \mathcal{T}_\delta(Y)(Q_{i, k} \times I) = \sum_{i = 1}^{n_k} \chi_{\mathcal{A}_{i, k}}(Y)
    \]
    for all $k \ge k_0$. Thus, our claim is proved. 
\end{proof}

\section*{Acknowledgments}

The author was funded by the Deutsche Forschungsgemeinschaft (DFG, German Research Foundation) under Germany's Excellence Strategy EXC 2044–390685587, Mathematics Münster: Dynamics–Geometry–Structure.

\bibliography{references}

@book {DVJ,
    AUTHOR = {Daley, D. J. and Vere-Jones, D.},
     TITLE = {An introduction to the theory of point processes. {V}ol. {II}},
    SERIES = {Probability and its Applications (New York)},
   EDITION = {Second},
      NOTE = {General theory and structure},
 PUBLISHER = {Springer, New York},
      YEAR = {2008},
     PAGES = {xviii+573},
      ISBN = {978-0-387-21337-8},
   MRCLASS = {60G55 (60-02 60G57)},
  MRNUMBER = {2371524},
MRREVIEWER = {Gail\ Ivanoff},
       DOI = {10.1007/978-0-387-49835-5},
       URL = {https://doi.org/10.1007/978-0-387-49835-5},
}

@book {CSKM,
    AUTHOR = {Chiu, S. N. and Stoyan, D. and Kendall, W. S.
              and Mecke, J.},
     TITLE = {Stochastic geometry and its applications},
    SERIES = {Wiley Series in Probability and Statistics},
   EDITION = {Third},
 PUBLISHER = {John Wiley \& Sons, Ltd., Chichester},
      YEAR = {2013},
     PAGES = {xxvi+544},
      ISBN = {978-0-470-66481-0},
   MRCLASS = {60D05 (52A22 60G55 60G57)},
  MRNUMBER = {3236788},
MRREVIEWER = {Elena\ Villa},
       DOI = {10.1002/9781118658222},
       URL = {https://doi.org/10.1002/9781118658222},
}

@article {D,
    AUTHOR = {Damlamian, A.},
     TITLE = {Le probl\`eme de la passoire de {N}eumann},
   JOURNAL = {Rend. Sem. Mat. Univ. Politec. Torino},
  FJOURNAL = {Rendiconti del Seminario Matematico (gi\`a{} ``Conferenze di
              Fisica e di Matematica''). Universit\`a{} e Politecnico di
              Torino},
    VOLUME = {43},
      YEAR = {1985},
    NUMBER = {3},
     PAGES = {427--450},
      ISSN = {0373-1243},
   MRCLASS = {35B25 (35J05 35J85)},
  MRNUMBER = {884870},
MRREVIEWER = {Dietrich\ G\"ohde},
}

@article {P,
    AUTHOR = {Picard, C.},
     TITLE = {Analyse limite d'\'equations variationnelles dans un domaine
              contenant une grille},
   JOURNAL = {RAIRO Mod\'el. Math. Anal. Num\'er.},
  FJOURNAL = {RAIRO Mod\'elisation Math\'ematique et Analyse Num\'erique},
    VOLUME = {21},
      YEAR = {1987},
    NUMBER = {2},
     PAGES = {293--326},
      ISSN = {0764-583X,1290-3841},
   MRCLASS = {35B40 (35J05 35J85 49A29 65N15)},
  MRNUMBER = {896245},
MRREVIEWER = {Shuzi\ Zhou},
       DOI = {10.1051/m2an/1987210202931},
       URL = {https://doi.org/10.1051/m2an/1987210202931},
}

@article {CD,
    AUTHOR = {Casado-Díaz, J.},
     TITLE = {Existence of a sequence satisfying {C}ioranescu-{M}urat
              conditions in homogenization of {D}irichlet problems in
              perforated domains},
   JOURNAL = {Rend. Mat. Appl. (7)},
  FJOURNAL = {Rendiconti di Matematica e delle sue Applicazioni. Serie VII},
    VOLUME = {16},
      YEAR = {1996},
    NUMBER = {3},
     PAGES = {387--413},
      ISSN = {1120-7183,2532-3350},
   MRCLASS = {35B27 (35J25)},
  MRNUMBER = {1422390},
MRREVIEWER = {Patrizia\ Donato},
}

@incollection {M,
    AUTHOR = {Murat, F.},
     TITLE = {The {N}eumann sieve},
 BOOKTITLE = {Nonlinear variational problems ({I}sola d'{E}lba, 1983)},
    SERIES = {Res. Notes in Math.},
    VOLUME = {127},
     PAGES = {24--32},
 PUBLISHER = {Pitman, Boston, MA},
      YEAR = {1985},
      ISBN = {0-273-08670-7},
   MRCLASS = {35J25 (35B25 73K10)},
  MRNUMBER = {807534},
}

@incollection {SP,
    AUTHOR = {S\'anchez-Palencia, E.},
     TITLE = {Boundary value problems in domains containing perforated
              walls},
 BOOKTITLE = {Nonlinear partial differential equations and their
              applications. {C}oll\`ege de {F}rance {S}eminar, {V}ol. {III}
              ({P}aris, 1980/1981)},
    SERIES = {Res. Notes in Math.},
    VOLUME = {70},
     PAGES = {309--325},
 PUBLISHER = {Pitman, Boston, Mass.-London},
      YEAR = {1982},
      ISBN = {0-273-08568-9},
   MRCLASS = {35C20 (35Q20 76S05)},
  MRNUMBER = {670282},
MRREVIEWER = {Alberto\ Valli},
}

@article {SHSP,
    AUTHOR = {Sanchez-Hubert, J. and S\'anchez-Palencia, E.},
     TITLE = {Acoustic fluid flow through holes and permeability of
              perforated walls},
   JOURNAL = {J. Math. Anal. Appl.},
  FJOURNAL = {Journal of Mathematical Analysis and Applications},
    VOLUME = {87},
      YEAR = {1982},
    NUMBER = {2},
     PAGES = {427--453},
      ISSN = {0022-247X},
   MRCLASS = {76D99 (76Q05)},
  MRNUMBER = {658023},
       DOI = {10.1016/0022-247X(82)90133-0},
       URL = {https://doi.org/10.1016/0022-247X(82)90133-0},
}

@incollection {CM1,
    AUTHOR = {Cioranescu, D. and Murat, F.},
     TITLE = {Un terme \'etrange venu d'ailleurs},
 BOOKTITLE = {Nonlinear partial differential equations and their
              applications. {C}oll\`ege de {F}rance {S}eminar, {V}ol. {II}
              ({P}aris, 1979/1980)},
    SERIES = {Res. Notes in Math.},
    VOLUME = {60},
     PAGES = {98--138, 389--390},
 PUBLISHER = {Pitman, Boston, Mass.-London},
      YEAR = {1982},
      ISBN = {0-273-08541-7},
   MRCLASS = {35J05 (35B25)},
  MRNUMBER = {652509},
MRREVIEWER = {M.\ H.\ Protter},
}

@incollection {CM2,
    AUTHOR = {Cioranescu, D. and Murat, F.},
     TITLE = {Un terme \'etrange venu d'ailleurs. {II}},
 BOOKTITLE = {Nonlinear partial differential equations and their
              applications. {C}oll\`ege de {F}rance {S}eminar, {V}ol. {III}
              ({P}aris, 1980/1981)},
    SERIES = {Res. Notes in Math.},
    VOLUME = {70},
     PAGES = {154--178, 425--426},
 PUBLISHER = {Pitman, Boston, Mass.-London},
      YEAR = {1982},
      ISBN = {0-273-08568-9},
   MRCLASS = {35J05 (35B25)},
  MRNUMBER = {670272},
MRREVIEWER = {M.\ H.\ Protter},
}

@article {O,
    AUTHOR = {Onofrei, D.},
     TITLE = {The unfolding operator near a hyperplane and its applications
              to the {N}eumann sieve model},
   JOURNAL = {Adv. Math. Sci. Appl.},
  FJOURNAL = {Advances in Mathematical Sciences and Applications},
    VOLUME = {16},
      YEAR = {2006},
    NUMBER = {1},
     PAGES = {239--258},
      ISSN = {1343-4373},
   MRCLASS = {35J65 (80M50)},
  MRNUMBER = {2253234},
}

@article {CDGO,
    AUTHOR = {Cioranescu, D. and Damlamian, A. and Griso, G. and Onofrei,
              D.},
     TITLE = {The periodic unfolding method for perforated domains and
              {N}eumann sieve models},
   JOURNAL = {J. Math. Pures Appl. (9)},
  FJOURNAL = {Journal de Math\'ematiques Pures et Appliqu\'ees. Neuvi\`eme
              S\'erie},
    VOLUME = {89},
      YEAR = {2008},
    NUMBER = {3},
     PAGES = {248--277},
      ISSN = {0021-7824},
   MRCLASS = {35B27 (35J25 74Q05)},
  MRNUMBER = {2401689},
MRREVIEWER = {Karsten\ Matthies},
       DOI = {10.1016/j.matpur.2007.12.008},
       URL = {https://doi.org/10.1016/j.matpur.2007.12.008},
}

@article {ABZ,
    AUTHOR = {Ansini, N. and Babadjian, J.-F. and Zeppieri,
              C. I.},
     TITLE = {The {N}eumann sieve problem and dimensional reduction: a
              multiscale approach},
   JOURNAL = {Math. Models Methods Appl. Sci.},
  FJOURNAL = {Mathematical Models and Methods in Applied Sciences},
    VOLUME = {17},
      YEAR = {2007},
    NUMBER = {5},
     PAGES = {681--735},
      ISSN = {0218-2025,1793-6314},
   MRCLASS = {35B27 (35Q72 49J45 82C24)},
  MRNUMBER = {2325836},
MRREVIEWER = {Marco\ Veneroni},
       DOI = {10.1142/S0218202507002078},
       URL = {https://doi.org/10.1142/S0218202507002078},
}

@article {A,
    AUTHOR = {Ansini, N.},
     TITLE = {The nonlinear sieve problem and applications to thin films},
   JOURNAL = {Asymptot. Anal.},
  FJOURNAL = {Asymptotic Analysis},
    VOLUME = {39},
      YEAR = {2004},
    NUMBER = {2},
     PAGES = {113--145},
      ISSN = {0921-7134,1875-8576},
   MRCLASS = {35B27 (74K35 74Q05)},
  MRNUMBER = {2093896},
MRREVIEWER = {Isabelle\ Gruais},
}

@article {C1,
    AUTHOR = {Conca, C.},
     TITLE = {\'Etude d'un fluide traversant une paroi perfor\'ee. {I}.
              {C}omportement limite pr\`es de la paroi},
   JOURNAL = {J. Math. Pures Appl. (9)},
  FJOURNAL = {Journal de Math\'ematiques Pures et Appliqu\'ees. Neuvi\`eme
              S\'erie},
    VOLUME = {66},
      YEAR = {1987},
    NUMBER = {1},
     PAGES = {1--43},
      ISSN = {0021-7824},
   MRCLASS = {76D07 (35Q10)},
  MRNUMBER = {884812},
MRREVIEWER = {Monique\ Dauge},
}

@article {C2,
    AUTHOR = {Conca, C.},
     TITLE = {\'Etude d'un fluide traversant une paroi perfor\'ee. {II}.
              {C}omportement limite loin de la paroi},
   JOURNAL = {J. Math. Pures Appl. (9)},
  FJOURNAL = {Journal de Math\'ematiques Pures et Appliqu\'ees. Neuvi\`eme
              S\'erie},
    VOLUME = {66},
      YEAR = {1987},
    NUMBER = {1},
     PAGES = {45--70},
      ISSN = {0021-7824},
   MRCLASS = {76D07 (35Q10)},
  MRNUMBER = {884813},
MRREVIEWER = {Monique\ Dauge},
}

@article {K,
    AUTHOR = {Khrabustovskyi, A.},
     TITLE = {Operator estimates for the {N}eumann sieve problem},
   JOURNAL = {Ann. Mat. Pura Appl. (4)},
  FJOURNAL = {Annali di Matematica Pura ed Applicata. Series IV},
    VOLUME = {202},
      YEAR = {2023},
    NUMBER = {4},
     PAGES = {1955--1990},
      ISSN = {0373-3114,1618-1891},
   MRCLASS = {35B27 (35B40 35P05 47A55)},
  MRNUMBER = {4597609},
       DOI = {10.1007/s10231-023-01308-z},
       URL = {https://doi.org/10.1007/s10231-023-01308-z},
}

@article {GPM,
    AUTHOR = {G\'omez, D. and P\'erez-Mart\'inez, M.-E.},
     TITLE = {Boundary homogenization with large reaction terms on a
              strainer-type wall},
   JOURNAL = {Z. Angew. Math. Phys.},
  FJOURNAL = {Zeitschrift f\"ur Angewandte Mathematik und Physik. ZAMP.
              Journal of Applied Mathematics and Physics. Journal de
              Math\'ematiques et de Physique Appliqu\'ees},
    VOLUME = {73},
      YEAR = {2022},
    NUMBER = {6},
     PAGES = {Paper No. 234, 28},
      ISSN = {0044-2275,1420-9039},
   MRCLASS = {35Q74 (35B27 35J05 47J20 74Q05)},
  MRNUMBER = {4499087},
       DOI = {10.1007/s00033-022-01869-8},
       URL = {https://doi.org/10.1007/s00033-022-01869-8},
}

@article {DNPM,
    AUTHOR = {G\'omez, D. and Nazarov, S. A. and
              P\'erez-Mart\'inez, M.-E.},
     TITLE = {Asymptotics for spectral problems with rapidly alternating
              boundary conditions on a strainer {W}inkler foundation},
   JOURNAL = {J. Elasticity},
  FJOURNAL = {Journal of Elasticity. The Physical and Mathematical Science
              of Solids},
    VOLUME = {142},
      YEAR = {2020},
    NUMBER = {1},
     PAGES = {89--120},
      ISSN = {0374-3535,1573-2681},
   MRCLASS = {35B27 (35P05 35P15 35Q74 74B05 74Q20)},
  MRNUMBER = {4162420},
       DOI = {10.1007/s10659-020-09791-8},
       URL = {https://doi.org/10.1007/s10659-020-09791-8},
}

@article {AP,
    AUTHOR = {Attouch, H. and Picard, C.},
     TITLE = {Comportement limite de probl\`emes de transmission unilateraux
              \`a{} travers des grilles de forme quelconque},
   JOURNAL = {Rend. Sem. Mat. Univ. Politec. Torino},
  FJOURNAL = {Rendiconti del Seminario Matematico (gi\`a{} ``Conferenze di
              Fisica e di Matematica''). Universit\`a{} e Politecnico di
              Torino},
    VOLUME = {45},
      YEAR = {1987},
    NUMBER = {1},
     PAGES = {71--85},
      ISSN = {0373-1243},
   MRCLASS = {49A50 (35B40 49A29)},
  MRNUMBER = {981155},
}

@article {DV,
    AUTHOR = {Del Vecchio, T.},
     TITLE = {The thick {N}eumann's sieve},
   JOURNAL = {Ann. Mat. Pura Appl. (4)},
  FJOURNAL = {Annali di Matematica Pura ed Applicata. Serie Quarta},
    VOLUME = {147},
      YEAR = {1987},
     PAGES = {363--402},
      ISSN = {0003-4622},
   MRCLASS = {35J25 (35B20 35B25 35B40 35J67)},
  MRNUMBER = {916715},
MRREVIEWER = {Dietrich\ G\"ohde},
       DOI = {10.1007/BF01762424},
       URL = {https://doi.org/10.1007/BF01762424},
}

@article {ABCP,
    AUTHOR = {Amirat, Y. and Bodart, O. and Chechkin, G. A.
              and Piatnitski, A. L.},
     TITLE = {Asymptotics of a spectral-sieve problem},
   JOURNAL = {J. Math. Anal. Appl.},
  FJOURNAL = {Journal of Mathematical Analysis and Applications},
    VOLUME = {435},
      YEAR = {2016},
    NUMBER = {2},
     PAGES = {1652--1671},
      ISSN = {0022-247X,1096-0813},
   MRCLASS = {35B27 (35J05 35J25 35P15 35P20)},
  MRNUMBER = {3429664},
MRREVIEWER = {Taras\ A.\ Mel\cprime nyk},
       DOI = {10.1016/j.jmaa.2015.11.014},
       URL = {https://doi.org/10.1016/j.jmaa.2015.11.014},
}

@article {GHV,
    AUTHOR = {Giunti, A. and H\"ofer, R. and Vel\'azquez, J. J.
              L.},
     TITLE = {Homogenization for the {P}oisson equation in randomly
              perforated domains under minimal assumptions on the size of
              the holes},
   JOURNAL = {Comm. Partial Differential Equations},
  FJOURNAL = {Communications in Partial Differential Equations},
    VOLUME = {43},
      YEAR = {2018},
    NUMBER = {9},
     PAGES = {1377--1412},
      ISSN = {0360-5302,1532-4133},
   MRCLASS = {35B27 (35J05 35J25 35R60 60H15)},
  MRNUMBER = {3915491},
MRREVIEWER = {Dan\ Polisevski},
       DOI = {10.1080/03605302.2018.1531425},
       URL = {https://doi.org/10.1080/03605302.2018.1531425},
}

@article {C,
    AUTHOR = {Cortesani, G.},
     TITLE = {Asymptotic behaviour of a sequence of {N}eumann problems},
   JOURNAL = {Comm. Partial Differential Equations},
  FJOURNAL = {Communications in Partial Differential Equations},
    VOLUME = {22},
      YEAR = {1997},
    NUMBER = {9-10},
     PAGES = {1691--1729},
      ISSN = {0360-5302,1532-4133},
   MRCLASS = {35J25 (35B40)},
  MRNUMBER = {1469587},
MRREVIEWER = {I.\ N.\ Tavkhelidze},
       DOI = {10.1080/03605309708821316},
       URL = {https://doi.org/10.1080/03605309708821316},
}

@article {DMFZ,
    AUTHOR = {Dal Maso, G. and Franzina, G. and Zucco, D.},
     TITLE = {Transmission conditions obtained by homogenisation},
   JOURNAL = {Nonlinear Anal.},
  FJOURNAL = {Nonlinear Analysis. Theory, Methods \& Applications. An
              International Multidisciplinary Journal},
    VOLUME = {177},
      YEAR = {2018},
     PAGES = {361--386},
      ISSN = {0362-546X,1873-5215},
   MRCLASS = {49J45},
  MRNUMBER = {3865203},
MRREVIEWER = {Michele\ Carriero},
       DOI = {10.1016/j.na.2018.04.015},
       URL = {https://doi.org/10.1016/j.na.2018.04.015},
}

@book {LL,
    AUTHOR = {Lieb, E. H. and Loss, M.},
     TITLE = {Analysis},
    SERIES = {Graduate Studies in Mathematics},
    VOLUME = {14},
   EDITION = {Second},
 PUBLISHER = {American Mathematical Society, Providence, RI},
      YEAR = {2001},
     PAGES = {xxii+346},
      ISBN = {0-8218-2783-9},
   MRCLASS = {00A05 (26-01 28-01 31-01 35J10 42-01)},
  MRNUMBER = {1817225},
       DOI = {10.1090/gsm/014},
       URL = {https://doi.org/10.1090/gsm/014},
}

@book {MK,
    AUTHOR = {Marchenko, V. A. and Khruslov, E. Y.},
     TITLE = {Homogenization of partial differential equations},
    SERIES = {Progress in Mathematical Physics},
    VOLUME = {46},
      NOTE = {Translated from the 2005 Russian original by M. Goncharenko
              and D. Shepelsky},
 PUBLISHER = {Birkh\"auser Boston, Inc., Boston, MA},
      YEAR = {2006},
     PAGES = {xiv+398},
      ISBN = {978-0-8176-4351-5; 0-8176-4351-6},
   MRCLASS = {35B27 (35-02 35J25 35J40 35K20 74Q05 76M50)},
  MRNUMBER = {2182441},
MRREVIEWER = {I.\ Aganovi\'c},
}

@incollection {PV,
    AUTHOR = {Papanicolaou, G. C. and Varadhan, S. R. S.},
     TITLE = {Diffusion in regions with many small holes},
 BOOKTITLE = {Stochastic differential systems ({P}roc. {IFIP}-{WG} 7/1
              {W}orking {C}onf., {V}ilnius, 1978)},
    SERIES = {Lect. Notes Control Inf. Sci.},
    VOLUME = {25},
     PAGES = {190--206},
 PUBLISHER = {Springer, Berlin-New York},
      YEAR = {1980},
      ISBN = {3-540-10498-4},
   MRCLASS = {60J60},
  MRNUMBER = {609184},
MRREVIEWER = {Constantin\ Tudor},
}

@article {GH,
    AUTHOR = {Giunti, A. and Höfer, R. M.},
     TITLE = {Homogenisation for the {S}tokes equations in randomly
              perforated domains under almost minimal assumptions on the
              size of the holes},
   JOURNAL = {Ann. Inst. H. Poincar\'e{} C Anal. Non Lin\'eaire},
  FJOURNAL = {Annales de l'Institut Henri Poincar\'e{} C. Analyse Non
              Lin\'eaire},
    VOLUME = {36},
      YEAR = {2019},
    NUMBER = {7},
     PAGES = {1829--1868},
      ISSN = {0294-1449,1873-1430},
   MRCLASS = {35B27 (35Q35 35R60)},
  MRNUMBER = {4020526},
       DOI = {10.1016/j.anihpc.2019.06.002},
       URL = {https://doi.org/10.1016/j.anihpc.2019.06.002},
}

@article{SZZ,
    AUTHOR = {Scardia, L. and Zemas, K. and Zeppieri, C. I.},
    TITLE = {Homogenisation of nonlinear Dirichlet problems in randomly perforated domains under minimal assumptions on the size of perforations},
    JOURNAL = {Probab. Theory Relat. Fields},
    YEAR = {2024},
    PAGES = {},
    DOI = {10.1007/s00440-024-01320-1},
    URL = {https://doi.org/10.1007/s00440-024-01320-1}
}

@article {DMD,
    AUTHOR = {Dal Maso, G. and Defranceschi, A.},
     TITLE = {Limits of nonlinear {D}irichlet problems in varying domains},
   JOURNAL = {Manuscripta Math.},
  FJOURNAL = {Manuscripta Mathematica},
    VOLUME = {61},
      YEAR = {1988},
    NUMBER = {3},
     PAGES = {251--278},
      ISSN = {0025-2611,1432-1785},
   MRCLASS = {49A50 (31B20 31B25 35J65)},
  MRNUMBER = {949817},
MRREVIEWER = {Luciano\ Modica},
       DOI = {10.1007/BF01258438},
       URL = {https://doi.org/10.1007/BF01258438},
}

@article {DMG,
    AUTHOR = {Dal Maso, G. and Garroni, A.},
     TITLE = {New results on the asymptotic behavior of {D}irichlet problems
              in perforated domains},
   JOURNAL = {Math. Models Methods Appl. Sci.},
  FJOURNAL = {Mathematical Models and Methods in Applied Sciences},
    VOLUME = {4},
      YEAR = {1994},
    NUMBER = {3},
     PAGES = {373--407},
      ISSN = {0218-2025,1793-6314},
   MRCLASS = {35B27 (35J25 49J20)},
  MRNUMBER = {1282241},
MRREVIEWER = {Alexander\ Belyaev},
       DOI = {10.1142/S0218202594000224},
       URL = {https://doi.org/10.1142/S0218202594000224},
}

@article {CDG,
    AUTHOR = {Casado Díaz, J. and Garroni, A.},
     TITLE = {Asymptotic behavior of nonlinear elliptic systems on varying
              domains},
   JOURNAL = {SIAM J. Math. Anal.},
  FJOURNAL = {SIAM Journal on Mathematical Analysis},
    VOLUME = {31},
      YEAR = {2000},
    NUMBER = {3},
     PAGES = {581--624},
      ISSN = {0036-1410,1095-7154},
   MRCLASS = {35B27 (35J25 74Q05)},
  MRNUMBER = {1741039},
MRREVIEWER = {Giovanni\ Alberti},
       DOI = {10.1137/S0036141097329627},
       URL = {https://doi.org/10.1137/S0036141097329627},
}
\bibliographystyle{plain}

\end{document}